%% file: nonlocal_multilane.tex
\documentclass[a4paper,11pt]{article}

\usepackage{a4wide}
\usepackage{graphicx}
\usepackage{amssymb,amsmath,amsthm,mathrsfs,xfrac,mathtools}
\usepackage{enumerate, enumitem,float}
\usepackage[title]{appendix}
\usepackage{amsfonts}
\usepackage{hyperref}
\usepackage[utf8]{inputenc}
\usepackage{color}
\usepackage{pgf,tikz}
\usetikzlibrary{external}
\usepackage{pgfplots}
\usetikzlibrary{arrows}
\newlength\fwidth

\newtheorem{proposition}{Proposition}[section]
\newtheorem{theorem}[proposition]{Theorem}
\newtheorem{lemma}[proposition]{Lemma}
\newtheorem{definition}[proposition]{Definition}
\newtheorem{corollary}[proposition]{Corollary}
\newtheorem{remark}[proposition]{Remark}
\newtheorem{lgrthm}[proposition]{Algorithm}

\numberwithin{equation}{section}
\allowdisplaybreaks[4]

\renewenvironment{proof}{\smallskip\noindent\emph{\textbf{Proof.}}%
  \hspace{1pt}}{\hspace{-5pt}{\nobreak\quad\nobreak\hfill\nobreak%
    $\square$\vspace{2pt}\par}\smallskip\goodbreak}

\newcommand{\C}[1]{\mathbf{C^{#1}}}
\newcommand{\Cc}[1]{\mathbf{C_c^{#1}}}
\renewcommand{\L}[1]{\mathbf{L^#1}}

\newcommand{\BV}{\mathbf{BV}}

\newcommand{\modulo}[1]{{\left|#1\right|}}
\newcommand{\norma}[1]{{\left\|#1\right\|}}
\newcommand{\caratt}[1]{{\chi_{\strut#1}}}
\newcommand{\reali}{{\mathbb{R}}}

\newcommand{\interi}{{\mathbb{Z}}}

\renewcommand{\epsilon}{\varepsilon}
\renewcommand{\phi}{\varphi}

\renewcommand{\theta}{\vartheta}

\newcommand{\tv}{\mathinner{\rm TV}}
\newcommand{\spt}{\mathop{\rm spt}}

\renewcommand{\d}[1]{\mathinner{\mathrm{d}{#1}}}

\newcommand{\brho}{\boldsymbol{\rho}}
\newcommand{\bpi}{\boldsymbol{\pi{}}}
\newcommand{\bv}{\boldsymbol{v}}

\newcommand{\dt}{{\Delta t}}
\newcommand{\dx}{{\Delta x}}

\newcommand{\nds}{{\nu}}
\newcommand{\ndt}{{\iota}}

\DeclareMathOperator{\sgn}{sgn}

\newcommand{\be}{\begin{equation}}
\newcommand{\ee}{\end{equation}}

\definecolor{ffqqqq}{rgb}{1.,0.,0.}
\definecolor{uuuuuu}{rgb}{0.26666666666666666,0.26666666666666666,0.26666666666666666}

\makeatletter
\let\@fnsymbol\@arabic
\makeatother

\makeatletter
\newcommand\appendix@section[1]{%
\refstepcounter{section}%
\orig@section*{Appendix \@Alph\c@section: #1}%
\addcontentsline{toc}{section}{Appendix \@Alph\c@section: #1}%
}
\let\orig@section\section
\g@addto@macro\appendix{\let\section\appendix@section}
\makeatother

\title{Nonlocal approaches for multilane traffic models}

\author{Jan Friedrich\footnotemark[1] \and Simone G\"ottlich\footnotemark[1] \and Elena Rossi\footnotemark[2]}

\date{\today}

\begin{document}
\maketitle
\footnotetext[1]{University of Mannheim, Department of Mathematics, B6
  28-29, 68131 Mannheim, Germany. Email: \texttt{\{jan.friedrich,
    goettlich\}@uni-mannheim.de}} \footnotetext[2]{Università degli
  Studi di Modena e Reggio Emilia, Dipartimento di Scienze e Metodi
  dell'Ingegneria, Via Amendola 2 - Pad. Morselli, 42122 Reggio
  Emilia. E-mail: \texttt{elena.rossi13@unimore.it} }

\begin{abstract}
  We present a multilane traffic model based on balance laws, where
  the nonlocal source term is used to describe the lane changing
  rate. The modelling framework includes the consideration of local
  and nonlocal flux functions. Based on a Godunov type numerical
  scheme, we provide BV estimates and a discrete entropy
  inequality. Together with the $\L1$-contractivity property, we prove
  existence and uniqueness of weak solutions. Numerical examples show
  the nonlocal impact compared to local flux functions and local sources.

  \medskip

  \noindent\textit{AMS Subject Classification:} 35L65, 90B20, 65M12

  \medskip

  \noindent\textit{Keywords:} nonlocal balance laws, multilane traffic
  flow, Godunov scheme
\end{abstract}

\section{Introduction}
\label{sec:intro}
The progress in autonomous driving brings new challenges for the
modelling of traffic flow. Classical approaches such as the
well-established Lighthill-Whitham-Richards (LWR)
model \cite{LighthillWhitham, Richards} have been recently extended to
include more information on the surrounding traffic, see for
example~\cite{BlandinGoatin,friedrich2018godunov,GoatinScialanga,KeimerPflug2017,RidderShen}.
Therefore, we distinguish between \emph{local} traffic flow models
governed by conservation laws, where the fundamental diagram gives the
relation between flux and density, and \emph{nonlocal} models with
flux functions depending on an integral evaluation of the density or
velocity through a convolution product. In case of autonomous vehicles
the nonlocal models allow for an interpretation as the connection
radius.

Nonlocal traffic flow models have been introduced
in~\cite{BlandinGoatin} and since then have been studied regarding
existence and well-posedness,
e.g.~\cite{ChiarelloGoatin,friedrich2018godunov,KeimerPflug2017},
numerical schemes~\cite{BlandinGoatin,
  chalons2018high,friedrich2019maximum,friedrich2018godunov,
  GoatinScialanga}, convergence to local conservation laws
e.g.~\cite{keimer2019local} (even if this question is still an open
research problem), microscopic modelling
approaches~\cite{friedrich2020micromacro,shen2019stationary,GoatinRossi2017,
  RidderShen}, second order models~\cite{friedrich2020micromacro},
multi-class models~\cite{chiarello2019multiclass}, time delay
models~\cite{keimer2019nonlocal} and network
formulations~\cite{chiarelloFriedrichGoatinGK,shen2019stationary}.

The aim of this paper is to study a multilane model with local and
nonlocal flux combined with a source term that also incorporates a
nonlocality. Here, the nonlocal source term describes the lane
changing rate depending on a (nonlinear) evaluation of the
velocity. In this context, we refer to~\cite{colombo2007non}, where a
nonlocal source term is used to describe the lane
change. 
However, the modelling of our source term is inspired
by~\cite{HoldenRisebro}. 
We would also like to mention that a similar multilane model with
nonlocal flux and source has been recently introduced
in~\cite{BayenKeimer-preprint}. Therein, well-posedness and uniqueness
are proven based on Banach's fixed point theorem using the method of
characteristics.
In contrast to the
contributions~\cite{BayenKeimer-preprint,colombo2007non,HoldenRisebro},
we do not investigate the model on the continuous description and
present a Godunov type numerical scheme instead that can be used to
show existence and uniqueness of approximate solutions.

From a modelling point of view, the key differences
to~\cite{BayenKeimer-preprint} are that the nonlocal terms in the flux
and in the source do not have necessarily the same kernels. More
precisely, in~\cite{BayenKeimer-preprint}
only forward looking and decreasing kernels can be considered within
the convolution product.
It seems that our approach is more flexible, since also back- and
forward looking kernels can be considered. 

The paper is organised as follows: In
Section~\ref{sec:nonlocal_source} we present the model with local flux
and nonlocal source, while in Section~\ref{sec:numScheme} the Godunov
type discretization is addressed to show the existence
of a solution to the model as the limit of a 
sequence of approximate solutions. The uniqueness result is then
discussed in Section~\ref{sec:unique}. The extension to the model with
nonlocal flux and source is given in Section~\ref{sec:ext}, with
particular focus on differences to the model with local flux only. In
Section~\ref{sec:numEx} a collection of numerical experiments is
carried out.

\section{A multilane model with nonlocal source term}
\label{sec:nonlocal_source}

In~\cite{HoldenRisebro} the authors exploit the traditional LWR model
to study vehicular traffic on a road with multiple lanes. The key
feature of the model in~\cite{HoldenRisebro} is that drivers tend to
change to a neighbouring lane proportionally to the difference in the
(local) velocity between the lanes.

However, as it is already well known, the use of nonlocal terms may
lead to other dynamical behaviour, see e.g.~\cite{BlandinGoatin}.
In this paper, we aim to extend the multilane
model~\cite{HoldenRisebro} to account for a \emph{nonlocal} evaluation
of the velocity influencing the lane changing rate. The idea is that
at position $x$ the flow between neighbouring lanes is governed by the
difference in the velocity evaluated on the average density
\emph{around} position $x$, e.g.~on the interval $[x-\nu,x+\nu]$,
$\nu > 0$.  This modelling hypothesis is motivated by a feature
typical of drivers behaviour: when driving on a multilane road, at the
moment of changing lane, the driver checks what is happening behind
and in front of him/her, both on his/her lane and on the neighbouring
one(s).

Recall the model introduced in~\cite{HoldenRisebro} for a road with
$M$ lanes:
\begin{displaymath}
  \left\{
    \begin{array}{lr}
      \partial_t \rho_j + \partial_x \left(\rho_j \, v_j(\rho_j)\right) =
      S_{j-1}(\rho_{j-1}, \rho_1)-S_j(\rho_j, \rho_{j+1})
      & j = 1, \ldots, M,
      \\
      \rho_{j}(0,x) = \rho_{o,j}(x)
      & j = 1, \ldots, M,
    \end{array}
  \right.
\end{displaymath}
with
\begin{displaymath}
  S_j(\rho_j, \rho_{j+1}) =
  K \, (v_{j+1}(\rho_{j+1}) - v_j(\rho_j)) \,
  \begin{cases}
    \rho_j & \mbox{ if } v_{j+1}(\rho_{j+1}) \geq v_j(\rho_j),\\
    \rho_{j+1} & \mbox{ if } v_{j+1}(\rho_{j+1}) < v_j(\rho_j),
  \end{cases}
\end{displaymath}
and the boundary conditions
\begin{displaymath}
  S_0 (\rho_0, \rho_1) = S_M (\rho_M, \rho_{M+1}) = 0,
\end{displaymath}
$K$ being a dimensional constant ($1/m$). The modelling idea behind the term $S_j(\rho_j, \rho_{j+1})$ lies in the assumption that drivers prefer to be in the faster lane, and that the lane changing rate is proportional to the difference in the (local) velocity.\\
In contrast, our modelling approach accounts for a \emph{nonlocal}
evaluation of the velocity influencing the lane changing
rate. Therefore, we introduce a kernel function
$w_\nds \in \C0([-\nds, \nds], \reali_+)$, with $\nds >0$ and
$\int_\reali w_\nds(x) \d{x}=1$,
and define the flow from lane $j$ to lane $j+1$ as follows: for
$j=1, \ldots, M-1$
\begin{align}
  \nonumber
  & S_j(\rho_j, \rho_{j+1}, R_j, R_{j+1})
  \\
  \label{eq:source}
  = \ &
        K \, (v_{j+1}(R_{j+1}) - v_j(R_j)) \,
        \begin{cases}
          \rho_j \, (1 - \rho_{j+1}) & \mbox{ if } v_{j+1}(R_{j+1}) \geq v_j(R_j),\\
          \rho_{j+1} \, (1- \rho_j) & \mbox{ if } v_{j+1}(R_{j+1}) <
          v_j(R_j),
        \end{cases}
  \\
  \nonumber
  = \ &
        K \, \left[ (v_{j+1}(R_{j+1}) - v_j(R_j))^+ \, \rho_j \, (1 - \rho_{j+1})
        - (v_{j+1}(R_{j+1}) - v_j(R_j)) ^- \, \rho_{j+1} \, (1 - \rho_j)
        \right],
  \\
  \label{eq:Rj}
  & \mbox{with } \quad
    R_j = R_j (t,x) =  \left(\rho_j(t)  * w_\nds\right) (x),
\end{align}
where $(s)^+ = \max\{s,0\}$, $(s)^- = - \min\{s,0\}$. Conversely, the
flow from lane $j+1$ to lane $j$ equals
$- S_j (\rho_j, \rho_{j+1}, R_j, R_{j+1})$. Here, $K$ is still a
dimensional constant ($1/m$). For simplicity, in the following time
and space are scaled so that $K=1$. The model we study is thus
\begin{equation}
  \label{eq:model}
  \left\{
    \begin{array}{lr}
      \partial_t \rho_j + \partial_x \left(\rho_j \, v_j(\rho_j)\right) = S_{j-1}(\rho_{j-1}, \rho_j, R_{j-1}, R_j) -S_j(\rho_j, \rho_{j+1}, R_j, R_{j+1})
      & j= 1, \ldots, M,
      \\
      \rho_{j}(0,x) = \rho_{o,j}(x)
      & j= 1, \ldots, M,
    \end{array}
  \right.
\end{equation}
with boundary conditions
\begin{equation}\label{eq:boundaryconditions}
  S_0(\rho_0, \rho_1, R_0, R_1) = S_M (\rho_M, \rho_{M+1}, R_M, R_{M+1}) = 0.
\end{equation}

The meaning of the source term defined by~\eqref{eq:source} is the
following: similarly to the model studied in~\cite{HoldenRisebro}, the
lane changing rate is proportional to the difference in the velocity
between two adjacent lanes, but the velocities are now evaluated
\emph{nonlocally}, i.e.~in a neighbourhood of the \emph{current}
position.  Moreover, this rate is now proportional also to the density
in the receiving lane, meaning that, if that lane is crowded, only a
few vehicles can actually change lane. We remark that including this
latter factor allows to prove the invariance of the set $[0,1]^M$ for
model~\eqref{eq:model}, see Section~\ref{sec:invariance}. We emphasize
that this is not necessary for the \emph{local}
model~\cite{HoldenRisebro}.

In a next step, we define a \emph{weak solution} to~\eqref{eq:model}
and present the key result of this paper for the existence and
uniqueness of the solution.
As in~\cite{HoldenRisebro}, we assume that the velocity functions
$v_i$ are strictly decreasing, positive and scaled such that
$v_i(1) = 0$, $i=1,\ldots, M$. For simplicity, space and time are
scaled so that $K=1$. We assume that each map $f_j(u) = u \, v_j (u)$
admits a unique global maximum in the interval $[0,1]$ attained at
$u= \theta_j$.

\begin{definition}
  \label{def:sol}
  Let $\rho_{o,j} \in (\L1 \cap \BV)(\reali;[0,1])$, for
  $j=1, \ldots, M$. We say that
  $\rho_j \in \C0([0,T]; \L1(\reali; [0,1]))$, with
  $\rho_j(t, \cdot) \in \BV(\reali; [0,1])$ for $t \in [0,T]$, is a
  \emph{weak solution} to~\eqref{eq:model} with initial datum
  $\rho_{o,j}$ if for any $\phi \in \Cc1([0,T[ \times \reali; \reali)$
  and for all $j=1, \ldots, M$
  \begin{multline*}
    \int_0^T \!\!\!\int_\reali\!\! \left( \rho_j \, \partial_t \phi +
      \rho_j \, v_j(\rho_j) \, \partial_x \phi +
      \left(S_{j-1}(\rho_{j-1},\rho_j, R_{j-1}, R_j) - S_j(\rho_j,
        \rho_{j+1}, R_j, R_{j+1}) \right)\phi \right)\! \d{x} \d{t}
    \\
    + \int_\reali \rho_{o,j} \, \phi(0,x) \d{x} = 0,
  \end{multline*}
  with $S_j $ as in~\eqref{eq:source} and
  $R_j = R_j (t,x) = \left(\rho_j (t) * w_\nds\right) (x)$.  The
  solution $\rho_j$ is an \emph{entropy solution} if for any
  $\phi \in \Cc1([0,T[ \times \reali; \reali_+)$, for all convex
  entropy-entropy flux pairs $(\eta,q)$ and for all $j=1,\ldots, M$

  \begin{multline*}
    \int_0^T \int_\reali \left( \eta(\rho_j) \, \partial_t \phi +
      q(\rho_j) \, \partial_x \phi \right) \d{x} \d{t} + \int_\reali
    \eta(\rho_{o,j}) \, \phi(0,x) \d{x}
    \\
    \geq \int_0^T \int_\reali \eta'(\rho_j) \left( S_j(\rho_j,
      \rho_{j+1}, R_j, R_{j+1}) -S_{j-1}(\rho_{j-1}, \rho_j, R_{j-1},
      R_j) \right) \phi \d{x} \d{t}.
  \end{multline*}
\end{definition}

In the following it will be convenient to use the notation
$\brho = \left(\rho_1, \ldots, \rho_M\right)$ to denote the vector of
component $\rho_j$, $j=1, \ldots, M$. The initial datum to
problem~\eqref{eq:model} is then $\brho_o$.

\begin{theorem}
  \label{thm:main}
  Let $\brho_o \in (\L1 \cap \BV)(\reali; [0,1]^M)$. Then, for all
  $T>0$, problem~\eqref{eq:model} has a unique solution
  $\brho \in \C0([0,T]; \L1(\reali; [0,1]^M))$ in the sense of
  Definition~\ref{def:sol}. Moreover, the following estimates hold:
  for any $t \in [0,T]$
  \begin{align*}
    \norma{\brho(t)}_{\L1(\reali)} = \
    & \sum_{j=1}^M \norma{\rho_j(t)}_{\L1(\reali)}=
      \norma{\brho_o}_{\L1 (\reali)},
    \\
    \mbox{for }  j = 1, \ldots, M:
    & \quad  0\leq \rho_j(t,x)\leq 1,
    \\
    \sum_{j=1}^M\tv\left(\rho_j(t)\right) \leq \
    & \sum_{j=1}^M \tv(\rho_{j,o}).
  \end{align*}
\end{theorem}

Existence of solutions to problem~\eqref{eq:model} is ensured by the
convergence of a sequence of approximate solutions, constructed
through a Godunov scheme, see Section~\ref{sec:conv}. Uniqueness
follows from the $\L1$-contractivty property for the whole solution
to~\eqref{eq:model}, see Section~\ref{sec:unique}. The $\L1$ and
$\L\infty$ bounds follow from the convergence of the scheme, while the
total variation bound is a consequence of the $\L1$-contractivity, see
Corollary~\ref{cor:niceCons}.

\section{Numerical discretization: a Godunov type scheme}
\label{sec:numScheme}

To prove several properties of the model~\eqref{eq:model} and in
particular Theorem~\ref{thm:main}, we introduce a uniform space mesh
of width $\Delta x$ and a time step $\Delta t$, subject to a CFL
condition to be detailed later on. For any $k \in \interi$ denote the
centre of the $k$-th cell by $x_k$ and its interfaces by
$x_{k\pm 1/2}$:
\begin{align*}
  x_k = \
  & \left( k+\frac12\right) \Delta x,
  &
    x_{k-1/2} = \
  & k \, \Delta x.
\end{align*}
Set $N_T=\lfloor T/\Delta t\rfloor$ and define the time mesh as
$t^n = n \, \Delta t$, $n= 0, \ldots, N_T$. Set
$\lambda = \Delta t / \Delta x$. The initial data are approximated as
follows: for $j=1,\ldots, M$ and $k \in \interi$,
\begin{displaymath}
  \rho_{j,k}^0 = \frac{1}{\Delta x} \int_{x_{k-1/2}}^{x_{k+1/2}} \rho_{o,j} (x) \d x.
\end{displaymath}
We construct an approximate solution $\brho_\Delta$
to~\eqref{eq:model} as follows: for $j=1, \ldots, M$ set
\begin{equation}
  \label{eq:rhoDelta}
  \rho_{j,\Delta} (t,x) = \rho_{j,k}^n \quad
  \mbox{ for } \quad
  \left\{
    \begin{array}{l}
      t \in [t^n, t^{n+1}[,\\
      x \in [x_{k-1/2}, x_{k+1/2}],
    \end{array}
  \right.
  \quad \mbox{ with } \quad
  \begin{array}{l}
    n = 0, \ldots, N_T -1, \\
    k \in \interi.
  \end{array}
\end{equation}
The approximate solution $\brho_\Delta$ is obtained via a Godunov type
scheme together with operator splitting, to account for the source
terms, see Algorithm~\ref{alg:1}.
\begin{lgrthm}
  \label{alg:1}
  \begin{align}
    \label{eq:numFlux}
    & F_j(u,w) =  \min\left\{
      f_j\left( \min\{u, \, \theta_j\}\right),
      f_j\left( \max\{w, \, \theta_j\}\right)
      \right\}
      \quad j=1, \ldots, M
    \\
    \nonumber
    & \mbox{ for }
      n= 0, \ldots, N_T -1 :
    \\
    \nonumber
    & \qquad \mbox{ for }
      j=1, \ldots, M \mbox{ and } k \in \interi:
    \\
    \label{eq:convStep}
    & \qquad \qquad
      \rho_{j,k}^{n+1/2} =
      \rho_{j,k}^n - \lambda \left[ F_j(\rho_{j,k}^n, \rho_{j,k+1}^n)
      -
      F_j(\rho_{j,k-1}^n, \rho_{j,k}^n)\right]
    \\
    \label{eq:relaxStep}
    & \qquad \qquad
      \begin{aligned}
        \rho_{j,k}^{n+1} = \ & \rho_{j,k}^{n+1/2} + \Delta t \,
        S_{j-1} (\rho_{j-1,k}^{n+1/2}, \rho_{j,k}^{n+1/2},
        R_{j-1,k}^{n+1/2}, R_{j,k}^{n+1/2})
        \\
        &- \Delta t \, S_j (\rho_{j,k}^{n+1/2}, \rho_{j+1,k}^{n+1/2},
        R_{j,k}^{n+1/2}, R_{j+1,k}^{n+1/2})
      \end{aligned}
  \end{align}
\end{lgrthm}

Above, $R_{j,k}^{n+1/2}$, for $j=1, \ldots, M$, $k \in \interi$ and
$n=0, \ldots, N_T-1$, denotes the discrete convolution operator, which
is defined in the Lemma below.
\begin{lemma}
  \label{lem:convOp}
  Let $w_\nds \in \C0([-\nds, \nds]; \reali_+)$ be such that
  $\int_\reali w_\nds = 1$.
  Define the set
  \begin{equation}
    \label{eq:H}
    \mathcal{H} = \left\{
      h \in \interi :
      \left\lfloor \frac{\inf \spt w_\nds}{\dx}\right\rfloor
      \leq h \leq
      \left\lfloor \frac{\sup \spt w_\nds}{\dx}\right\rfloor -1
    \right\}
  \end{equation}
  and for all $h \in \mathcal{H}$ set
  \begin{displaymath}
    \gamma_h := \int_{x_{h-1/2}}^{x_{h+1/2}}
    w_\nds(y-x) \d{y}.
  \end{displaymath}
  Given $r(x) = r_k \, \caratt{[x_{k-1/2}, x_{k+1/2}]}(x)$, with
  $r_k \in [0,1]$ and $k \in \interi$, the discrete convolution
  operator defined for all $k\in\interi$ as
  \begin{equation}
    \label{eq:disConvOp}
    R_{k} = \sum_{h \in \mathcal{H}} \gamma_h \, r_{k+h+1}
  \end{equation}
  satisfies the following properties:
  \begin{align}
    \label{eq:Rbound}
    R_k \in [0,1]
    & \mbox{ for all } k \in \interi,
    \\
    \label{eq:Rtv}
    \sum_{k \in\interi} \modulo{R_{k+1} - R_k}
    \leq \ &
             \sum_{k \in\interi} \modulo{r_{k+1} - r_k}.
  \end{align}
  Given
  $\tilde r (x) = \tilde r_k \,\caratt{[x_{k-1/2}, x_{k+1/2}]}(x)$,
  with $\tilde r_k \in [0,1]$ and $k \in \interi$, and $\tilde R_k$
  defined accordingly to~\eqref{eq:disConvOp}, then
  \begin{equation}
    \label{eq:R-2}
    \sum_{k \in \interi} \modulo{R_k - \tilde R_k}
    \leq \sum_{k \in \interi} \modulo{r_k - \tilde r_k}.
  \end{equation}
\end{lemma}

\begin{proof}
  It is immediate to see that $\gamma_h \in [0,1]$ for all
  $h \in \mathcal{H}$, due to the properties of $w_\nds$. Hence, for
  all $k \in \interi$, we clearly have $R_k\geq 0$ and
  \begin{displaymath}
    R_k = \sum_{h \in \mathcal{H}} \gamma_h \, r_{k+h+1}
    \leq
    \sum_{h \in \mathcal{H}} \gamma_h
    =
    \int_{\spt w_\nds} w_\nds (y-x) \d{y}
    =
    1,
  \end{displaymath}
  since each $r_k \in [0,1]$.

  Pass now to~\eqref{eq:Rtv}: rearranging the indexes yields
  \begin{align*}
    \sum_{k \in\interi} \modulo{R_{k+1} - R_k}
    \leq \
    & \sum_{k \in\interi} \sum_{h \in \mathcal{H}}
      \gamma_h \, \modulo{r_{k+h+2} - r_{k+h+1}}
    \\
    = \ &
          \left( \sum_{h \in \mathcal{H}} \gamma_h\right)
          \sum_{k \in\interi}  \modulo{r_{k+1} - r_{k}}
          =
          \sum_{k \in\interi}  \modulo{r_{k+1} - r_{k}}.
  \end{align*}
  The proof of~\eqref{eq:R-2} is entirely analogous.
\end{proof}

\begin{remark}
  According to the support of the kernel function $w_\nds$, the
  discrete convolution operator defined by~\eqref{eq:disConvOp} has
  one of the following two forms:
  \begin{itemize}
  \item {\bf Forward looking kernel:} if
    $\spt w_\nds \subseteq [0, \nds]$, then
    $\mathcal{H} = \left[0, \lfloor\frac{\nds}{\Delta
        x}\rfloor-1\right]$, so that
    \begin{equation}
      \label{eq:forward}
      R_{j,k}^{n+1/2} =
      \sum_{h=0}^{\lfloor\nds/\Delta x\rfloor-1}\gamma_h \,
      \rho^{n+1/2}_{j,k+h+1}.
    \end{equation}

  \item {\bf Back- and forward looking kernel:} if
    $\spt w_\nds \subseteq [-\nds, \nds]$,
    $\mathcal{H}= \left[-\lfloor\frac{\nds}{\Delta x}\rfloor,
      \lfloor\frac{\nds}{\Delta x}\rfloor-1\right]$, so that
    \begin{equation}
      \label{eq:backfor}
      R_{j,k}^{n+1/2} =
      \sum_{h=-\lfloor\nds/\Delta x\rfloor}^{\lfloor\nds/\Delta x\rfloor-1}\gamma_h \, \rho^{n+1/2}_{j,k+h+1}.
    \end{equation}
  \end{itemize}
\end{remark}

\subsection{Invariance of the set \texorpdfstring{$[0,1]^M$}{[0,1]}}
\label{sec:invariance}

Under a suitable CFL condition, if each component of the initial datum
takes values in the interval $[0, 1]$, then also the components of the
approximate solution constructed via Algorithm~\ref{alg:1} attain
values in the same interval $[0, 1]$: the set $[0,1]^M$ is thus
invariant for problem~\eqref{eq:model}.

\begin{lemma}
  \label{lem:invariance}
  Let $\brho_o \in \L\infty (\reali; [0,1]^M)$. Assume that
  \begin{equation}
    \label{eq:CFL}
    \lambda \,\mathcal{V} \leq  \frac12,
  \end{equation}
  where
  \begin{align}
    \label{eq:VC1}
    \mathcal{V} = \
    & \norma{\bv}_{\C1([0,1]; \reali^M)}
      = V_{\mathrm{max}} + V'_{\mathrm{max}},
    \\
    \label{eq:vmax}
    V_{\mathrm{max}}
    = \
    & \norma{\bv}_{\C0([0,1]; \reali^M)}
      = \max_{j=1, \ldots, M} \norma{v_j}_{\L\infty([0,1]; \reali)},
    \\
    \label{eq:vdiffmax}
    V'_{\mathrm{max}}
    = \
    & \norma{\bv'}_{\C0([0,1]; \reali^M)}
      = \max_{j=1, \ldots, M} \norma{v'_j}_{\L\infty([0,1]; \reali)}.
  \end{align}
  Then, for all $t>0$ and $x \in \reali$, the piece-wise constant
  approximate solution $\brho_\Delta$ constructed through
  Algorithm~\ref{alg:1} attains value in the set $[0,1]^M$, i.e.
  \begin{displaymath}
    0 \leq \rho_{j, \Delta} (t,x) \leq 1
    \quad \mbox{ for all } j = 1, \ldots, M.
  \end{displaymath}
\end{lemma}
\begin{proof}
  The proof is done by induction and follows the idea of the proof
  of~\cite[Lemma~2.2]{GoatinRossi}. Assume that
  $\brho_\Delta (t,x) \in [0,1]^M $ for all $x \in \reali$ and
  $t< t^{n+1}$. In particular, $0 \leq \rho_{j,k}^n \leq 1$ for
  $j=1, \ldots, M$ and all $k \in \interi$.  Consider the convective
  step~\eqref{eq:convStep} of Algorithm~\ref{alg:1}: the Godunov type
  scheme preserves the invariance of the set $[0,1]^M$,
  i.e.~$0 \leq \rho_{j,k}^{n+1/2} \leq 1$ for $j=1, \ldots, M$ and all
  $k \in \interi$,
  see~\cite[Proposition~3.1~(b)]{crandall1980monotone}.
  \\
  Now focus on the relaxation step~\eqref{eq:relaxStep}, taking care
  of the source term. Without loss of generality, we may fix $j=2$, so
  to have contributions from both the preceding and the subsequent
  lanes. Omitting the index $n+1/2$ to improve readability,
  by~\eqref{eq:convStep} and~\eqref{eq:source}, we get
  \begin{align}
        \rho_{2,k}^{n+1}
    \label{eq:1}
    = \ & \rho_{2,k} + \Delta t \, S_1 (\rho_{1,k},
          \rho_{2,k}, R_{1,k}, R_{2,k}) - \Delta t \, S_2 (\rho_{2,k},
          \rho_{3,k}, R_{2,k}, R_{3,k})
    \\
    \nonumber
    = \ & \rho_{2,k} + \Delta t \left[ \left(v_2(R_{2,k}) - v_1
          (R_{1,k})\right) ^+ \rho_{1,k} (1-\rho_{2,k}) -
          \left(v_2(R_{2,k}) - v_1 (R_{1,k})\right) ^- \rho_{2,k}
          (1-\rho_{1,k}) \right]
    \\
    \nonumber
      & \,\,\,\,\quad - \Delta t \left[ \left(v_3(R_{3,k}) - v_2
        (R_{2,k})\right) ^+ \rho_{2,k} (1-\rho_{3,k}) -
        \left(v_3(R_{3,k}) - v_2 (R_{2,k})\right) ^- \rho_{3,k}
        (1-\rho_{2,k}) \right].
  \end{align}
  There are four possibilities, according to the signs of the
  differences in the velocity:
  \begin{center}
    \begin{tabular}{c|c|c}
      & $v_2 (R_{2,k}) \geq v_1 (R_{1,k})$ & $v_2 (R_{2,k}) < v_1 (R_{1,k}) $\\
      \hline
      $v_3 (R_{3,k}) \geq v_2 (R_{2,k})$ &
                                           \bf{Case A} & \bf{Case B}
      \\
      \hline
      $v_3 (R_{3,k}) < v_2 (R_{2,k})$ & \bf{Case C} & \bf{Case D}
    \end{tabular}
  \end{center}
  We analyse them in detail below. Observe first that the following
  facts hold true:
  \begin{enumerate}[label=\textbf{(\roman*)}]
  \item\label{it:utile1} Whenever
    $v_{\ell}(R_{\ell,k}) \geq v_{j}(R_{j,k})$,
    $j, \ell \in\{1, \ldots, M\}$, $j \neq \ell$, then
    \begin{displaymath}
      1 - \Delta t \left( v_{\ell}(R_{\ell,k}) - v_{j}(R_{j,k}) \right) \rho_{j,k} \geq 0.
    \end{displaymath}
    Indeed, thanks to the CFL condition~\eqref{eq:CFL} and to the fact
    that $\Delta x <1$, we have
    \begin{displaymath}
      1 - \Delta t \,  v_{\ell}(R_{\ell,k}) \, \rho_{j,k}
      + \Delta t \,  v_{j}(R_{j,k}) \, \rho_{j,k}
      \geq
      1 - \Delta t \,  v_{\ell}(R_{\ell,k}) \, \rho_{j,k}
      \geq
      1 - \Delta t \, V_{\textrm{max}}
      \geq 0.
    \end{displaymath}

  \item\label{it:utile2} Whenever
    $v_{j+1}(R_{j+1,k}) < v_{j}(R_{j,k})$ and
    $v_{j}(R_{j,k}) \geq v_{j-1}(R_{j-1,k})$,
    $j \in \{2, \ldots, M-1\}$, then
    \begin{displaymath}
      1 - \Delta t \left[\left( v_{j}(R_{j,k}) - v_{j-1}(R_{j-1,k}) \right) \rho_{j-1,k}
        - \left( v_{j+1}(R_{j+1,k}) - v_{j}(R_{j,k}) \right) \rho_{j+1,k}\right]
      \geq 0.
    \end{displaymath}
    Indeed, thanks to the CFL condition~\eqref{eq:CFL}, to the fact
    that $\Delta x <1$ and that $\rho_{j,k} \leq 1$, we have
    \begin{align*}
      & 1  - \Delta t \left[\left( v_{j}(R_{j,k}) - v_{j-1}(R_{j-1,k}) \right) \rho_{j-1,k}
        - \left( v_{j+1}(R_{j+1,k}) - v_{j}(R_{j,k}) \right) \rho_{j+1,k}\right]
      \\
      \geq \
      & 1 - \Delta t \left[ v_{j}(R_{j,k}) \, \rho_{j-1,k}
        + v_{j}(R_{j,k}) \, \rho_{j+1,k}\right]
      \\
      \geq \
      &  1 -  2 \, \Delta t \,  v_{j}(R_{j,k})
      \\
      \geq  \
      & 1 - 2 \, \Delta t \, V_{\textrm{max}}
        \geq 0.
    \end{align*}
  \end{enumerate}
  The properties above will be used in order to exploit the following
  elementary inequality:
  \begin{displaymath}
    0\leq \rho \leq 1
    \quad \mbox{ and } \quad
    A \geq 0
    \quad \Longrightarrow \quad
    \rho \, A \leq A.
  \end{displaymath}

  \paragraph{Case A.} Here~\eqref{eq:1} reads
  \begin{align*}
    \rho_{2,k}^{n+1} = \
    & \rho_{2,k} + \Delta t \,
      \left(v_2(R_{2,k}) - v_1 (R_{1,k})\right) \rho_{1,k} (1-\rho_{2,k})
      - \Delta t \,
      \left(v_3(R_{3,k}) - v_2 (R_{2,k})\right) \rho_{2,k} (1-\rho_{3,k}).
  \end{align*}
  Thus, aiming for the bound from above, thanks to~\ref{it:utile1} and
  since $0 \leq \rho_{2,k} \leq 1$, we get
  \begin{align*}
    \rho_{2,k}^{n+1} \leq \
    & \rho_{2,k} + \Delta t \,
      \left(v_2(R_{2,k}) - v_1 (R_{1,k})\right) \rho_{1,k} (1-\rho_{2,k})
    \\
    = \
    & \rho_{2,k} \left(
      1 - \Delta t  \left(v_2(R_{2,k}) - v_1 (R_{1,k})\right) \rho_{1,k} \right)
      + \Delta t  \left(v_2(R_{2,k}) - v_1 (R_{1,k})\right) \rho_{1,k}
    \\
    \leq \
    & 1 - \Delta t  \left(v_2(R_{2,k}) - v_1 (R_{1,k})\right) \rho_{1,k}
      + \Delta t  \left(v_2(R_{2,k}) - v_1 (R_{1,k})\right) \rho_{1,k}
    \\
    = \
    & 1.
  \end{align*}
  Pass now to the positivity: thanks to~\ref{it:utile1} we obtain
  \begin{align*}
    \rho_{2,k}^{n+1} \geq \
    & \rho_{2,k} - \Delta t \,
      \left(v_3(R_{3,k})- v_2 (R_{2,k})\right) \rho_{2,k} (1-\rho_{3,k})
    \\
    \geq \
    &
      \rho_{2,k} - \Delta t \,
      \left(v_3(R_{3,k})- v_2 (R_{2,k})\right) \rho_{2,k}
    \\
    \geq \ & 0.
  \end{align*}

  \paragraph{Case B.} In this case~\eqref{eq:1} reads
  \begin{align*}
    \rho_{2,k}^{n+1} = \
    & \rho_{2,k} + \Delta t \,
      \left(v_2(R_{2,k}) - v_1 (R_{1,k})\right) \rho_{2,k} (1-\rho_{1,k})
      - \Delta t \,
      \left(v_3(R_{3,k}) - v_2 (R_{2,k})\right) \rho_{2,k} (1-\rho_{3,k}).
  \end{align*}
  Since $v_2(R_{2,k}) - v_1 (R_{1,k})<0$ and
  $v_3(R_{3,k}) -v_2 (R_{2,k})\geq 0$, it is immediate to prove that
  $\rho_{2,k}^{n+1} \leq \rho_{2,k}^{n+1/2}$ and thus
  $\rho_{2,k}^{n+1}$ is bounded by $1$ from above.  Moreover, by the
  CFL condition~\eqref{eq:CFL},
  \begin{align*}
    \rho_{2,k}^{n+1} \geq \
    & \rho_{2,k} + \Delta t \left[
      \left(v_2(R_{2,k}) - v_1 (R_{1,k})\right) \rho_{2,k}
      - \left(v_3(R_{3,k}) - v_2 (R_{2,k})\right) \rho_{2,k}
      \right]
    \\
    \geq \
    & \rho_{2,k} \left(1
      - \Delta t \, \left(v_1 (R_{1,k}) + v_3(R_{3,k})\right)
      \right)
    \\
    \geq \
    & \rho_{2,k} (1 - 2\, \Delta t \, V_{\textrm{max}})
    \\
    \geq \
    & 0.
  \end{align*}

  \paragraph{Case C.} Here~\eqref{eq:1} reads
  \begin{align*}
    \rho_{2,k}^{n+1} = \
    & \rho_{2,k} + \Delta t
      \left[\left(v_2(R_{2,k}) - v_1 (R_{1,k})\right) \rho_{1,k} (1-\rho_{2,k})
      -
      \left(v_3(R_{3,k}) - v_2 (R_{2,k})\right) \rho_{3,k} (1-\rho_{2,k})\right].
  \end{align*}
  The positivity of $\rho_{2,k}^{n+1}$ follows immediately, since
  $v_2(R_{2,k}) - v_1 (R_{1,k})\geq 0$ and
  $v_3(R_{3,k}) -v_2 (R_{2,k})< 0$. On the other hand, thanks
  to~\ref{it:utile2},
  \begin{align*}
    \rho_{2,k}^{n+1}  = \
    & \rho_{2,k}
      \left( 1
      - \Delta t \left(
      \left(v_2(R_{2,k}) - v_1 (R_{1,k})\right) \rho_{1,k}
      + \left(v_3(R_{3,k}) - v_2 (R_{2,k})\right) \rho_{3,k}
      \right)\right)
    \\
    & + \Delta t
      \left(\left(v_2(R_{2,k}) - v_1 (R_{1,k})\right) \rho_{1,k}
      - \left(v_3(R_{3,k}) - v_2 (R_{2,k})\right) \rho_{3,k} \right)
    \\
    \leq \
    & 1
      - \Delta t \left(
      \left(v_2(R_{2,k}) - v_1 (R_{1,k})\right) \rho_{1,k}
      + \left(v_3(R_{3,k}) - v_2 (R_{2,k})\right) \rho_{3,k}
      \right)
    \\
    &
      + \Delta t
      \left(\left(v_2(R_{2,k}) - v_1 (R_{1,k})\right) \rho_{1,k}
      - \left(v_3(R_{3,k}) - v_2 (R_{2,k})\right) \rho_{3,k} \right)
    \\
    \leq \ & 1.
  \end{align*}

  \paragraph{Case D.} In this latter case~\eqref{eq:1} reads
  \begin{align*}
    \rho_{2,k}^{n+1} = \
    & \rho_{2,k} + \Delta t
      \left[\left(v_2(R_{2,k}) - v_1 (R_{1,k})\right) \rho_{2,k} (1-\rho_{1,k})
      -
      \left(v_3(R_{3,k}) - v_2 (R_{2,k})\right) \rho_{3,k} (1-\rho_{2,k})\right].
  \end{align*}
  By~\ref{it:utile1}, since $0 \leq \rho_{2,k} \leq 1$, we get
  \begin{align*}
    \rho_{2,k}^{n+1} \leq \
    & \rho_{2,k} - \Delta t
      \left(v_3(R_{3,k}) - v_2 (R_{2,k})\right) \rho_{3,k} (1-\rho_{2,k})
    \\
    = \
    &
      \rho_{2,k} \left( 1 + \Delta t
      \left(v_3(R_{3,k}) - v_2 (R_{2,k})\right) \rho_{3,k}
      \right)
      - \Delta t
      \left(v_3(R_{3,k}) - v_2 (R_{2,k})\right) \rho_{3,k}
    \\
    \leq \
    &
      1 + \Delta t
      \left(v_3(R_{3,k}) - v_2 (R_{2,k})\right) \rho_{3,k}
      - \Delta t
      \left(v_3(R_{3,k}) - v_2 (R_{2,k})\right) \rho_{3,k}
    \\
    \leq \ & 1.
  \end{align*}
  The positivity of $\rho_{2,k}^{n+1}$ follows from the CFL
  condition~\eqref{eq:CFL}:
  \begin{align*}
    \rho_{2,k}^{n+1} \geq \
    & \rho_{2,k} + \Delta t
      \left(v_2(R_{2,k}) - v_1 (R_{1,k})\right) \rho_{2,k} (1-\rho_{1,k})
    \\
    \geq \
    &
      \rho_{2,k} \left( 1 +  \Delta t
      \left(v_2(R_{2,k}) - v_1 (R_{1,k})\right)\right)
    \\
    \geq \
    &  \rho_{2,k} \left( 1 -  \Delta t \,  v_1 (R_{1,k})\right)
    \\
    \geq \
    & 0.
  \end{align*}
  \smallskip
  \noindent The proof is completed.
\end{proof}

\subsection{Conservation of total mass}
\label{sec:conservation}

When considering an initial datum $\brho_o$ with finite total mass,
that is $\sum_{j=1}^M \norma{\rho_{o,j}}_{\L1(\reali)} < + \infty$, it
is possible to prove that the corresponding solution preserves this
norm. Clearly, because of lane changing, the $\L1$-norm is not
preserved in each lane, but only in the whole.

\begin{lemma}
  \label{lem:cons}
  Let $\brho_o \in \L1 (\reali; [0,1]^M)$. Under the CFL
  condition~\eqref{eq:CFL}, the piece-wise constant approximate
  solution $\brho_\Delta$ constructed through Algorithm~\ref{alg:1}
  preserves the $\L1$-norm, in the sense that for all $t>0$
  \begin{displaymath}
    \norma{\brho_\Delta (t)}_{\L1 (\reali)}
    =
    \sum_{j=1}^M \norma{\rho_{j,\Delta}(t)}_{\L1(\reali)}
    =
    \sum_{j=1}^M \norma{\rho_{o,j}}_{\L1(\reali)}
    =
    \norma{\brho_o}_{\L1 (\reali)}.
  \end{displaymath}
\end{lemma}

\begin{proof}
  The proof is done by induction. Since the Godunov type
  scheme~\eqref{eq:convStep} is
  conservative~\cite[Chapter~13]{LeVeque}, we have
  \begin{displaymath}
    \sum_{j=1}^M \norma{\rho^{n+1/2}_j}_{\L1(\reali)}
    =
    \sum_{j=1}^M \norma{\rho_{o,j}}_{\L1(\reali)}.
  \end{displaymath}
  The positivity of $\brho_\Delta$ and the fact that the source terms
  sum up to $0$ when considering the relaxation step
  in~\eqref{eq:relaxStep} yields the thesis:
  \begin{displaymath}
    \sum_{j=1}^M \norma{\rho^{n+1}_j}_{\L1(\reali)}
    =
    \sum_{j=1}^M \norma{\rho^{n+1/2}_j}_{\L1(\reali)}
    =
    \sum_{j=1}^M \norma{\rho_{o,j}}_{\L1(\reali)}.
  \end{displaymath}
\end{proof}

\subsection{\texorpdfstring{$\BV$}{BV} estimates}
\label{sec:BV}

We first prove the Lipschitz continuity of the source
term~\eqref{eq:source} in each of its argument.
\begin{lemma}
  \label{lem:Lipsource}
  For all $j=1, \ldots, M$, the map $S_j$ defined in~\eqref{eq:source}
  is Lipschitz continuous in each argument with Lipschitz constant
  \begin{equation}\label{eq:lipschitzconstant}
    \mathcal{K}= \max\{ V_{\mathrm{max}},2 \, V'_{\mathrm{max}}\},
  \end{equation} where $V_{\mathrm{max}}$ and $V'_{\mathrm{max}}$ are defined
  in~\eqref{eq:vmax} and~\eqref{eq:vdiffmax} respectively.
\end{lemma}
\begin{proof}
  For $j\in \{1,\dots,M-1\}$ we have
  \begin{align}
    & \modulo{S_j(\rho_j,\rho_{j+1},R_j,R_{j+1})- S_j(\tilde \rho_j,\tilde \rho_{j+1},\tilde R_j,\tilde R_{j+1})}\notag\\
    \leq  \
    & \modulo{S_j(\rho_j,\rho_{j+1},R_j,R_{j+1})- S_j(\tilde \rho_j,\rho_{j+1},R_j,R_{j+1})} \label{eq:lipschitz1}\\
    & + \modulo{S_j(\tilde \rho_j,\rho_{j+1},R_j,R_{j+1})- S_j(\tilde \rho_j,\tilde \rho_{j+1},R_j,R_{j+1})}
      \label{eq:lipschitz2}\\
    & + \modulo{S_j(\tilde \rho_j,\tilde \rho_{j+1},R_j,R_{j+1})- S_j(\tilde \rho_j,\tilde \rho_{j+1},\tilde R_j,R_{j+1})}
      \label{eq:lipschitz3}\\
    & + \modulo{S_j(\tilde \rho_j,\tilde \rho_{j+1},\tilde R_j,R_{j+1})- S_j(\tilde \rho_j,\tilde \rho_{j+1},\tilde R_j,\tilde R_{j+1})}.
      \label{eq:lipschitz4}
  \end{align}
  By the definition of the source term~\eqref{eq:source} we have
  \begin{align*}
    [\eqref{eq:lipschitz1}]= \
    & \modulo{(v_{j+1}(R_{j+1})-v_j(R_j))^+(1-\rho_{j+1})(\rho_j-\tilde \rho_j)-(v_{j+1}(R_{j+1})-v_j(R_j))^- \rho_{j+1} (\tilde \rho_j -\rho_j)}\\
    & \leq V_{\max} \modulo{\rho_j-\tilde \rho_j},
    \\
    [\eqref{eq:lipschitz2}] \leq \
    & V_{\max} \modulo{\rho_{j+1}-\tilde \rho_{j+1}}.
  \end{align*}
  Pass now to~\eqref{eq:lipschitz3}:
  \begin{align*}
    [\eqref{eq:lipschitz3}] = \
    & \left|\left((v_{j+1}(R_{j+1})-v_j(R_j))^+-(v_{j+1}(R_{j+1})-v_j(\tilde R_j))^+\right)\tilde \rho_j (1-\tilde \rho_{j+1})\right.\\
    & \left.-\left((v_{j+1}(R_{j+1})-v_j(R_j))^--(v_{j+1}(R_{j+1})-v_j(\tilde R_j))^-\right)\tilde \rho_{j+1} (1-\tilde \rho_{j})\right| .
  \end{align*}
  We distinguish the following cases:
  \begin{center}
    \begin{tabular}{c|c|c}
      & $v_{j+1} (R_{j+1}) \geq v_j (R_{j})$
      & $v_j (R_{j+1}) < v_j (R_{j}) $\\
      \hline
      $v_{j+1} (R_{j+1}) \geq v_j (\tilde R_{j})$
      &
        \bf{Case A} & \bf{Case B}
      \\
      \hline
      $v_{j+1} (R_{j+1}) < v_j (\tilde R_{j})$ & \bf{Case C} & \bf{Case D}
    \end{tabular}
  \end{center}
  We analyse in detail cases A and B, the others being entirely
  similar.
  \paragraph{Case A.} We have
  \begin{displaymath}
    [\eqref{eq:lipschitz3}] =
    \modulo{\left(v_j(\tilde R_j)-v_j(R_j)\right)\tilde \rho_j (1-\tilde \rho_{j+1})}
    \leq V'_{\max} \modulo{R_j-\tilde R_j}.
  \end{displaymath}
  \paragraph{Case B.} Add and subtract
  $(v_{j+1}(R_{j+1})-v_j(\tilde R_j))\tilde \rho_{j+1}(1-\tilde
  \rho_j)$ inside the absolute value in~\eqref{eq:lipschitz3} to
  obtain
  \begin{align*}
    [\eqref{eq:lipschitz3}] = \
    &\modulo{(v_j(\tilde R_j)-v_j(R_j))\tilde \rho_{j+1} (1-\tilde \rho_j)
      + (v_{j+1}(R_{j+1})-v_j(\tilde R_j))(\tilde \rho_{j+1}(1-\tilde \rho_j)-\tilde \rho_j(1-\tilde \rho_{j+1}))}\\
    &\leq V'_{\max}\modulo{R_j-\tilde R_j}+ (v_{j+1}(R_{j+1})-v_j(\tilde R_j))\\
    &< V'_{\max} \modulo{R_j-\tilde R_j}+(v_{j}(R_{j})-v_j(\tilde R_j))\\
    &\leq 2 \, V'_{\max}\modulo{R_j-\tilde R_j},
  \end{align*}
  since $v_{j+1} (R_{j+1}) < v_j (R_j)$ and
  $\modulo{\tilde \rho_{j+1} - \tilde \rho_j}\leq 1$, with
  $\tilde \rho_j, \tilde \rho_{j+1} \in [0,1]$.

  Cases D and C are treated similarly to Case A and Case B,
  respectively. Therefore we have
  \begin{displaymath}
    [ \eqref{eq:lipschitz3} ] \leq
    2 \, V'_{\max}\modulo{R_j-\tilde R_j}.
  \end{displaymath}

  The term~\eqref{eq:lipschitz4} is treated analogously
  to~\eqref{eq:lipschitz3}, leading to
  \begin{displaymath}
    [\eqref{eq:lipschitz4}] \leq
    2 \,V'_{\max}\modulo{R_{j+1}-\tilde R_{j+1}}.
  \end{displaymath}
  The proof is completed.
\end{proof}

The Lipschitz continuity of the source term proved in
Lemma~\ref{lem:Lipsource} is one of the key ingredients in order to
prove the following total variation bound on the numerical
approximation.

\begin{proposition}[\bf{$\BV$ estimate in space}]\label{prop:BVspace}
  Let $\brho_o \in (\L1 \cap \BV) (\reali; [0,1]^M)$.
  Assume that the CFL condition~\eqref{eq:CFL} holds. Then, for
  $n=0,\dots, N_T-1$ the following estimate holds
  \begin{equation}
    \label{eq:BVspace}
    \sum_{j=1}^M \sum_{k\in\interi}
    \modulo{\rho_{j,k+1}^{n} - \rho_{j,k}^n}
    \leq
    e^{8 \, t^n \, \mathcal{K}}
    \sum_{j=1}^M  \sum_{k\in\interi}  \modulo{\rho_{j,k+1}^{0} - \rho_{j,k}^0}
    =
    e^{8 \, t^n \, \mathcal{K}}
    \sum_{j=1}^M \tv(\rho_{j}^0).
  \end{equation}
\end{proposition}
\begin{proof}
  By~\eqref{eq:relaxStep}, for $j=1,\dots, M$ we have
  \begin{align*}
    & \rho_{j,k+1}^{n+1}-\rho_{j,k}^{n+1}
    \\
    = \
    & \rho_{j,k+1}^{n+1/2} - \rho_{j,k}^{n+1/2}
    \\
    & + \Delta t
      \left[
      S_{j-1}\left(\rho_{j-1, k+1}^{n+1/2}, \rho_{j, k+1}^{n+1/2},
      R_{j-1, k+1}^{n+1/2}, R_{j, k+1}^{n+1/2}\right)
      -
      S_{j-1}\left(\rho_{j-1, k}^{n+1/2}, \rho_{j, k}^{n+1/2},
      R_{j-1, k}^{n+1/2}, R_{j, k}^{n+1/2}\right)
      \right]
    \\
    & - \Delta t \left[
      S_{j}\left(\rho_{j, k+1}^{n+1/2}, \rho_{j+1, k+1}^{n+1/2},
      R_{j, k+1}^{n+1/2}, R_{j+1, k+1}^{n+1/2}\right)
      -
      S_{j}\left(\rho_{j, k}^{n+1/2}, \rho_{j+1, k}^{n+1/2},
      R_{j, k}^{n+1/2}, R_{j+1, k}^{n+1/2}\right)
      \right].
  \end{align*}
  By the Lipschitz continuity of the maps in the source term, see
  Lemma~\ref{lem:Lipsource}, and the properties of the discrete
  convolution operator, see Lemma~\ref{lem:convOp}, we obtain
  \begin{align*}
    &\sum_{j=1}^M \sum_{k\in\interi}
      \modulo{\rho_{j, k+1}^{n+1} - \rho_{j, k}^{n+1}}
    \\
    \leq  \
    & \sum_{j=1}^M \sum_{k\in\interi}
      (1 + 4 \, \Delta t \, \mathcal{K})
      \modulo{\rho_{j, k+1}^{n+1/2} - \rho_{j, k}^{n+1/2}}
      +
      4 \, \Delta t \, \mathcal{K}
      \sum_{j=1}^M \sum_{k\in\interi}
      \modulo{R_{j, k+1}^{n+1/2}-R_{j, k}^{n+1/2}}
    \\
    \leq \
    &
      (1+8 \, \dt \, \mathcal{K})
      \sum_{j=1}^M \sum_{k\in\interi}
      \modulo{\rho_{j, k+1}^{n+1/2}-\rho_{j, k}^{n+1/2}}.
  \end{align*}
  Since the Godunov scheme used in~\eqref{eq:numFlux} is total
  variation diminishing~\cite[Proposition 3.1
  (d)]{crandall1980monotone}, we get
  \begin{equation}
    \label{eq:TVDstep}
    \sum_{j=1}^M \sum_{k\in\interi}
    \modulo{\rho_{j, k+1}^{n+1} - \rho_{j, k}^{n+1}}
    \leq
    (1 + 8 \, \dt \, \mathcal{K})
    \sum_{j=1}^M\sum_{k\in\interi}
    \modulo{\rho_{j, k+1}^{n} - \rho_{j, k}^{n}}
    \leq
    e^{8 \, \Delta t \, \mathcal{K}}
    \sum_{j=1}^M\sum_{k\in\interi}
    \modulo{\rho_{j, k+1}^{n}-\rho_{j, k}^{n}},
  \end{equation}
  which applied recursively yields the thesis.
\end{proof}

\begin{proposition}
  \label{prop:LipTime}
  Let $\brho_o \in (\L1 \cap \BV) (\reali; [0,1]^M)$.
  Assume that the CFL condition~\eqref{eq:CFL} holds. Then, for
  $n=0,\dots, N_T-1$,
  \begin{equation}
    \label{eq:lipTime}
    \Delta x \sum_{j=1}^M \sum_{k\in\interi} \modulo{\rho_{j,k}^{n+1}-\rho_{j,k}^n}
    \leq
    2 \, \dt \left(
      2 \, V_{\max} \norma{\brho_o}_{\L1 (\reali)}
      +
      \mathcal{V} \,  e^{8 \, t^n \, \mathcal{K}}
      \sum_{j=1}^M \tv(\rho_j^0)\right),
  \end{equation}
  with $\mathcal{K}$ as in~\eqref{eq:lipschitzconstant} and
  $\mathcal{V}$ as in~\eqref{eq:VC1}.
\end{proposition}
\begin{proof}
  Observe that
  \begin{displaymath}
    \modulo{\rho^{n+1}_{j,k} - \rho^n_{j,k}}
    \leq
    \modulo{\rho^{n+1}_{j,k} - \rho^{n+1/2}_{j,k}}
    +
    \modulo{\rho^{n+1/2}_{j,k} - \rho^n_{j,k}}.
  \end{displaymath}
  We then estimate separately each term on the right hand side of the
  inequality above.

  By the relaxation step~\eqref{eq:relaxStep} we have
  \begin{displaymath}
    \modulo{\rho^{n+1}_{j,k} - \rho^{n+1/2}_{j,k}}
    =
    \dt \modulo{
      S_{j-1}\left( \rho_{j-1,k}^{n+1/2}, \rho_{j,k}^{n+1/2},
        R_{j-1,k}^{n+1/2}, R_{j,k}^{n+1/2}\right)
      -
      S_{j}\left( \rho_{j,k}^{n+1/2}, \rho_{j+1,k}^{n+1/2},
        R_{j,k}^{n+1/2}, R_{j+1,k}^{n+1/2}\right)
    }.
  \end{displaymath}
  It is easy to see that the numerical source term
  $S_j$~\eqref{eq:source}--\eqref{eq:boundaryconditions} satisfies,
  for $j=1, \ldots, M$,
  \begin{equation}
    \label{eq:boundSource}
    \modulo{S_{j}\left(
        \rho_{j,k}^{n+1/2},\rho_{j+1,k}^{n+1/2},R_{j,k}^{n+1/2},R_{j+1,k}^{n+1/2}
      \right)}
    \leq
    V_{\max}\left(\rho_{j,k}^{n+1/2} + \rho_{j+1,k}^{n+1/2}\right).
  \end{equation}
  Thus,
  \begin{equation}
    \label{eq:17}
    \modulo{\rho^{n+1}_{j,k} - \rho^{n+1/2}_{j,k}}
    \leq
    \dt \, V_{\max} \left(
      \rho_{j-1,k}^{n+1/2}
      +
      2 \, \rho_{j-1,k}^{n+1/2}
      +
      \rho_{j-1,k}^{n+1/2}
    \right).
  \end{equation}

  By the convective step~\eqref{eq:convStep}, since the numerical flux
  defined in~\eqref{eq:numFlux} is Lipschitz continuous in both
  arguments with Lipschitz constant $\mathcal{V}$~\eqref{eq:VC1}, we
  have
  \begin{align}
    \nonumber
    \modulo{\rho^{n+1/2}_{j,k} - \rho^n_{j,k}}
    = \
    & \lambda \modulo{F_j\left(\rho_{j,k}^n, \rho_{j,k+1}^n\right)
      -
      F_j\left(\rho_{j,k-1}^n, \rho_{j,k}^n\right)}
    \\
    \label{eq:18}
    \leq \
    & \lambda \, \mathcal{V} \left(
      \modulo{\rho_{j,k}^n- \rho_{j,k-1}^n}
      +
      \modulo{\rho_{j,k+1}^n- \rho_{j,k}^n}
      \right).
  \end{align}

  Collecting together~\eqref{eq:17} and~\eqref{eq:18} and exploiting
  Lemma~\ref{lem:cons} and Proposition~\ref{prop:BVspace} yields
  \begin{align*}
    \dx \sum_{j=1}^M \sum_{k \in \interi}
    \modulo{\rho^{n+1}_{j,k} - \rho^n_{j,k}}
    \leq \
    & \dt \,  V_{\max}\,
      \sum_{j=1}^M \sum_{k \in \interi}
      4 \norma{\rho_{j}^{n+1/2}}_{\L1 (\reali)}
      +
      2 \, \dt \,\mathcal{V}
      \sum_{j=1}^M \sum_{k \in \interi}
      \modulo{\rho_{j,k}^n- \rho_{j,k-1}^n}
    \\
    \leq \
    & 2 \, \dt \left(
      2 \, V_{\max} \norma{\brho_o}_{\L1 (\reali)}
      +
      \mathcal{V}\,  e^{8 \, t^n \, \mathcal{K}}
      \sum_{j=1}^M \tv(\rho_j^0)\right).
  \end{align*}
\end{proof}

Using the estimates provided by Propositions~\ref{prop:BVspace}
and~\ref{prop:LipTime}, we obtain the following $\BV$ estimate in
space and time.

\begin{corollary}[\bf{$\BV$ estimate in space and time}]\label{cor:BVspaceTime}
  Let $\brho_o \in (\L1 \cap \BV) (\reali; [0,1]^M)$.
  Assume that the CFL condition~\eqref{eq:CFL} holds. Then, for all
  $n=1, \ldots, N_T$, the following estimate holds
  \begin{align*}
    & \sum_{m=0}^{n-1} \sum_{j=1}^M
      \sum_{k\in\interi} \left(
      \Delta t
      \modulo{\rho_{j,k+1}^{m}-\rho_{j,k}^m}
      + \Delta x \modulo{\rho_{j,k}^{m+1}-\rho_{j,k}^m}\right)
    \\
    \leq \
    &
      n \, \dt \, e^{8 \, t^n \, \mathcal{K}}
      \left(\left(2 \, \mathcal{V} + 1 \right)
      \sum_{j=1}^M \tv(\rho_{j}^0)
      +
      4 \, V_{\max}\sum_{j=1}^M
      \norma{\rho_{o,j}}_{\L1(\reali)}\right).
  \end{align*}
\end{corollary}

\subsection{Discrete entropy inequality}
\label{sec:entropyIneq}

We derive a discrete entropy inequality for the approximate solution
$\brho_\Delta$ constructed through Algorithm~\ref{alg:1}. The proof is
entirely similar to~\cite[Lemma~2.7]{GoatinRossi}, with the
simplification that now the flux does not depend on the spatial
variable.

Define, for each $c\in [0,1]$ and $j= 1 , \ldots, M$, the Kru\v zkov
numerical entropy flux as
\begin{displaymath}
  \mathscr{F}_j^c (u,w)
  =
  F_j(u \vee c, w \vee c) - F_j (u \wedge c, w\wedge c),
\end{displaymath}
where $a \vee b = \max\{a,b\}$ and $a \wedge b = \min\{a, b\}$.

\begin{lemma}
  \label{lemma:discreteEntropy}
  Let $\brho_o \in (\L1 \cap \BV) (\reali; [0,1]^M)$. Assume that the
  CFL condition~\eqref{eq:CFL} holds. Then, the approximate solution
  $\brho_\Delta$ constructed by Algorithm~\ref{alg:1} satisfies the
  following discrete entropy inequality: for all $j=1, \ldots, M$, for
  $k \in \interi$, for $n=0, \ldots, N_T-1$ and for any $c\in [0,1]$,
  \begin{align}
    \nonumber
    \modulo{\rho_{j,k}^{n+1} - c}
    -
    \modulo{\rho_{j,k}^{n} - c}
    + \lambda \left(
    \mathscr{F}_j^c\left(\rho_{j,k}^n, \rho_{j,k+1}^n\right)
    -
    \mathscr{F}_j^c\left(\rho_{j,k-1}^n, \rho_{j,k}^n\right)
    \right)
    \\
    \label{eq:8}
    - \dt \, \sgn\left(\rho_{j,k}^{n+1} - c\right)
    \left(
    S_{j-1}\left(\rho_{j-1,k}^{n+1/2}, \rho_{j,k}^{n+1/2}, R_{j-1,k}^{n+1/2}, R_{j,k}^{n+1/2}\right)\right.
    \\
    \nonumber
    \left.
    -
    S_{j}\left(\rho_{j,k}^{n+1/2}, \rho_{j+1,k}^{n+1/2}, R_{j,k}^{n+1/2}, R_{j+1,k}^{n+1/2}\right)
    \right)
 & \leq 0.
  \end{align}
\end{lemma}

\subsection{Convergence}
\label{sec:conv}

The results obtained in the preceding sections, namely
Lemma~\ref{lem:invariance} for the invariance of the set $[0,1]^M$ and
Corollary~\ref{cor:BVspaceTime} for the total variation bound in space
and time, allow to apply Helly's compactness theorem, which ensures
the existence of a subsequence of $\brho_\Delta$ converging in $\L1$
to a function $\brho \in \L\infty([0,T] \times \reali; [0,1]^M)$, with
the additional property of preserving the initial mass, that is
$\norma{\brho(t)}_{\L1(\reali)}=\norma{\brho_o}_{\L1(\reali)}$ for
$t \in [0,T]$. Moreover, Proposition~\ref{prop:LipTime} and in
particular formula~\eqref{eq:lipTime}, imply that
$\brho \in \C0([0, T];\L1 (\reali; [0, 1]^M))$.

The limit function $\brho$ is a solution to problem~\eqref{eq:model}
in the sense of Definition~\ref{def:sol}. Indeed, the weak
formulation, i.e.~the integral equality in the first part of
Definition~\ref{def:sol}, follows from a Lax--Wendroff type
calculation~\cite[Theorem~12.1]{LeVeque}, and the presence of the
source terms does not add any difficulty in the proof.

Concerning the entropy inequality in the second part of
Definition~\ref{def:sol}, rather standard computations starting from
the discrete entropy inequality in Lemma~\ref{lemma:discreteEntropy}
yield the desired result.

\section{Uniqueness of solutions: \texorpdfstring{$\L1$}{L1}
  contractivity}
\label{sec:unique}

As for the \emph{local} model~\cite{HoldenRisebro}, the special form
of the source terms implies the $\L1$-contractivity of the solution
to~\eqref{eq:model}. In particular, this results guarantees uniqueness
of solutions to problem~\eqref{eq:model}.

\begin{theorem}
  \label{thm:contract}
  Let $\brho$ and $\bpi$ be two solutions to problem~\eqref{eq:model}
  in the sense of Definition~\ref{def:sol}, with initial data
  $\brho_o, \, \bpi_o \in (\L1 \cap \BV)(\reali; [0,1]^M),$
  respectively.
  Then, for a.e.~$t \in [0,T]$,
  \begin{equation}
    \label{eq:9}
    \sum_{j=1}^M \norma{\rho_j (t) - \pi_j (t)}_{\L1 (\reali)}
    \leq
    \sum_{j=1}^M \norma{\rho_{j,o} - \pi_{j,o}}_{\L1 (\reali)}.
  \end{equation}
\end{theorem}

\begin{proof}
  The proof follows the idea of~\cite[Theorem~3.3]{HoldenRisebro},
  with the main difference that now the source terms are
  \emph{nonlocal} functions of the solution. We recall the proof
  briefly for completeness, focusing mainly on those parts where the
  nonlocality comes in.

  Kru\v zkov doubling of variables techniques, together with the fact
  that $\brho$ and $\bpi$ are solutions to~\eqref{eq:model}, yields,
  for $\tau\in [0,T]$ and for any $j=1, \ldots, M$,
  \begin{equation}
    \label{eq:10}
    \begin{aligned}
      \int_\reali \left(\rho_j(\tau) - \pi_j(\tau)\right)^+ \d{x} \leq
      \ & \int_\reali \left(\rho_j(0) - \pi_j(0)\right)^+ \d{x}
      \\
      & + \int_0^\tau \int_\reali H(\rho_j-\pi_j) \left(
        \mathcal{S}(\brho, \boldsymbol{R}, j) - \mathcal{S}(\bpi,
        \boldsymbol{P}, j) \right) \d{x} \d{t},
    \end{aligned}
  \end{equation}
  where $H$ is the Heaviside function,
  \begin{align*}
    \mathcal{S}(\boldsymbol{u}, \boldsymbol{U}, j) = \
    & S_{j-1}(u_{j-1},u_j, U_{j-1}, U_j) -
      S_{j}(u_{j},u_{j+1}, U_{j}, U_{j+1}),
    \\
    U_j(t,x) = \
    & (u_j(t) * w_\nds)(x),
  \end{align*}
  and we denote by $\boldsymbol{U}$ the vector of components $U_j$,
  $j=1, \ldots, M$.  It can be easily verified that the map
  $S_j(u,w,U,W)$ defined in~\eqref{eq:source} is nondecreasing in the
  first and third variables and nonincreasing in the second and fourth
  variables, thus $\partial_u S_j, \, \partial_U S_j \geq 0$ and
  $\partial_w S_j, \, \partial_W S_j \leq 0$. Hence, if
  $\rho_j > \pi_j$, clearly $R_j > P_j$ and moreover
  \begin{align*}
    & \mathcal{S}(\brho, \boldsymbol{R}, j)
      -
      \mathcal{S}(\bpi, \boldsymbol{P}, j)
    \\
    = \
    & S_{j-1} (\rho_{j-1}, \rho_{j}, R_{j-1}, R_{j})
      -
      S_{j-1} (\pi_{j-1}, \pi_{j}, P_{j-1}, P_{j})
    \\
    & -
      S_{j} (\rho_{j}, \rho_{j+1}, R_{j}, R_{j+1})
      +
      S_{j} (\pi_{j}, \pi_{j+1}, P_{j}, P_{j+1})
    \\
    \leq \
    & S_{j-1} (\rho_{j-1}, \pi_{j}, R_{j-1}, P_{j})
      -
      S_{j-1} (\pi_{j-1}, \pi_{j}, P_{j-1}, P_{j})
    \\
    & -
      S_{j} (\rho_{j}, \rho_{j+1}, R_{j}, R_{j+1})
      +
      S_{j} (\rho_{j}, \pi_{j+1}, R_{j}, P_{j+1})
    \\
    = \
    & \partial_u S_{j-1} (\sigma_{j-1}, \pi_{j}, R_{j-1}, P_{j})
      \, \left(\rho_{j-1} - \pi_{j-1}\right)
      +
      \partial_U S_{j-1} (\pi_{j-1}, \pi_{j}, T_{j-1}, P_{j})
      \, \left(R_{j-1} - P_{j-1}\right)
    \\
    & -
      \partial_w S_{j} (\rho_{j}, \sigma_{j+1}, R_{j}, R_{j+1})
      \, \left(\rho_{j+1} - \pi_{j+1}\right)
      -
      \partial_W S_{j} (\rho_{j}, \sigma_{j+1}, R_{j}, T_{j+1})
      \, \left(R_{j+1}- P_{j+1}\right)
    \\
    \leq \
    & \mathcal{K} \left(
      \left(\rho_{j-1} - \pi_{j-1}\right)^+
      +
      \left(R_{j-1} - P_{j-1}\right)^+
      +
      \left(\rho_{j+1} - \pi_{j+1}\right)^+
      +
      \left(R_{j+1} - P_{j+1}\right)^+
      \right),
  \end{align*}
  where $\sigma_{j \pm 1}$ lies in the interval between
  $\rho_{j \pm 1}$ and $\pi_{j \pm 1}$, $T_{j \pm 1}$ lies in the
  interval between $R_{j \pm 1}$ and $P_{j \pm 1}$ and the Lipschitz
  constant $\mathcal{K}$ of the map $S_j$ is as
  in~\eqref{eq:lipschitzconstant}.  Therefore
  \begin{equation}
    \label{eq:11}
    \sum_{j=1}^M H(\rho_j - \pi_j)
    \left(\mathcal{S}(\brho, \boldsymbol{R},j)
      -
      \mathcal{S}(\bpi, \boldsymbol{P}, j)
    \right)
    \leq
    2 \, \mathcal{K} \sum_{j=1}^M (\rho_j - \pi_j)^+
    +
    2 \, \mathcal{K} \sum_{j=1}^M (R_j - P_j)^+.
  \end{equation}
  Observe that
  $\int_\reali (g*w_\nds)^+(x) \d{x} = \int_\reali g^+ (x) \d{x}$,
  thus, due to~\eqref{eq:disConvOp}, when integrating~\eqref{eq:11}
  over $\reali$ we obtain
  \begin{equation}
    \label{eq:12}
    \sum_{j=1}^M \int_\reali H(\rho_j - \pi_j)
    \left(\mathcal{S}(\brho, \boldsymbol{R},j)
      -
      \mathcal{S}(\bpi, \boldsymbol{P}, j)
    \right) \d{x}
    \leq
    4 \, \mathcal{K} \sum_{j=1}^M \int_\reali(\rho_j - \pi_j)^+ \d{x}.
  \end{equation}
  Define
  \begin{displaymath}
    \Theta(t) = \sum_{j=1}^M \int_\reali (\rho_j - \pi_j)^+ \d{x},
  \end{displaymath}
  so that, collecting together~\eqref{eq:10} and~\eqref{eq:12}, we get
  \begin{displaymath}
    \Theta (\tau)
    \leq \Theta (0) + 4 \, \mathcal{K} \int_0^\tau \Theta (t) \d{t}.
  \end{displaymath}
  Gronwall's inequality yields
  $\Theta(t) \leq e^{4 \, \mathcal{K} \, t} \, \Theta (0)$. If
  $\Theta (0) = 0$, that is $\rho_{o,j}(x) \leq \pi_{o,j}(x)$ a.e.~in
  $\reali$ for all $j$, then $\Theta(t) = 0$ for $t>0$, that is
  $\rho_j (t,x)\leq \pi_j (t,x)$ a.e.~in $\reali$ for all $j$.
  \\
  The proof of $\L1$-contractivity is concluded by an application of
  the Crandall--Tartar lemma~\cite[Lemma~2.13]{HoldenRisebroBook2015}.
\end{proof}

Following~\cite[Corollary~3.4]{HoldenRisebro}, the $\L1$-contractivty
of the solution proved in Theorem~\ref{thm:contract} guarantees that
the solution to problem~\eqref{eq:model} satisfies some \emph{a
  priori} estimates.

\begin{corollary}
  \label{cor:niceCons}
  Let $\brho $ be a solution to problem~\eqref{eq:model} in the sense
  of Definition~\ref{def:sol}, with initial datum
  $\brho_o \in (\L1 \cap \BV)(\reali; [0,1]^M)$. Then,
  \begin{align}
    \label{eq:cons1}
    \sum_{j=1}^{M-1} \norma{\rho_{j+1}(t) - \rho_j(t)}_{\L1(\reali)}
    \leq \
    & \sum_{j=1}^{M-1} \norma{\rho_{j+1,o} - \rho_{j,o}}_{\L1(\reali)},
    \\
    \label{eq:cons2}
    \sum_{j=1}^M \tv\left(\rho_j(t)\right)
    \leq \
    & \sum_{j=1}^M \tv\left(\rho_{j,o}\right),
    \\
    \label{eq:cons3}
    \sum_{j=1}^{M} \norma{\rho_{j}(t+h) - \rho_j(t)}_{\L1(\reali)}
    \leq \
    & \sum_{j=1}^{M} \norma{\rho_{j}(h) - \rho_{j,o}}_{\L1(\reali)}, \quad h \in \reali.
  \end{align}
\end{corollary}
The proof relies solely on~\eqref{eq:9}, together with the enforced
boundary conditions $\rho_0(t,x) = \rho_1(t,x)$, $v_0(u) = v_1(u)$,
$\rho_{M+1} (t,x) = \rho_M (t,x)$, $v_{M+1}(u)=v_M(u)$.

Notice that Corollary~\ref{cor:niceCons} provides better estimates
than those coming from the approximate solution built in
Section~\ref{sec:numScheme}. Compare in particular~\eqref{eq:cons2} to
the total variation in space provided by~\eqref{eq:BVspace}.

\section{A multilane model with nonlocal flux and nonlocal source
  term}
\label{sec:ext}

In the following, we consider a modification of
problem~\eqref{eq:model} assuming additionally a \emph{nonlocal}
velocity in the flux function.  In particular, the treatment of the
nonlocal flux in each lane is inspired by~\cite{BlandinGoatin}. The
problem under consideration reads
\begin{equation}
  \label{eq:nonlocaltransport}
  \left\{
    \begin{array}{lr}
      \partial_t \rho_j + \partial_x \left(\rho_j \, v_j\left(\rho_j \ast w_\ndt \right)\right) = S_{j-1}(\rho_{j-1}, \rho_j, R_{j-1}, R_j) -S_j(\rho_j, \rho_{j+1}, R_j, R_{j+1})
      & j= 1, \ldots, M,
      \\
      \rho_{j}(0,x) = \rho_{o,j}(x)
      & j= 1, \ldots, M.
    \end{array}
  \right.
\end{equation}
In order to have a well defined model, we only consider kernel
functions such that $\spt w_\ndt \subseteq [0,\ndt]$, meaning that
drivers adapt their speed to the downstream traffic. In addition, we
assume that the kernel $w_\ndt \in \C1([0, \ndt]; \reali_+)$ is
non-increasing, i.e.~$w_\ndt'\leq 0$, and, as usual
$\int_\reali w_\ndt = 1$. The convolution product is thus defined as
\begin{equation}
  \label{eq:Rjtransport}
  R^\ndt_j = R^\ndt_j(t,x) = \left(\rho(t) \ast w_\ndt\right)(x):= \int_x^{x+\ndt} w_\ndt(y-x)\rho(t,y) dy.
\end{equation}
We remark that the additional assumptions on the kernel $w_\ndt$ in
the flux are not needed for the kernel $w_\nds$ in the source term,
see Section~\ref{sec:nonlocal_source}. Moreover, when considering both
nonlocal flux and nonlocal source, we underline that the kernels may
differ.  In the following, we denote the convolution products in the
source by $R_j^\nds$~\eqref{eq:Rj} and those in the flux by
$R_j^\ndt$~\eqref{eq:Rjtransport}, to emphasize the different kernels.

We underline that the kernel function $w_\nds$ appearing in the source
can look either only forward or both back- and forward, differently
from the kernel function $w_\ndt$ appearing in the flux, which is
assumed to be only forward-looking.
As already mentioned in the introduction, these are the key points in
which the proposed model~\eqref{eq:nonlocaltransport} differs from the
approach presented in~\cite{BayenKeimer-preprint}. Therein, the
uniqueness is only shown for the same nonlocality in the flux and
source term, such that both have to be forward looking with the same
non-increasing kernel and the same nonlocal range. So the model~\eqref{eq:nonlocaltransport} provides more flexibility in terms of
modelling. However, using the same non-increasing, forward looking
kernel and nonlocal range, the model~\eqref{eq:nonlocaltransport} fits
into the framework proposed in~\cite[Definition 1.1, Assumption 2.2,
Assumption 3.1]{BayenKeimer-preprint}.
We also note that in~\cite{BayenKeimer-preprint} the authors use a
different technique to show existence and uniqueness of solutions,
which enables them to prove uniqueness without an entropy
condition.

We proceed as in Section~\ref{sec:numScheme}: We construct a sequence
of approximate solutions to problem~\eqref{eq:nonlocaltransport} and
prove its convergence. The approximate solution $\brho_\Delta$ is
defined as in~\eqref{eq:rhoDelta} and it is constructed as in
Algorithm~\ref{alg:1}, substituting the numerical flux
in~\eqref{eq:numFlux} by
\begin{equation}
  \label{eq:numFluxNL}
  F_j(\rho^n_{j,k},  R^{\ndt,n}_{j,k})=v_j(R^{\ndt,n}_{j,k}) \, \rho^n_{j,k},
\end{equation}
and the convective step~\eqref{eq:convStep} by
\begin{equation}
  \label{eq:convStepNL}
  \rho^{n+1/2}_{j,k}
  =
  \rho^n_{j,k}
  -
  \lambda \left[
    F_j(\rho^n_{j,k},  R^{\ndt,n}_{j,k}) - F_j(\rho^n_{j,k-1},  R^{\ndt,n}_{j,k-1}),
  \right]
\end{equation}
where $R^{\ndt,n}_{j,k}$ is computed as in~\eqref{eq:disConvOp}, and
in particular as in~\eqref{eq:forward}, with $w_\ndt$ instead of
$w_\nds$. Due to the definition of the kernel $w_\ndt$, notice that
the case~\eqref{eq:backfor} does not apply to the present
setting.
Accordingly, we rename the discrete convolution appearing in the
source, defined by~\eqref{eq:disConvOp}, as
$R^{\nds,n+1/2}_{j,k}$. The choice of the numerical
flux~\eqref{eq:numFluxNL} follows from~\cite{chiarello2019multiclass,
  friedrich2018godunov}.

We report below the definition of solution to
problem~\eqref{eq:nonlocaltransport}, analogous to
Definition~\ref{def:sol}, and then recall the main results, analogous
to those in Section~\ref{sec:numScheme}. Only those parts of the
proofs which are substantially different will be reported.

\begin{definition}
  \label{def:solNL}
  Let $\rho_{o,j} \in (\L1 \cap \BV)(\reali;[0,1])$, for
  $j=1, \ldots, M$. We say that
  $\rho_j \in \C0([0,T]; \L1(\reali; [0,1]))$, with
  $\rho_j(t, \cdot) \in \BV(\reali; [0,1])$ for $t \in [0,T]$, is a
  \emph{weak solution} to~\eqref{eq:nonlocaltransport} with initial
  datum $\rho_{o,j}$ if for any
  $\phi \in \Cc1([0,T[ \times \reali; \reali)$ and for all
  $j=1, \ldots, M$
  \begin{multline*}
    \int_0^T \!\!\!\int_\reali\!\! \left( \rho_j \, \partial_t \phi +
      \rho_j \, V_j \, \partial_x \phi +
      \left(S_{j-1}(\rho_{j-1},\rho_j, R^\nds_{j-1}, R^\nds_j) -
        S_j(\rho_j, \rho_{j+1}, R^\nds_j, R^\nds_{j+1}) \right)\phi
    \right)\! \d{x} \d{t}
    \\
    + \int_\reali \rho_{o,j} \, \phi(0,x) \d{x} = 0,
  \end{multline*}
  where
  $V_j (t,x) = v_j\left(\left(\rho_j(t) * w_\ndt\right)(x)\right)$,
  $S_j $ is as in~\eqref{eq:source} and
  $R^\nds_j = R^\nds_j (t,x) = \left(\rho_j (t) * w_\nds\right) (x)$.
  The solution $\rho_j$ is an \emph{entropy solution} if for any
  $\phi \in \Cc1([0,T[ \times \reali; \reali_+)$, for all
  $\kappa \in \reali$ and for all $j=1,\ldots, M$
  \begin{multline*}
    \int_0^T \int_\reali \left( \modulo{\rho_j- \kappa} \partial_t
      \phi + \modulo{\rho_j- \kappa} V_j \, \partial_x \phi \right)
    \d{x} \d{t} + \int_\reali \modulo{\rho_{o,j}- \kappa} \phi(0,x)
    \d{x}
    \\
    \geq \int_0^T \int_\reali \sgn(\rho_j-\kappa) \left( S_j(\rho_j,
      \rho_{j+1}, R_j, R_{j+1}) -S_{j-1}(\rho_{j-1}, \rho_j, R_{j-1},
      R_j) + \kappa \, \partial_x V_j \right) \phi \d{x} \d{t}.
  \end{multline*}
\end{definition}
\bigskip

In the following, whenever we refer to the \emph{modified Algorithm}
we mean Algorithm~\ref{alg:1} with~\eqref{eq:numFlux}
and~\eqref{eq:convStep} substituted by~\eqref{eq:numFluxNL}
and~\eqref{eq:convStepNL}, respectively. All the approximate solutions
appearing in the results below are constructed via this modified
Algorithm.

\begin{lemma}
  \label{lem:invNL}
  Let $\brho_o \in \L\infty (\reali; [0,1]^M)$. Assume that the CFL
  condition~\eqref{eq:CFL} holds.  Then, for all $t>0$ and
  $x \in \reali$, the piece-wise constant approximate solution
  $\brho_\Delta$ constructed through the modified Algorithm attains
  value in the set $[0,1]^M$, i.e.
  \begin{displaymath}
    0 \leq \rho_{j, \Delta} (t,x) \leq 1
    \quad \mbox{ for all } j = 1, \ldots, M.
  \end{displaymath}
\end{lemma}
\begin{proof}
  Since the CFL condition~\eqref{eq:CFL} is more restrictive than that
  necessary for the convergence of the Godunov type scheme,
  see~\cite[Theorem 3.1]{friedrich2018godunov}, the convective
  step~\eqref{eq:convStep} still preserves the invariance of the set
  $[0,1]^M$ and the rest of the proof of Lemma~\ref{lem:invariance}
  can be applied. 
\end{proof}

\noindent Lemma~\ref{lem:cons} still holds, since the modified
Algorithm preserves the $\L1$-norm.

\begin{proposition}[\bf{$\BV$ estimate in space}]
  \label{prop:BVspaceNL}
  Let $\brho_o \in (\L1 \cap \BV) (\reali; [0,1]^M)$.
  Assume that the CFL condition~\eqref{eq:CFL} holds. Then, for
  $n=0,\dots, N_T-1$ the following estimate holds
  \begin{equation}
    \label{eq:BVspaceNL}
    \sum_{j=1}^M \sum_{k\in\interi}
    \modulo{\rho_{j,k+1}^{n} - \rho_{j,k}^n}
    \leq
    e^{t^n\left(\, 8 \, \mathcal{K}+w_\ndt (0) \mathcal{V}\right)}
    \sum_{j=1}^M \tv(\rho_{j}^0).
  \end{equation}
\end{proposition}

\begin{proof}
  The proof of Proposition~\ref{prop:BVspace} can be easily adapted.
  We just have to replace estimate~\eqref{eq:TVDstep}, involving the
  convective step, since the scheme with the new numerical
  flux~\eqref{eq:numFluxNL} is not total variation diminishing.
  Following~\cite[Theorem 3.2]{friedrich2018godunov} we obtain
  \begin{align*}
    \sum_{j=1}^M \sum_{k\in\interi}
    \modulo{\rho_{j, k+1}^{n+1} - \rho_{j, k}^{n+1}}
    \leq \
    &
      (1 + 8 \, \dt \, \mathcal{K})
      (1 + \dt \, w_\ndt (0) \, \mathcal{V})
      \sum_{j=1}^M\sum_{k\in\interi}
      \modulo{\rho_{j, k+1}^{n} - \rho_{j, k}^{n}}
    \\
    \leq \
    &
      e^{\dt \left(8 \, \mathcal{K} + w_\ndt(0) \, \mathcal{V}\right)}
      \sum_{j=1}^M\sum_{k\in\interi}
      \modulo{\rho_{j, k+1}^{n}-\rho_{j, k}^{n}},
  \end{align*}
  which applied recursively yields the thesis.
\end{proof}

\begin{proposition}
  \label{prop:LipTimeNL}
  Let $\brho_o \in (\L1 \cap \BV) (\reali; [0,1]^M)$.
  Assume that the CFL condition~\eqref{eq:CFL} holds.
  Then, for $n=0,\dots, N_T-1$,
  \begin{equation}
    \Delta x \sum_{j=1}^M \sum_{k\in\interi} \modulo{\rho_{j,k}^{n+1}-\rho_{j,k}^n}
    \leq
    2 \, \dt \left(
      2 \, V_{\max}\, \norma{\brho_{o}}_{\L1 (\reali)}
      +
      \mathcal{V}\,
      e^{t^n (8 \, \mathcal{K} + w_\ndt (0) \, \mathcal{V})}
      \sum_{j=1}^M \tv(\rho_j^0)\right),
  \end{equation}
  with $\mathcal{K}$ as in~\eqref{eq:lipschitzconstant}, $\mathcal{V}$
  as in~\eqref{eq:VC1} and $V_{\max}$ as in~\eqref{eq:vmax}.
\end{proposition}

\begin{proof}
  Observe that
  \begin{displaymath}
    \modulo{\rho^{n+1}_{j,k} - \rho^n_{j,k}}
    \leq
    \modulo{\rho^{n+1}_{j,k} - \rho^{n+1/2}_{j,k}}
    +
    \modulo{\rho^{n+1/2}_{j,k} - \rho^n_{j,k}}.
  \end{displaymath}
  We then estimate each term on the right hand side separately.

  By the relaxation step~\eqref{eq:relaxStep} and the
  bound~\eqref{eq:boundSource} we have
  \begin{align*}
    \modulo{\rho^{n+1}_{j,k} - \rho^{n+1/2}_{j,k}}
    = \
    & \dt \modulo{
      S_{j-1}\left( \rho_{j-1,k}^{n+1/2}, \rho_{j,k}^{n+1/2},
      R_{j-1,k}^{n+1/2}, R_{j,k}^{n+1/2}\right)
      -
      S_{j}\left( \rho_{j,k}^{n+1/2}, \rho_{j+1,k}^{n+1/2},
      R_{j,k}^{n+1/2}, R_{j+1,k}^{n+1/2}\right)
      }
    \\
    \leq \
    & \dt \, V_{\max} \left(
      \rho_{j-1,k}^{n+1/2} + 2 \, \rho_{j,k}^{n+1/2} + \rho_{j+1,k}^{n+1/2}\right).
  \end{align*}
  Therefore, thanks to Lemma~\ref{lem:cons}
  \begin{equation}
    \label{eq:15}
    \dx \sum_{j=1}^M\sum_{k \in \interi}
    \modulo{\rho^{n+1}_{j,k} - \rho^{n+1/2}_{j,k}}
    \leq
    \dt \, V_{\max} \sum_{j=1}^M 4 \norma{\rho_j^{n+1/2}}_{\L1 (\reali)}
    = 4 \, \dt \, V_{\max}  \sum_{j=1}^M  \norma{\rho_{j,o}}_{\L1 (\reali)}.
  \end{equation}

  Exploiting the modified convective step~\eqref{eq:convStepNL}, since
  the numerical flux defined in~\eqref{eq:numFluxNL} is Lipschitz
  continuous in both variables with Lipschitz constant
  $\mathcal{V}$~\eqref{eq:VC1}, we have
  \begin{align*}
    \modulo{\rho^{n+1/2}_{j,k} - \rho^n_{j,k}}
    = \
    & \lambda \modulo{F_j\left(\rho_{j,k}^n, R_{j,k}^{\ndt,n}\right)
      -
      F_j\left(\rho_{j,k-1}^n, R_{j,k-1}^{\ndt,n}\right)}
    \\
    \leq \
    & \lambda \, \mathcal{V} \left(
      \modulo{\rho_{j,k}^n- \rho_{j,k-1}^n}
      +
      \modulo{R_{j,k}^{\ndt,n} - R_{j,k-1}^{\ndt,n}}
      \right).
  \end{align*}
  Hence, using also~\eqref{eq:Rtv} and the total variation bound
  provided by Proposition~\ref{prop:BVspaceNL}, we get
  \begin{equation}
    \label{eq:16}
    \dx \sum_{j=1}^M \sum_{k \in \interi}
    \modulo{\rho^{n+1/2}_{j,k} - \rho^n_{j,k}}
    \leq 2 \, \dt \,  \mathcal{V}
    \sum_{j=1}^M \sum_{k \in \interi}
    \modulo{\rho^{n}_{j,k} - \rho^n_{j,k-1}}
    \leq
    2 \, \dt \,  \mathcal{V}
    e^{t^n (8 \, \mathcal{K} + w_\ndt (0) \, \mathcal{V})}
    \sum_{j=1}^M \tv(\rho_j^0).
  \end{equation}

  Collecting together~\eqref{eq:15} and~\eqref{eq:16} yields the
  thesis
  \begin{displaymath}
    \dx \sum_{j=1}^M \sum_{k \in \interi}
    \modulo{\rho^{n+1}_{j,k} - \rho^n_{j,k}}
    \leq
    2 \, \dt \left(
      2 \, V_{\max}\, \norma{\brho_{o}}_{\L1 (\reali)}
      +
      \mathcal{V}\,
      e^{t^n (8 \, \mathcal{K} + w_\ndt (0) \, \mathcal{V})}
      \sum_{j=1}^M \tv(\rho_j^0)\right).
  \end{displaymath}
\end{proof}

Proceeding as in Corollary~\ref{cor:BVspaceTime}, combining the
results of Proposition~\ref{prop:BVspaceNL} and
Proposition~\ref{prop:LipTimeNL} we obtain a $\BV$ estimate in space
and time.

Analogously to Section~\ref{sec:entropyIneq}, a discrete entropy
inequality could be derived also in the case of nonlocal flux,
see~\cite[Proposition~2.8]{AmorimColomboTeixeira}. Indeed,
combining~\cite[Theorem~3.4]{friedrich2018godunov} for the nonlocal
flux and Lemma~\ref{lemma:discreteEntropy} for the treatment of the
source terms we get the following result.
\begin{lemma}
  Let $\brho \in (\L1 \cap \BV) (\reali; [0,1]^M)$. Let the CFL
  condition~\eqref{eq:CFL} hold. Then the approximate solution
  $\brho_\Delta$ constructed through the modified Algorithm satisfies
  the following discrete entropy inequality: for all $j=1, \ldots, M$,
  for $k \in \interi$, for $n=0, \ldots, N_T -1$ and for any
  $c \in [0,1]$
  \begin{align*}
    \modulo{\rho_{j,k}^{n+1} - c}
    -
    \modulo{\rho_{j,k}^{n} - c}
    + \lambda \left(
    \mathscr{F}_j^c\left(\rho_{j,k}^n\right)
    -
    \mathscr{F}_j^c\left(\rho_{j,k-1}^n\right)
    \right)
    \\
    - \dt \, \sgn\left(\rho_{j,k}^{n+1} - c\right)
    \left(
    S_{j-1}\left(\rho_{j-1,k}^{n+1/2}, \rho_{j,k}^{n+1/2}, R_{j-1,k}^{\nds,n+1/2}, R_{j,k}^{\nds,n+1/2}\right)\right.
    \\
    \left.
    -
    S_{j}\left(\rho_{j,k}^{n+1/2}, \rho_{j+1,k}^{n+1/2}, R_{j,k}^{\nds,n+1/2}, R_{j+1,k}^{\nds, n+1/2}\right)
    \right)\\
    + \lambda
    \sgn\left(\rho_{j,k}^{n+1} - c\right) c \,
    \left(v_j\left(R_{j,k+1}^{\ndt,n}\right)
    -
    v_j\left(R_{j,k}^{\ndt,n}\right) \right)
 & \leq 0,
  \end{align*}
  where $R_{j,k}^{\nds, n+1/2}$ and $R_{j,k}^{\ndt,n}$ are defined
  accordingly to~\eqref{eq:disConvOp} and
  \begin{align*}
    \mathscr{F}_j^c(u) =\
    & G_j(u \vee c) - G_j(u \wedge c),
    & \mbox{with }
      G_j(\rho_{j,k}^n) = \
    & \rho_{j,k}^n \, v_j\left(R_{j,k}^{\ndt,n}\right).
  \end{align*}
\end{lemma}

The results described in Section~\ref{sec:conv} hold analogously for
the modified Algorithm, given the bounds obtained in the present
section: this ensures the existence of solutions
to~\eqref{eq:nonlocaltransport}.

Uniqueness of solution follows from the Lipschitz continuous
dependence of the solution on the initial data. Differently from
Theorem~\ref{thm:contract}, in the case of \emph{nonlocal} flux
function the solution is not contractive in $\L1$.

\begin{theorem}
  \label{thm:LipDataNL}
  Let $\brho$ and $\bpi$ be two solutions to
  problem~\eqref{eq:nonlocaltransport} in the sense of
  Definition~\ref{def:solNL}, with initial data
  $\brho_o, \, \bpi_o \in (\L1 \cap \BV)(\reali; [0,1]^M)$
  respectively.  Assume $v\in \C2([0,1],\reali)$. Then, for
  a.e.~$t \in [0,T]$,
  \begin{displaymath}
    \sum_{j=1}^M \norma{\rho_j (t) - \pi_j (t)}_{\L1 (\reali)}
    \leq
    e^{\mathcal{C} \, t}\sum_{j=1}^M \norma{\rho_{j,o} - \pi_{j,o}}_{\L1 (\reali)},
  \end{displaymath}
  with $\mathcal{C}$ defined as in~\eqref{eq:22}.
\end{theorem}

\begin{proof}
  The doubling of variables technique~\cite{Kruzkov} allows to get the
  following estimate, see~\cite[Lemma~4]{NARWA2019} for the treatment
  of the nonlocal flux, while the source terms are treated similarly
  to~\cite[Theorem~3.3]{HoldenRisebro}: any $j=1, \ldots, M$,
  \begin{align}
    \nonumber
    \int_\reali \modulo{\rho_j (\tau,x) - \pi_j (\tau,x)} \d{x}
    \leq \
    &
      \int_\reali \modulo{\rho_j (0,x) - \pi_j (0,x)} \d{x}
    \\
    \label{eq:19}
    & + \int_0^\tau \int_\reali
      \modulo{ \mathcal{S}(\brho, \boldsymbol{R}^\nds, j)
      - \mathcal{S}(\bpi, \boldsymbol{P}^\nds, j)}
      \d{x} \d{t}
    \\
    \label{eq:20}
    &+ \int_0^\tau \int_\reali
      \modulo{v_j (R_j^\ndt) - v_j(P_j^\ndt)} \modulo{\partial_x \rho_j(t,x)}
      \d{x} \d{t}
    \\
    \label{eq:21}
    & + \int_0^\tau \int_\reali
      \modulo{\partial_x v_j (R_j^\ndt) - \partial_x v_j(P_j^\ndt)}
      \modulo{\rho_j(t,x)}
      \d{x} \d{t},
  \end{align}
  where for the source terms we use the notation introduced in the
  proof of Theorem~\ref{thm:contract}, while for the kernel we refer
  to~\eqref{eq:Rj}, emphasizing which kernel, $w_\ndt$ or $w_\nds$, is
  used. We remark that $\partial_x \rho$ should be understood in the
  sense of measures.

  To bound the term in~\eqref{eq:19}, exploit the Lipschitz continuity
  of the map $S_j$~\eqref{eq:source} in the source term:
  \begin{align*}
    \int_0^\tau \int_\reali
    \modulo{ \mathcal{S}(\brho, \boldsymbol{R}^\nds, j)
    - \mathcal{S}(\bpi, \boldsymbol{P}^\nds, j)}
    \d{x} \d{t}
    \leq \
    & \mathcal{K}  \int_0^\tau
      \left(
      \norma{\rho_{j-1} (t) - \pi_{j-1}(t)}_{\L1 (\reali)}
      \right.
    \\
    & \quad + 2 \, \norma{\rho_{j} (t) - \pi_{j}(t)}_{\L1 (\reali)}
      + \norma{\rho_{j+1} (t) - \pi_{j+1}(t)}_{\L1 (\reali)}
    \\
    & \quad
      + \norma{R_{j-1}^\nds (t) - P_{j-1}^\nds(t)}_{\L1 (\reali)}
      + 2 \, \norma{R_{j}^\nds (t) - P_{j}^\nds(t)}_{\L1 (\reali)}
    \\
    & \quad
      \left.+ \norma{R_{j+1}^\nds (t) - P_{j+1}^\nds(t)}_{\L1 (\reali)}
      \right) \d{t}.
  \end{align*}
  Observe that for each $j=1, \ldots, M$
  \begin{displaymath}
    \norma{R_{j}^\nds (t) - P_{j}^\nds(t)}_{\L1 (\reali)}
    \leq
    \norma{\rho_{j} (t) - \pi_{j}(t)}_{\L1 (\reali)},
  \end{displaymath}
  since $\int_\reali w_\nds =1$. Therefore
  \begin{equation}
    \label{eq:19ok}
    \sum_{j=1}^M [\eqref{eq:19}]
    \leq
    4 \, \mathcal{K}
    \sum_{j=1}^M
    \int_0^\tau \norma{\rho_{j} (t) - \pi_{j}(t)}_{\L1 (\reali)} \d{t}.
  \end{equation}

  Concerning~\eqref{eq:20}, note that
  \begin{displaymath}
    \modulo{v_j (R_j^\ndt) - v_j(P_j^\ndt)}
    \leq
    w_\ndt (0)  \, \norma{v'_j}_{\L\infty ([0,1]; \reali)}
    \norma{\rho_{j} (t) - \pi_{j}(t)}_{\L1 (\reali)},
  \end{displaymath}
  thus
  \begin{equation}
    \label{eq:20ok}
    \begin{aligned}
      \sum_{j=1}^M [\eqref{eq:20}] \leq \ & w_\ndt (0) \, V'_{\max}
      \sum_{j=1}^M \int_0^\tau \tv\left(\rho_j (t)\right)
      \norma{\rho_{j} (t) - \pi_{j}(t)}_{\L1 (\reali)} \d{t}
      \\
      \leq \ & w_\ndt (0) \, V'_{\max} \left(\sum_{j=1}^M \sup_{t \in
          [0,\tau]} \tv\left(\rho_j (t)\right)\right) \left(
        \sum_{j=1}^M \int_0^\tau \norma{\rho_{j} (t) -
          \pi_{j}(t)}_{\L1 (\reali)} \d{t}\right).
    \end{aligned}
  \end{equation}

  Pass now to~\eqref{eq:21}. Observe first that
  \begin{align*}
    \modulo{\partial_x R_j^\ndt (t,x) }= \
    & \modulo{\partial_x \left(\rho_j(t) * w_\ndt\right) (x)}
    \\
    = \
    & \modulo{- \int_x^{x + \ndt} \rho_j(t,y)\, w'_{\ndt}(x-y) \d{y}
      + \rho_j (t, x+\ndt) \, w_{\ndt} (\ndt)
      - \rho_j (t, x) \, w_{\ndt} (0)}
    \\
    \leq \
    & \modulo{\int_0^\ndt \rho_j(t, u+x)\, w'_{\ndt}(u) \d{u}}
      + \norma{\rho_j (t)}_{\L\infty (\reali)} \left(
      w_{\ndt} (\ndt)+ w_{\ndt} (0)
      \right)
    \\
    \leq \
    & \norma{\rho_j (t)}_{\L\infty (\reali)} \left(
      \int_0^\ndt \modulo{w'_{\ndt}(u)} \d{u}
      + w_{\ndt} (\ndt)+ w_{\ndt} (0)
      \right)
    \\
    = \
    & \norma{\rho_j (t)}_{\L\infty (\reali)} \left(
      - \int_0^\ndt w'_{\ndt}(u) \d{u}
      + w_{\ndt} (\ndt)+ w_{\ndt} (0)
      \right)
    \\
    = \
    & 2 \, w_{\ndt} (0) \,  \norma{\rho_j (t)}_{\L\infty (\reali)},
  \end{align*}
  since the kernel $w_{\ndt}$ is such that $w'_{\ndt} \leq 0$. Hence,
  \begin{align*}
    &\modulo{\partial_x v_j (R_j^\ndt) - \partial_x v_j(P_j^\ndt)}
    \\
    \leq \
    & \modulo{v'_j(R_j^\ndt) - v'_j(P_j)}\modulo{\partial_xR_j^\ndt}
      +
      \modulo{v'_j(P_j)}\modulo{\partial_x R_j^\ndt - \partial_x P_j^\ndt}
    \\
    \leq \
    & \norma{v''_j}_{\L\infty ([0,1])} \modulo{R_j^\ndt - P_j^\ndt}
      \, 2 \, w_{\ndt} (0) \,  \norma{\rho_j (t)}_{\L\infty (\reali)}
    \\
    & +
      \norma{v'_j}_{\L\infty ([0,1])}
      \modulo{\int_x^{x + \ndt} \!\!\!\left(\pi_j - \rho_j \right) \!(t,y)
      \, w'_{\ndt}(x-y) \d{y}
      + \left(\rho_j - \pi_j\right) \!(t, x+\ndt) \, w_{\ndt} (\ndt)
      - \left(\rho_j - \pi_j\right) \!(t, x) \, w_{\ndt} (0)}
    \\
    \leq \
    & \left(2 \, \left(w_\ndt(0)\right)^2 \, \norma{v''_j}_{\L\infty ([0,1])}
      +
      \norma{v'_j}_{\L\infty ([0,1])}
      \norma{w'_\ndt}_{\L\infty ([0, \ndt])}
      \right)\norma{\rho_j(t)- \pi_j(t)}_{\L1 (\reali)}
    \\
    & + w_\ndt(0) \, \norma{v'_j}_{\L\infty ([0,1])}
      \left(
      \modulo{\left(\rho_j - \pi_j\right)} (t, x+\ndt)
      +
      \modulo{\left(\rho_j - \pi_j\right)} (t, x)
      \right).
  \end{align*}
  Therefore, since the total mass is conserved and
  $\rho_j(t,x) \in [0,1]$ for all $j=1,\ldots, M$, $t \in [0,\tau]$
  and $x \in \reali$ by Lemma~\ref{lem:invNL},
  \begin{equation}
    \label{eq:21ok}
    \begin{aligned}
      \sum_{j=1}^M[\eqref{eq:21}] \leq \ & \left(2 \,
        \left(w_\ndt(0)\right)^2 \, V''_{\max} + V'_{\max}
        \norma{w'_\ndt}_{\L\infty ([0, \ndt])}
      \right)\norma{\brho_o}_{\L1 (\reali)} \sum_{j=1}^M\int_0^\tau
      \norma{\rho_j(t)- \pi_j(t)}_{\L1 (\reali)} \d{t}
      \\
      & + 2 \, w_\ndt(0) \, V'_{\max} \sum_{j=1}^M\int_0^\tau
      \norma{\rho_j(t)- \pi_j(t)}_{\L1 (\reali)} \d{t},
    \end{aligned}
  \end{equation}
  where
  \begin{displaymath}
    V''_{\max} = \norma{\bv''}_{\C0([0,1]; \reali^M)}
    = \max_{j=1, \ldots, M} \norma{v''_j}_{\L\infty([0,1]; \reali)}.
  \end{displaymath}
  Collecting together~\eqref{eq:19ok},\eqref{eq:20ok}
  and~\eqref{eq:21ok} we obtain
  \begin{equation}
    \label{eq:almostThere}
    \int_\reali \modulo{\rho_j (\tau,x) - \pi_j (\tau,x)} \d{x}
    \leq
    \int_\reali \modulo{\rho_j (0,x) - \pi_j (0,x)} \d{x}
    + \mathcal{C} \,
    \sum_{j=1}^M\int_0^\tau
    \norma{\rho_j(t)- \pi_j(t)}_{\L1 (\reali)} \d{t},
  \end{equation}
  where
  \begin{equation}
    \label{eq:22}
    \begin{aligned}
      \mathcal{C} = \ & 4 \, \mathcal{K} + 2 \, w_\ndt(0) \, V'_{\max}
      + w_{\ndt}(0) \, V'_{\max} \left(\sum_{j=1}^M \sup_{t \in
          [0,\tau]} \tv\left(\rho_j(t)\right)\right)
      \\
      & + \left(2 \, \left(w_\ndt(0)\right)^2 \, V''_{\max} +
        V'_{\max} \norma{w'_\ndt}_{\L\infty ([0, \ndt])}
      \right)\norma{\brho_o}_{\L1 (\reali)}.
    \end{aligned}
  \end{equation}
  An application of Gronwall Lemma to~\eqref{eq:almostThere} yields
  the desired result.
\end{proof}

The following theorem, analogous to Theorem~\ref{thm:main}, collects
the main result on problem~\eqref{eq:nonlocaltransport}, as well as
some \emph{a priori} estimates on its solution.
\begin{theorem}
  \label{thm:mainNL}
  Let $\brho_o \in (\L1 \cap \BV)(\reali; [0,1]^M)$. Assume
  $v\in \C2([0,1],\reali)$. Then, for all $T>0$,
  problem~\eqref{eq:nonlocaltransport} has a unique solution
  $\brho \in \C0([0,T]; \L1(\reali; [0,1]^M))$ in the sense of
  Definition~\ref{def:solNL}. Moreover, the following estimates hold:
  for any $t \in [0,T]$
  \begin{align*}
    \norma{\brho(t)}_{\L1(\reali)} = \
    & \sum_{j=1}^M \norma{\rho_j(t)}_{\L1(\reali)}=
      \norma{\brho_o}_{\L1 (\reali)},
    \\
    \mbox{for }  j = 1, \ldots, M:
    & \quad  0\leq \rho_j(t,x)\leq 1,
    \\
    \sum_{j=1}^M\tv\left(\rho_j(t)\right) \leq \
    & e^{t \left(8 \, \mathcal{K} + w_{\ndt} (0) \, \mathcal{V}\right)}
      \sum_{j=1}^M \tv(\rho_{j,o}).
  \end{align*}
\end{theorem}

\section{Numerical experiments}
\label{sec:numEx}

We present now some numerical examples. We divide this section in two
parts: in the first part, we discuss an example with local flux and
\emph{nonlocal} source, as in~\eqref{eq:model}, while in the second
part we focus on \emph{nonlocal} flux and source, as
in~\eqref{eq:nonlocaltransport}. For simplicity, we restrict ourselves
to only two lanes, i.e.~$M=2$, and scaling parameter $K=1$.

\subsection{Local flux and nonlocal source: results for model
  \texorpdfstring{\eqref{eq:model}}{}}
\label{sec:localFlux}

The first example is inspired by~\cite{HoldenRisebro}. We consider the
following velocity functions
\begin{align*}
  v_1(\rho)=1.5(1-\rho)\quad\text{and}\quad v_2(\rho)=2.5(1-\rho)
\end{align*}
and the initial data
\begin{align*}
  \rho_{1,o}(x)=\rho_{2,o}(x)=\sin(\pi x/2)^2.
\end{align*}
Figure~\ref{fig:N2_local} displays the density profiles with three
different source terms at times $T=0.75$ (left column) and $T=1.5$
(right column). We use the nonlocal source term~\eqref{eq:source} with
both~\eqref{eq:forward} and~\eqref{eq:backfor}, and constant kernel,
namely~$w_\nds(x)=1/\nds$ for ~\eqref{eq:forward} with $\nds=0.5$, and
$w_\nds(x)=1/(2\nds)$ for~\eqref{eq:backfor} with
$\nds=0.25$. Therefore, both nonlocal models have an interaction range
equal to $0.5$. To emphasize the influence of the nonlocality, we
include also a local version of the source term~\eqref{eq:source} with
\begin{equation}\label{eq:local_source}
  R_j=\rho_j(t,x).
\end{equation}
Notice that such a local version differs from that used
in~\cite{HoldenRisebro}: Here the lane changing rate is also
proportional to the density in the receiving lane.  In the
simulations, we consider $\Delta x=0.01$ and $\dt$ given by an
adaptive version of the CFL condition~\eqref{eq:CFL}, where
$\mathcal{V}$ is computed at each time step using finite differences
for the derivative of $v_j$.

As can be seen in Figure~\ref{fig:N2_local}, different source terms
give rise to slight differences for small times, but as the time grows
they become more significant. The nonlocal source terms transport mass
faster from the slower lane (lane 1) to the faster one in comparison
to the local model. Interestingly, there is only a slight difference
between the two nonlocal models with forward looking
kernel~\eqref{eq:forward} and back- and forward looking
kernel~\eqref{eq:backfor}. This is probably due to the fact that the
nonlocal range and the form of the kernel are equal.

\setlength\fwidth{0.85\textwidth}
\begin{figure}[htb]
  \centering \tikzsetnextfilename{2Lanes_example_local_density}
  \input{Plots/2Lanes_example_local_density}
  \caption{Density profiles on each lane at $T=0.75$ (left column) and
    $T=1.5$ (right column). The source term~\eqref{eq:source} is
    computed with~\eqref{eq:local_source} (top row),
   ~\eqref{eq:forward} with $w_\nds(x)=1/\nds$ and $\nds=0.5$ (middle
    row) and~\eqref{eq:backfor} with $w_\nds(x)=1/(2\nds)$ and
    $\nds =0.25$ (bottom row).}
  \label{fig:N2_local}
\end{figure}
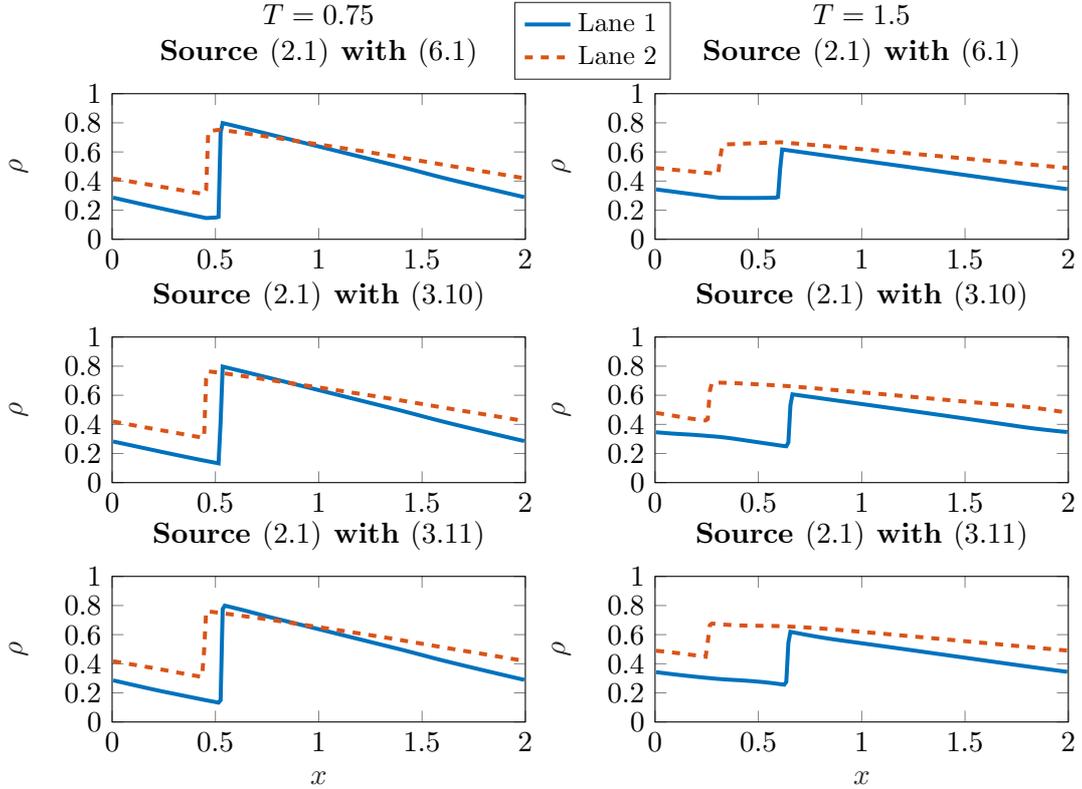

These observations are also supported by the evolution of the
$\L1$-norm over time, see Figure~\ref{fig:N2_L1norm}. We can see that
the nonlocal models transport the density faster from lane 1 to lane
2. In addition, the model with the forward looking
kernel~\eqref{eq:forward} is a bit faster than the one
with~\eqref{eq:backfor}.

\setlength\fwidth{0.75\textwidth}
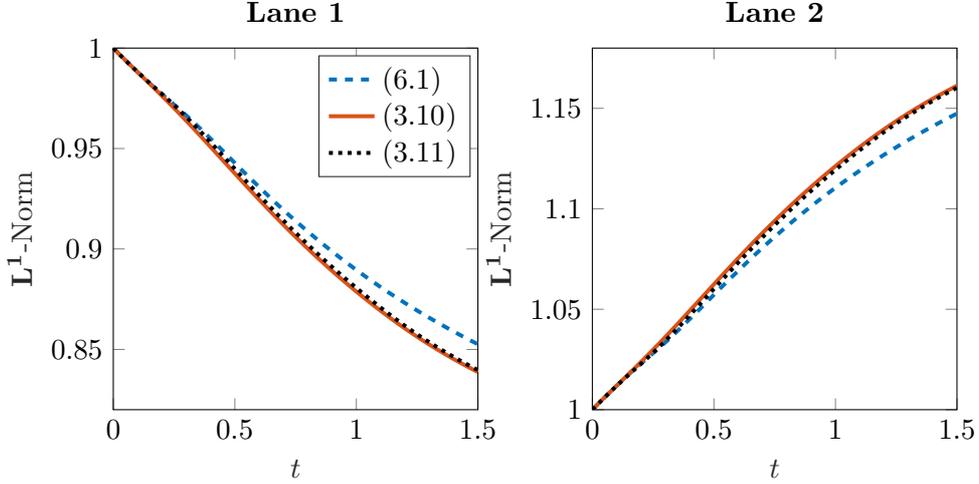
\begin{figure}
  \centering \tikzsetnextfilename{2Lanes_example_local_L1norm}
  \input{Plots/2Lanes_example_local_L1norm}
  \caption{Evolution of $\L1$-norm over time for the different source
    terms: Lane 1 (left) and Lane 2 (right). The dashed blue line
    represent the source term in~\eqref{eq:source}
    with~\eqref{eq:local_source}, the continuous orange line refers
    to~\eqref{eq:forward} and the dotted black line
    to~\eqref{eq:backfor}.}
  \label{fig:N2_L1norm}
\end{figure}

Finally, we consider the nonlocal models with both~\eqref{eq:forward}
and~\eqref{eq:backfor} and let $\nds$ tend to zero. In particular, we
choose the constant kernel, in the form $w_\nds(x)=1/\nds$
for~\eqref{eq:forward} and in the form $w_\nds(x)=1/(2\nds)$
for~\eqref{eq:backfor}, with
$\nds \in \{0.64,\ 0.32,\ 0.16,\ 0.08,\ 0.04,\ 0.02\}$. Note that for
simplicity, $\nds$ is the same for both models, even though one time
the non local range is in the interval $[x,x+\nds]$, one time in
$[x-\nds,x+\nds]$. Figure~\ref{fig:N2_etatozero} displays the lanes
separately to better appreciate the convergence. The convergence
against the source with~\eqref{eq:local_source} seems to hold for both
models. Moreover, when focusing only on the second lane, also the two
nonlocal models display some differences,
e.g.~compare~\eqref{eq:forward} with $\nds=0.64$ and
\eqref{eq:backfor} with $\nds=0.32$ (parameters are chosen so that the interaction ranges have the same width in both models).

\setlength\fwidth{0.9\textwidth}
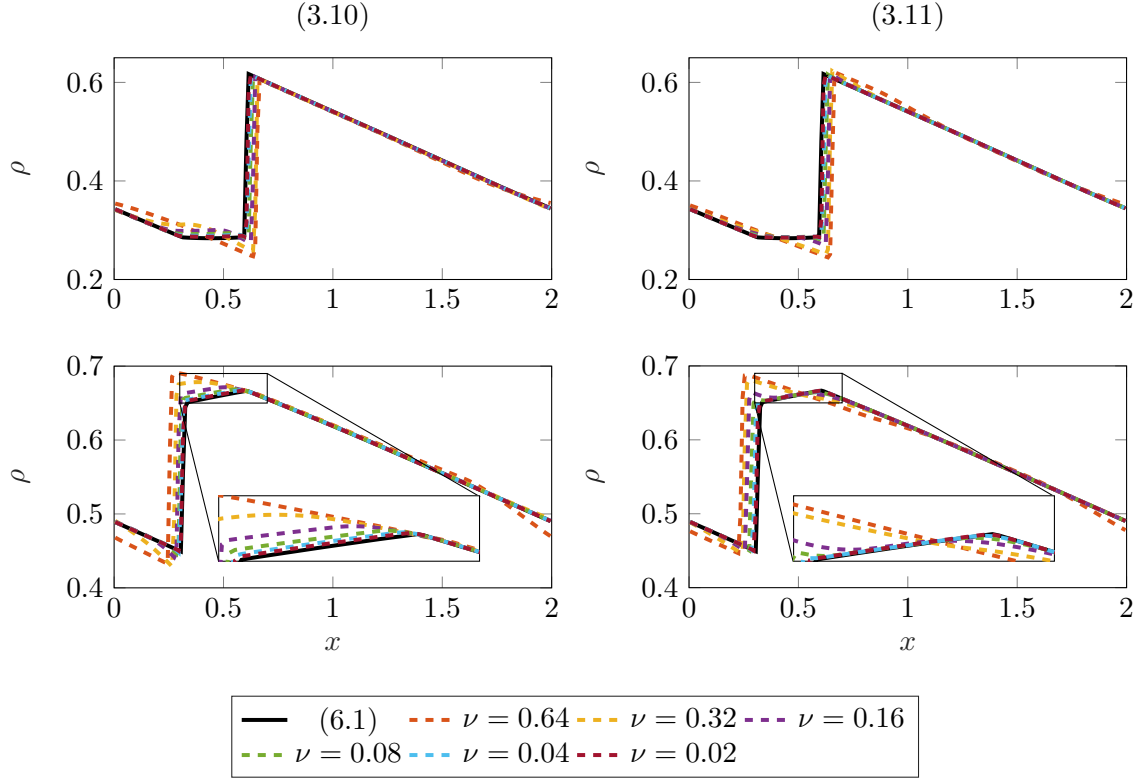
\begin{figure}[htb]
  \centering \tikzsetnextfilename{2Lanes_eta_to_zero_v4}
  \input{Plots/2Lanes_eta_to_zero_v4}
  \caption{Density profiles for the nonlocal source
    with~\eqref{eq:forward} (left column) and~\eqref{eq:backfor}
    (right column) with different nonlocal ranges $\nds$, for lane 1
    (top row) and lane 2 (bottom row). For lane 2 the zooms are into
    the spatial domain $x\in[0.3,0.7]$ and $\rho\in[0.65,0.69]$.}
  \label{fig:N2_etatozero}
\end{figure}

Table~\ref{tab:L1errors} presents the $\L1$-errors between the local
source term~\eqref{eq:local_source} and both the nonlocal source
terms: the data support the convergence against the local solution as
$\nds \to 0$. Interestingly, the source term with back- and forward
looking kernel~\eqref{eq:backfor} has smaller error terms on the first
lane and only slightly larger errors on the second lane, even tough
the nonlocal range is twice that of the forward looking
kernel~\eqref{eq:forward}.

\begin{table}[h!]
  \centering
  \begin{tabular}{c|c c | c c}
    Source term&\multicolumn{2}{c|}{\eqref{eq:forward}} &\multicolumn{2}{c}{\eqref{eq:backfor}}\\
    $\nds$ &  Lane 1 & Lane 2 & Lane 1 & Lane 2\\
    \hline
    0.64 &  0.0311   &    0.0313     &  0.0330   &    0.0310\\
    0.32 &  0.0239   &    0.0167     &  0.0208   &    0.0198\\
    0.16 &  0.0159   &    0.0089     &  0.0131   &    0.0120\\
    0.08 &  0.0095   &    0.0049     &  0.0078   &    0.0066\\
    0.04 &  0.0054   &    0.0026     &  0.0045   &    0.0035\\
    0.02 &  0.0030   &    0.0013    &   0.0023  &    0.0016
  \end{tabular}
  \caption{$\L1$-errors computed between the solutions with local
    source~\eqref{eq:local_source} and nonlocal source
    terms~\eqref{eq:forward} or~\eqref{eq:backfor}.}
  \label{tab:L1errors}
\end{table}

\subsection{Nonlocal flux and nonlocal source: results for model
  \texorpdfstring{\eqref{eq:nonlocaltransport}}{}}

In the following we consider model~\eqref{eq:nonlocaltransport}
including a nonlocality in the source with parameter $\nds$ and a
nonlocality in the flux with parameter $\ndt$. We focus on the
following two lanes example, inspired by~\cite{BayenKeimer-preprint}:
The velocity function is the same on both lanes
and given by
\begin{equation}
  \label{eq:23}
  v_1(\rho)=v_2(\rho)=1-\rho^2,
\end{equation}
and the initial condition is given by
\begin{equation}
  \label{eq:24}
  \rho_{1,o}(x)=q\left(2\,x-\frac{1}{2}\right)
  \quad\mbox{ and }\quad
  \rho_{2,o}(x)=
  q(x)
\end{equation}
with
\[q(x)=4\, x^2(1-x)^2\,\caratt{(0,1)} (x),\] $\caratt{A}$ being the
characteristic function of the set $A$.

Model~\eqref{eq:nonlocaltransport} fits into the model framework
proposed in~\cite{BayenKeimer-preprint}, if we consider $\nds=\ndt$,
the same kernel functions for the source and the flux and a forward
looking nonlocal term as in~\eqref{eq:forward}. Therefore we consider,
if not stated otherwise, the parameters $\ndt=\nds=0.5$ and the
kernels
\begin{align}
  \label{eq:25}
  w_\ndt(x)=\ & 2\, \frac{\ndt-x}{\ndt^2},
  &
    w_\nds(x)= \ & 2\, \frac{\nds-x}{\nds^2}.
\end{align}
Figure~\ref{fig:Keimer_flux} compares models~\eqref{eq:model}
and~\eqref{eq:nonlocaltransport} and clearly shows the impact of the
nonlocal flux. For both models the same nonlocal
term~\eqref{eq:forward} is used. Because of the nonlocal transport,
the solutions display completely different dynamics, mainly due to the
high nonlocal range.  Indeed, the density does not decrease at the
front part of its support on each lane since the vehicles just behind
the leading ones anticipate the free space ahead, so that the average
density is lower than in the local case.

\setlength\fwidth{0.85\textwidth}
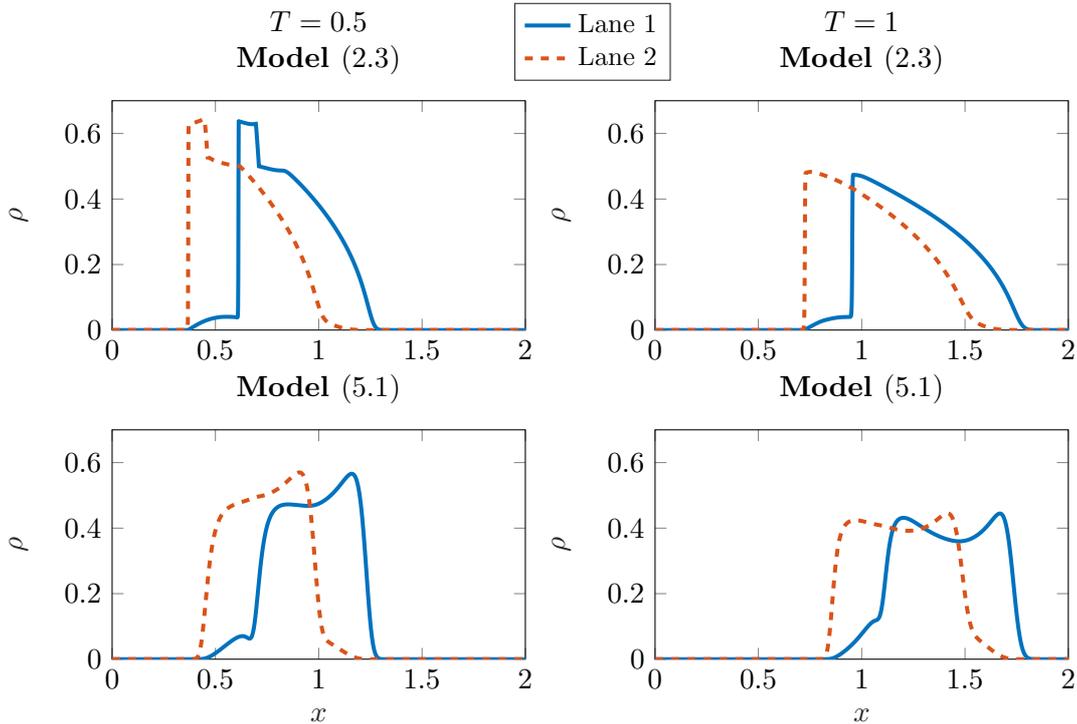
\begin{figure}[h!]
  \centering \tikzsetnextfilename{Keimer_different_flux_05}
  \input{Plots/Keimer_different_flux_05}
  \caption{Density profiles at $T=0.5$ (left) and $T=1$ (right) for
    the local flux model~\eqref{eq:model} (top row) and the nonlocal
    flux model~\eqref{eq:nonlocaltransport} (bottom row), both with
    the nonlocal source term using~\eqref{eq:forward}
    with~\eqref{eq:25}. Velocity functions as in~\eqref{eq:23} and
    initial datum as in~\eqref{eq:24}.}
  \label{fig:Keimer_flux}
\end{figure}

As already mentioned, the model introduced
in~\cite{BayenKeimer-preprint} has the same nonlocal term, i.e.~the
same kernel and nonlocal range, both in the source and in the
flux. On the other hand, the model~\eqref{eq:nonlocaltransport}
presented in this paper has more flexibility since it is able to deal
with different types of nonlocality in the flux and in the source, the
latter being independent of the forward nonlocal term.  Therefore, we
now focus on varying the nonlocality in the source term.

First of all, we observe that the nonlocal range in the flux and in
the source term do not necessarily have to be equal: If a driver wants
to overtake a car and thus starts to accelerate, getting ready to
change lane, he/she might look further ahead when performing a lane
change than if he/she keeps on driving in the same lane. In
Figure~\ref{fig:Keimer_eta}, we display the solutions
to~\eqref{eq:nonlocaltransport} with initial datum~\eqref{eq:24},
velocities~\eqref{eq:23}, kernels~\eqref{eq:25}, $\ndt=0.5$ and
varying the nonlocal range in the source, thus varying the parameter
$\nds$. Due to the initial condition, the main influence of the
different parameters $\nds$ in the source term can be seen at the back
of the support of the density in lane 1: the smaller the range $\nds$,
the smaller the average density (and thus the larger the velocity on
lane 1), the more vehicles move from lane 2 to lane 1.  An analogous
situation happens at the front of the support of the density in lane
2.
To sum up, the greater the nonlocal range $\nds$, the less the effect
of the source term: When $\nds$ is large, cars get a better awareness
of the actual free space ahead so that lane changing may be evaluated
as not necessary.

\setlength\fwidth{0.85\textwidth}
\begin{figure}[h!]
  \centering \tikzsetnextfilename{Keimer_different_eta_05}
  \input{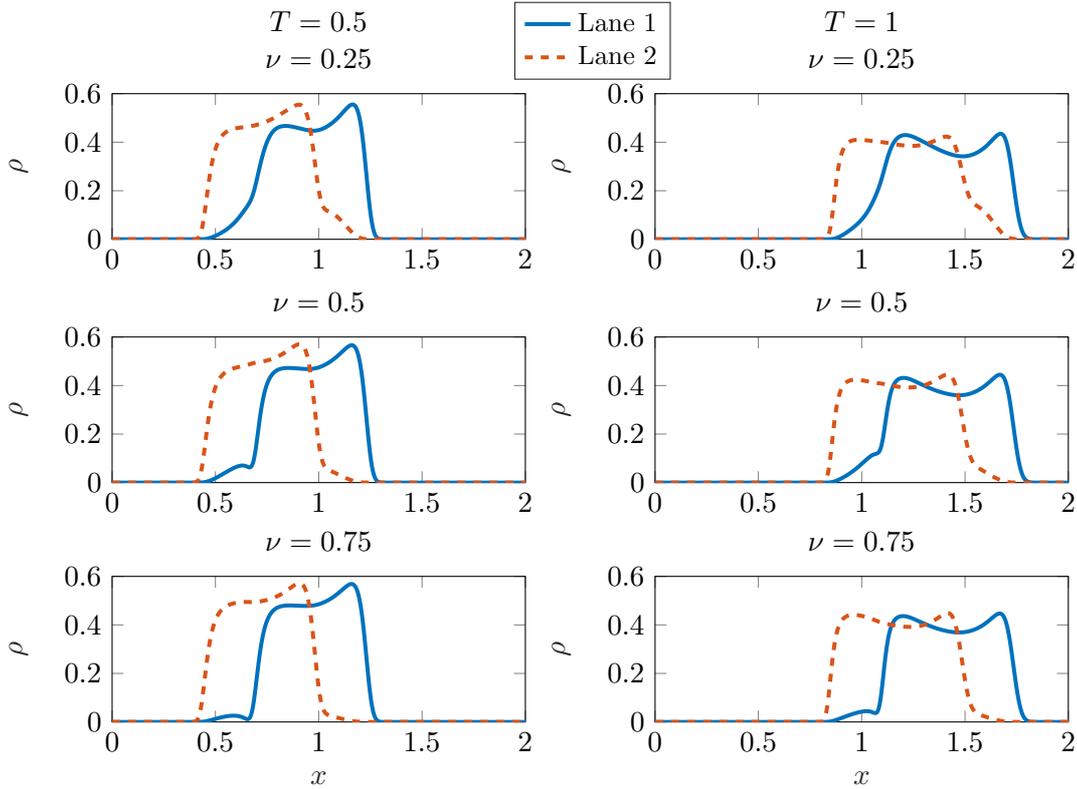}
  \caption{Density profiles at $T=0.5$ (left) and $T=1$ (right) for
    model~\eqref{eq:nonlocaltransport} with forward looking
    kernel~\eqref{eq:forward} with~\eqref{eq:25}, $\ndt=0.5$ and
    $\nds = 0.25 , \, 0.5 , \, 0.75$. Velocity functions as
    in~\eqref{eq:23} and initial datum as in~\eqref{eq:24}.}
  \label{fig:Keimer_eta}
\end{figure}

The second reasonable aspect to keep in mind when performing a lane
change is to take into account also the backward traffic, both in the
present lane and in the target lane. This can be done by considering
model~\eqref{eq:nonlocaltransport} with back- and forward looking
kernel~\eqref{eq:backfor} in the source term. For this example, we
consider a linear symmetric kernel, i.e.
\begin{equation}
  \label{eq:26}
  w_\nds(x)=\frac{\nds-|x|}{\nds^2}.
\end{equation}
Figure~\ref{fig:Keimer_sources} considers the solutions to
model~\eqref{eq:nonlocaltransport} with velocities~\eqref{eq:23},
initial datum~\eqref{eq:24}, $\ndt=0.5$, $w_\ndt$ as in~\eqref{eq:25}
and the following choices of $w_\nds$ and $\nds$:
\begin{enumerate}[label=(\alph*)]
\item\label{item:1} the back- and forward looking
  kernel~\eqref{eq:backfor}--\eqref{eq:26} with $\nds=0.25$, to have
  the same nonlocal influence as in the flux;

\item\label{item:2} the back- and forward looking
  kernel~\eqref{eq:backfor}--\eqref{eq:26} with $\nds=0.5$, to have the
  same look ahead parameter as in the flux;

\item\label{item:3} the forward looking
  kernel~\eqref{eq:forward}--\eqref{eq:24} with $\nds=0.5$, exactly as
  in the flux.
\end{enumerate}
In cases~\ref{item:1} and~\ref{item:2} with the nonlocal term of
type~\eqref{eq:backfor}, more mass is transported from lane 1 to lane
2, especially in the front part of the support of the density of lane
2, even though the leading part of lane 1 is aware of the density on
lane 2. In addition, more mass is transported with smaller nonlocal
range due to similar effects as already described above. In
contrast
, more mass seems to be transported from the rear part of lane 2 to
lane 1 when the nonlocal term with forward looking
kernel~\eqref{eq:forward} is used. This may be due to the fact that
for the back- and forward looking kernel~\eqref{eq:backfor} the
nonlocal velocities on both roads depend on free space and density,
but for the forward looking kernel~\eqref{eq:backfor} the velocity of
the second lane does not include some free space and lane changing
becomes favourable.

\setlength\fwidth{0.85\textwidth}
\begin{figure}[h!]
  \centering \tikzsetnextfilename{Keimer_different_sources_05_v2}
  \input{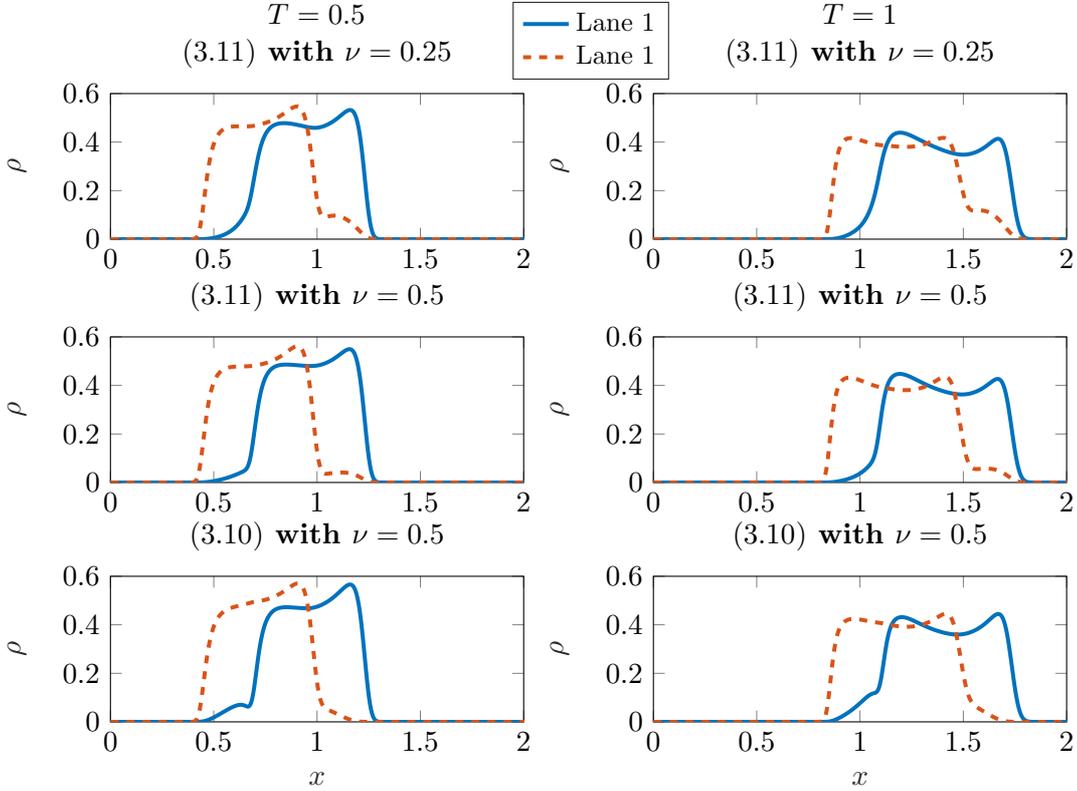}
  \caption{Density profiles for the model~\eqref{eq:nonlocaltransport}
    at $T=0.5$ (left column) and $T=1$ (right column),
    velocities~\eqref{eq:23}, initial datum~\eqref{eq:24}, $\ndt=0.5$,
    $w_\ndt$ as in~\eqref{eq:25}. Concerning $w_\nds$ and the
    parameter $\nds$: first row represents case~\ref{item:1}, second
    row case~\ref{item:2}, third row case~\ref{item:3}.}
  \label{fig:Keimer_sources}
\end{figure}

\section*{Conclusion}
Inspired by the models presented in~\cite{BayenKeimer-preprint} and~\cite{HoldenRisebro}, we have introduced a multilane traffic model
that allows for nonlocality in the source and in the flux term. For
both approaches we have shown existence and uniqueness of
solutions. Based on a Godunov type discretization, we also present a
numerical study comparing the influence of the nonlocality and
different kernels. Future works include the consideration of the
continuum limit for infinitely many lanes and comparisons to real
data.

\section*{Acknowledgment}
S.G.~was supported by the German Research Foundation (DFG) under grant
GO 1920/10-1.\\
E.R.~is a member of INdAM-GNAMPA (Gruppo Nazionale per l'Analisi
Matematica, la Probabilità e le loro Applicazioni) and was partly
supported by the GNAMPA 2020 project "From Wellposedness to Game
Theory in Conservation Laws".

\bigskip

{ \small

  \bibliography{traffic}

  \bibliographystyle{abbrv}

}

\end{document}

%% file: Plots/2Lanes_example_local_density.tex
%
%
\definecolor{mycolor1}{rgb}{0.00000,0.44700,0.74100}%
\definecolor{mycolor2}{rgb}{0.85000,0.32500,0.09800}%
\begin{tikzpicture}

\begin{axis}[%
width=0.411\fwidth,
height=0.146\fwidth,
at={(0\fwidth,0.404\fwidth)},
scale only axis,
xmin=0,
xmax=2,
xlabel style={font=\color{white!15!black}},
ymin=0,
ymax=1,
ylabel style={font=\color{white!15!black}},
ylabel={$\rho$},
axis background/.style={fill=white},
title style={font=\bfseries, align=center},
title={ $T=0.75$\\Source \eqref{eq:source} with \eqref{eq:local_source}},
legend style={legend cell align=left, align=left, draw=white!15!black, font=\small,at={(0.4\fwidth,0.16\fwidth)},anchor=south west}
]
\addplot [color=mycolor1, line width=1.5pt]
  table[row sep=crcr]{%
0.00499999999999989	0.286198060735236\\
0.075	0.263410618238683\\
0.145	0.240921453862713\\
0.215	0.218765719647276\\
0.275	0.200070917474671\\
0.335	0.181679858163632\\
0.395	0.163626126678616\\
0.445	0.148867029400679\\
0.455	0.14598863444731\\
0.485	0.148122073105424\\
0.495	0.14899230350326\\
0.505	0.150024060500326\\
0.515	0.154087079105952\\
0.525	0.732308910068371\\
0.535	0.798643211239791\\
0.615	0.771380007720345\\
0.705	0.740393068546927\\
0.815	0.702199207004913\\
0.965	0.64976747210068\\
1.395	0.499006786209079\\
1.455	0.476800816739905\\
1.505	0.458015269809848\\
1.545	0.442930786230456\\
1.575	0.431914950829349\\
1.615	0.417646821548228\\
1.655	0.403713925244457\\
1.715	0.383169628917575\\
1.805	0.35268686872973\\
1.895	0.3225363626803\\
1.975	0.296047622971253\\
1.995	0.289475909213425\\
};
\addlegendentry{Lane 1}

\addplot [color=mycolor2, dashed,  line width=1.5pt]
  table[row sep=crcr]{%
0.00499999999999989	0.417119960940545\\
0.135	0.385873143624987\\
0.265	0.354304025226143\\
0.395	0.322418516038446\\
0.445	0.310074032049446\\
0.455	0.341016502551162\\
0.465	0.734040489464701\\
0.475	0.745006993956153\\
0.515	0.751083314163215\\
0.525	0.752672095535351\\
0.545	0.748659084180286\\
0.655	0.725684441961673\\
0.765	0.702387625404118\\
0.875	0.678788859457694\\
0.995	0.652719614669548\\
1.115	0.626326813822132\\
1.235	0.599607271439287\\
1.315	0.581528013776763\\
1.365	0.569971400201655\\
1.405	0.560443378320242\\
1.465	0.545725070657334\\
1.535	0.528610339067081\\
1.655	0.499665105709591\\
1.845	0.455135927506748\\
1.975	0.424284967213451\\
1.995	0.419510200809808\\
};
\addlegendentry{Lane 2}

\end{axis}

\begin{axis}[%
width=0.411\fwidth,
height=0.146\fwidth,
at={(0.54\fwidth,0.404\fwidth)},
scale only axis,
xmin=0,
xmax=2,
xlabel style={font=\color{white!15!black}},
ymin=0,
ymax=1,
ylabel style={font=\color{white!15!black}},
ylabel={$\rho$},
axis background/.style={fill=white},
title style={font=\bfseries, align=center},
title={ $T=1.5$\\Source \eqref{eq:source} with \eqref{eq:local_source}}
]
\addplot [color=mycolor1, line width=1.5pt, forget plot]
  table[row sep=crcr]{%
0.00499999999999989	0.342698381521172\\
0.115	0.322084035588218\\
0.225	0.301649076875233\\
0.305	0.286902369846698\\
0.315	0.285466835890332\\
0.335	0.284912113609609\\
0.355	0.284509663363976\\
0.385	0.284089720622931\\
0.415	0.283837452833848\\
0.445	0.283759776660482\\
0.475	0.283855574931506\\
0.505	0.284123883667164\\
0.535	0.284564577288525\\
0.565	0.285177787022274\\
0.585	0.285711578561004\\
0.595	0.29076111432746\\
0.605	0.481397470258657\\
0.615	0.616909594049621\\
0.865	0.567496241921973\\
1.105	0.520255915580234\\
1.225	0.496609345816916\\
1.485	0.444532342949872\\
1.735	0.394629693258114\\
1.805	0.380915544185286\\
1.875	0.367389650853262\\
1.955	0.35213839547412\\
1.995	0.344582245026943\\
};
\addplot [color=mycolor2, dashed,  line width=1.5pt, forget plot]
  table[row sep=crcr]{%
0.00499999999999989	0.488896858170809\\
0.155	0.469176059261756\\
0.305	0.449290815405434\\
0.315	0.556839655769866\\
0.325	0.644109458840408\\
0.335	0.650822709376988\\
0.345	0.651664867656761\\
0.415	0.655903278353667\\
0.495	0.66057531000994\\
0.565	0.664464065963154\\
0.605	0.666560213227759\\
0.615	0.665754471613117\\
0.765	0.647713372804675\\
0.905	0.630682175317705\\
1.015	0.617115913279284\\
1.105	0.605843317430472\\
1.195	0.5943918629789\\
1.345	0.575056703282301\\
1.555	0.547812441338977\\
1.785	0.517765439175851\\
1.995	0.490204590543324\\
};
\end{axis}

\begin{axis}[%
width=0.411\fwidth,
height=0.146\fwidth,
at={(0\fwidth,0.16\fwidth)},
scale only axis,
xmin=0,
xmax=2,
xlabel style={font=\color{white!15!black}},
ymin=0,
ymax=1,
ylabel style={font=\color{white!15!black}},
ylabel={$\rho$},
axis background/.style={fill=white},
title style={font=\bfseries},
title={Source \eqref{eq:source} with \eqref{eq:forward}}
]
\addplot [color=mycolor1, line width=1.5pt, forget plot]
  table[row sep=crcr]{%
0.00499999999999989	0.282460370435377\\
0.085	0.258248364671164\\
0.225	0.215596286975029\\
0.285	0.19763798877111\\
0.345	0.17998831334771\\
0.405	0.162686094219627\\
0.445	0.151358474799135\\
0.515	0.132375116653482\\
0.525	0.414861186861967\\
0.535	0.79699891364523\\
0.615	0.769698856656232\\
0.705	0.738647769555785\\
0.815	0.700340758943597\\
0.945	0.654726746036836\\
1.155	0.580670405155779\\
1.395	0.49586526415061\\
1.475	0.465992768807858\\
1.555	0.435954843196273\\
1.585	0.425021577342815\\
1.625	0.410781210062744\\
1.675	0.393328501066003\\
1.755	0.365798175830316\\
1.845	0.33514873514168\\
1.925	0.308190617674878\\
1.965	0.294954338912817\\
1.995	0.285514011535555\\
};
\addplot [color=mycolor2, dashed,  line width=1.5pt, forget plot]
  table[row sep=crcr]{%
0.00499999999999989	0.419838582865661\\
0.0449999999999999	0.408969437578633\\
0.085	0.398542215396322\\
0.155	0.38069056910256\\
0.405	0.317149861789203\\
0.435	0.309498408172967\\
0.445	0.308853452287828\\
0.455	0.71389327832162\\
0.465	0.764383781595708\\
0.525	0.754223326001857\\
0.535	0.752363776210663\\
0.645	0.729468983672619\\
0.765	0.704162832227345\\
0.885	0.678539965078209\\
1.005	0.65261502531947\\
1.125	0.626393429133528\\
1.235	0.602059995131532\\
1.305	0.586312855116434\\
1.355	0.574806143050583\\
1.405	0.562969966416691\\
1.495	0.541166093065868\\
1.595	0.517161200744122\\
1.675	0.498110459988834\\
1.815	0.465692818000912\\
1.945	0.435264779620082\\
1.965	0.43037226969037\\
1.985	0.425213851519565\\
1.995	0.422550682082717\\
};
\end{axis}

\begin{axis}[%
width=0.411\fwidth,
height=0.146\fwidth,
at={(0.54\fwidth,0.16\fwidth)},
scale only axis,
xmin=0,
xmax=2,
xlabel style={font=\color{white!15!black}},
ymin=0,
ymax=1,
ylabel style={font=\color{white!15!black}},
ylabel={$\rho$},
axis background/.style={fill=white},
title style={font=\bfseries},
title={Source \eqref{eq:source} with \eqref{eq:forward}}
]
\addplot [color=mycolor1, line width=1.5pt, forget plot]
  table[row sep=crcr]{%
0.00499999999999989	0.346025608476249\\
0.0349999999999999	0.342320144346484\\
0.0549999999999999	0.339981283854552\\
0.075	0.337818069289346\\
0.095	0.335858939991693\\
0.115	0.334100717332343\\
0.135	0.332533314108598\\
0.155	0.331000944596826\\
0.175	0.329255346853347\\
0.195	0.327309551665263\\
0.215	0.325190798922924\\
0.235	0.322926907650215\\
0.255	0.320547259731587\\
0.285	0.317301300864307\\
0.305	0.314742548494894\\
0.325	0.311826317156638\\
0.345	0.308565852171497\\
0.365	0.304973044804786\\
0.385	0.3010598840432\\
0.405	0.296855954892132\\
0.475	0.282043375381367\\
0.545	0.267441873884587\\
0.615	0.253038847430526\\
0.625	0.251010219922049\\
0.635	0.249699215911004\\
0.645	0.275405775793308\\
0.655	0.531672542400944\\
0.665	0.606381005485409\\
0.975	0.54477138263457\\
1.145	0.511183086770696\\
1.225	0.495359313944484\\
1.495	0.441143797360464\\
1.565	0.426842939942248\\
1.635	0.412343034791095\\
1.775	0.383173789372658\\
1.795	0.379250468531471\\
1.815	0.375475832599448\\
1.835	0.371843142684276\\
1.855	0.36834637836669\\
1.885	0.363344603940284\\
1.915	0.358622133357775\\
1.945	0.354168350392978\\
1.975	0.3499725842569\\
1.995	0.347314113890446\\
};
\addplot [color=mycolor2, dashed,  line width=1.5pt, forget plot]
  table[row sep=crcr]{%
0.00499999999999989	0.478750402560742\\
0.0449999999999999	0.469851133981978\\
0.0649999999999999	0.465295747447551\\
0.085	0.460515793413499\\
0.105	0.455552609342313\\
0.135	0.447983302509817\\
0.145	0.445590545861043\\
0.155	0.443386954004601\\
0.165	0.441333077729267\\
0.185	0.437513873009675\\
0.205	0.433926916043591\\
0.235	0.428779902697929\\
0.245	0.427102957465356\\
0.255	0.436480871637119\\
0.265	0.629049188250755\\
0.275	0.685225653478982\\
0.285	0.687257207979162\\
0.295	0.687107925537147\\
0.315	0.686550323536823\\
0.335	0.685820249853303\\
0.355	0.684918517685843\\
0.375	0.683844276424924\\
0.395	0.68259596388642\\
0.425	0.680438620371119\\
0.485	0.675898661783148\\
0.545	0.671144643645281\\
0.605	0.666169080742435\\
0.655	0.661821400295381\\
0.685	0.658293310787889\\
0.815	0.642724755026641\\
0.925	0.629363270217701\\
1.015	0.618244723614319\\
1.105	0.606926232046316\\
1.205	0.594145621734445\\
1.375	0.572378529485104\\
1.455	0.562357239531735\\
1.535	0.552533881399336\\
1.615	0.542899932045397\\
1.735	0.528678136514444\\
1.765	0.525044631800497\\
1.785	0.522228612263625\\
1.805	0.519088607362525\\
1.825	0.515699374350945\\
1.845	0.512095455411114\\
1.865	0.508280968929903\\
1.885	0.504265483878441\\
1.955	0.489591361296072\\
1.995	0.480945624053237\\
};
\end{axis}

\begin{axis}[%
width=0.411\fwidth,
height=0.146\fwidth,
at={(0\fwidth,-0.08\fwidth)},
scale only axis,
xmin=0,
xmax=2,
xlabel style={font=\color{white!15!black}},
xlabel={$x$},
ymin=0,
ymax=1,
ylabel style={font=\color{white!15!black}},
ylabel={$\rho$},
axis background/.style={fill=white},
title style={font=\bfseries},
title={Source \eqref{eq:source} with \eqref{eq:backfor}}
]
\addplot [color=mycolor1, line width=1.5pt, forget plot]
  table[row sep=crcr]{%
0.00499999999999989	0.28598906041569\\
0.085	0.259948509477971\\
0.155	0.237480015110642\\
0.205	0.221684190281028\\
0.235	0.212562838348963\\
0.265	0.203737610546194\\
0.445	0.151966969131212\\
0.515	0.133029862994245\\
0.525	0.149802779757226\\
0.535	0.775445252376929\\
0.545	0.79926286757392\\
0.585	0.785507260787001\\
0.625	0.771402130386596\\
0.655	0.760536201780163\\
0.685	0.749359366585816\\
0.735	0.73044893847408\\
0.765	0.719569447304652\\
0.805	0.705613482682858\\
0.935	0.660228113683037\\
1.165	0.579508926209789\\
1.405	0.495134275414446\\
1.475	0.469137177158732\\
1.575	0.431689624163224\\
1.605	0.421004662379302\\
1.645	0.407059660698154\\
1.715	0.383055304321366\\
1.805	0.352538477748569\\
1.895	0.322355259824267\\
1.975	0.295845407439663\\
1.995	0.289269177322161\\
};
\addplot [color=mycolor2, dashed,  line width=1.5pt, forget plot]
  table[row sep=crcr]{%
0.00499999999999989	0.417311859300335\\
0.135	0.386094307898104\\
0.205	0.369099997237827\\
0.235	0.361466276129203\\
0.275	0.350898436821918\\
0.325	0.337742398148976\\
0.395	0.319709911592792\\
0.435	0.30951486450622\\
0.445	0.48586500735496\\
0.455	0.760832118235474\\
0.465	0.759902533657186\\
0.545	0.745263220760538\\
0.605	0.733257435944277\\
0.655	0.723561817357594\\
0.695	0.716081824204273\\
0.725	0.710457076478677\\
0.755	0.704512834930916\\
0.785	0.698219876247844\\
0.905	0.672383594650127\\
1.025	0.646254996449019\\
1.145	0.619825312095936\\
1.255	0.59530425832763\\
1.325	0.579430993659433\\
1.375	0.567788401067525\\
1.415	0.558169368048528\\
1.605	0.511817122290355\\
1.665	0.497447043511916\\
1.795	0.467073378072741\\
1.915	0.438746159242367\\
1.995	0.419700955673942\\
};
\end{axis}

\begin{axis}[%
width=0.411\fwidth,
height=0.146\fwidth,
at={(0.54\fwidth,-0.08\fwidth)},
scale only axis,
xmin=0,
xmax=2,
xlabel style={font=\color{white!15!black}},
xlabel={$x$},
ymin=0,
ymax=1,
ylabel style={font=\color{white!15!black}},
ylabel={$\rho$},
axis background/.style={fill=white},
title style={font=\bfseries},
title={Source \eqref{eq:source} with \eqref{eq:backfor}}
]
\addplot [color=mycolor1, line width=1.5pt, forget plot]
  table[row sep=crcr]{%
0.00499999999999989	0.342525527508165\\
0.0349999999999999	0.337005555854094\\
0.0649999999999999	0.331704566175923\\
0.095	0.326608302766091\\
0.125	0.321704760131635\\
0.155	0.316984005309535\\
0.185	0.312437990136623\\
0.215	0.30806003427657\\
0.245	0.303844588070077\\
0.275	0.299821032052848\\
0.295	0.297228760266765\\
0.315	0.294865278833953\\
0.335	0.292750776676751\\
0.355	0.290885446462022\\
0.375	0.289269227870498\\
0.395	0.287769184377497\\
0.415	0.286083951413631\\
0.435	0.28421043657334\\
0.455	0.282163887517741\\
0.475	0.279961646228295\\
0.495	0.277623632336572\\
0.515	0.275082064717933\\
0.535	0.27222633203629\\
0.555	0.269078606856141\\
0.575	0.265662017144594\\
0.595	0.261996617145797\\
0.615	0.258101675611745\\
0.625	0.256406527272663\\
0.635	0.275408363966165\\
0.645	0.543597554944029\\
0.655	0.619508420733494\\
0.685	0.611288485214066\\
0.715	0.603275157177994\\
0.745	0.595491914551073\\
0.775	0.587968087398487\\
0.795	0.583114659939758\\
0.815	0.578407056449814\\
0.835	0.573861186402272\\
0.855	0.569495112102849\\
0.875	0.565329538669735\\
0.895	0.561377275821261\\
0.995	0.541801868467655\\
1.215	0.498506896659817\\
1.455	0.450468933287164\\
1.655	0.410403992938884\\
1.735	0.39446186954231\\
1.795	0.382700766781977\\
1.865	0.36918486322161\\
1.955	0.352011406086839\\
1.995	0.344415577066776\\
};
\addplot [color=mycolor2, dashed,  line width=1.5pt, forget plot]
  table[row sep=crcr]{%
0.00499999999999989	0.489182967805138\\
0.0249999999999999	0.486454254414886\\
0.0449999999999999	0.483534097846092\\
0.0649999999999999	0.480448943299338\\
0.085	0.477221598772104\\
0.115	0.472155034307581\\
0.145	0.466862945443515\\
0.175	0.461385062826345\\
0.215	0.453845073806163\\
0.245	0.448043664970695\\
0.255	0.550793643379188\\
0.265	0.665161250165775\\
0.275	0.675798652532594\\
0.285	0.675111133924534\\
0.305	0.672510789463421\\
0.325	0.67012703844834\\
0.345	0.668034594412662\\
0.365	0.6662570657306\\
0.385	0.664817556080024\\
0.405	0.663688232757617\\
0.425	0.662734334708758\\
0.445	0.661964440982991\\
0.465	0.661391801160413\\
0.485	0.661027220147802\\
0.505	0.66087861699967\\
0.525	0.66078613682235\\
0.545	0.660490676180971\\
0.565	0.659973766048062\\
0.585	0.659237959310183\\
0.605	0.658284801495209\\
0.625	0.657115468841236\\
0.655	0.655136450191483\\
0.685	0.652945464419457\\
0.715	0.650567522908011\\
0.745	0.648000746352409\\
0.775	0.645240923818644\\
0.805	0.64228129661703\\
0.835	0.63911320108712\\
0.865	0.635725798361595\\
0.895	0.6321086811715\\
1.035	0.614733424020121\\
1.125	0.60337121572534\\
1.235	0.589265578856084\\
1.525	0.551800573764539\\
1.715	0.527039691188206\\
1.995	0.490477235733421\\
};
\end{axis}

\begin{axis}[%
width=1.227\fwidth,
height=0.675\fwidth,
at={(-0.16\fwidth,-0.074\fwidth)},
scale only axis,
xmin=0,
xmax=1,
ymin=0,
ymax=1,
axis line style={draw=none},
ticks=none,
axis x line*=bottom,
axis y line*=left
]
\end{axis}
\end{tikzpicture}%

%% file: Plots/2Lanes_example_local_L1norm.tex
%
%
\definecolor{mycolor1}{rgb}{0.00000,0.44700,0.74100}%
\definecolor{mycolor2}{rgb}{0.85000,0.32500,0.09800}%
\definecolor{mycolor3}{rgb}{0,0,0}
\begin{tikzpicture}

\begin{axis}[%
width=0.411\fwidth,
height=0.411\fwidth,
at={(0\fwidth,0\fwidth)},
scale only axis,
xmin=0,
xmax=1.5,
xlabel style={font=\color{white!15!black}},
xlabel={$t$},
ymin=0.82,
ymax=1,
ylabel style={font=\color{white!15!black}},
ylabel={$\L1$-Norm},
axis background/.style={fill=white},
title style={font=\bfseries},
title={Lane 1},
legend style={legend cell align=left, align=left, draw=white!15!black}
]
\addplot [color=mycolor1, line width=1.5pt, dashed]
  table[row sep=crcr]{%
0	1\\
0.02	0.997528589870899\\
0.05	0.993913877337277\\
0.0800000000000001	0.990374798787232\\
0.12	0.985737532763329\\
0.14	0.983472345783477\\
0.16	0.981268950456699\\
0.19	0.978063788860636\\
0.22	0.974941193696517\\
0.26	0.970789788206917\\
0.28	0.968665423832083\\
0.3	0.966488421856433\\
0.32	0.964256787753156\\
0.35	0.960819311509556\\
0.38	0.957302983080115\\
0.42	0.952543092680881\\
0.45	0.948901064075224\\
0.53	0.939113044080837\\
0.56	0.935552470618159\\
0.64	0.926172289753597\\
0.67	0.922757460841017\\
0.71	0.918286582590649\\
0.73	0.916105117215306\\
0.76	0.912911125332568\\
0.8	0.908705460310401\\
0.83	0.905643940750125\\
0.86	0.902647017228047\\
0.89	0.899731748308189\\
0.93	0.895918627313883\\
0.96	0.89314608557998\\
0.99	0.890435969600829\\
1.03	0.886915126025228\\
1.06	0.884340565091345\\
1.1	0.881008364168676\\
1.13	0.87857014522426\\
1.16	0.876186752368409\\
1.2	0.873095359678381\\
1.23	0.870834695182675\\
1.26	0.868627012080001\\
1.3	0.865752538579926\\
1.34	0.862966905831781\\
1.38	0.860255925974851\\
1.42	0.857626693139689\\
1.46	0.855065905604945\\
1.5	0.852579954895998\\
};
\addlegendentry{\eqref{eq:local_source}}

\addplot [color=mycolor2, line width=1.5pt]
  table[row sep=crcr]{%
0	1\\
0.02	0.997526215617887\\
0.04	0.995105926959107\\
0.0700000000000001	0.991556077386039\\
0.17	0.979838551411565\\
0.21	0.975019871267615\\
0.25	0.970112219113134\\
0.28	0.966364756224102\\
0.31	0.962555192086168\\
0.34	0.958687870796712\\
0.38	0.953459604063126\\
0.44	0.945527485670088\\
0.52	0.934948298087339\\
0.56	0.929717623427947\\
0.6	0.92455412665331\\
0.64	0.919473691135103\\
0.67	0.915725617395473\\
0.7	0.912035681196825\\
0.73	0.908407188414672\\
0.76	0.904842803449811\\
0.79	0.901344624924784\\
0.82	0.897914094137549\\
0.85	0.894552221111401\\
0.88	0.891259682674259\\
0.91	0.888036689119989\\
0.94	0.884883247801038\\
0.97	0.881799082574971\\
1	0.878783678942611\\
1.03	0.875836398191709\\
1.06	0.872956374814071\\
1.09	0.870142700596662\\
1.12	0.867395079548412\\
1.15	0.864716036138416\\
1.18	0.862108545337211\\
1.21	0.859575412692491\\
1.24	0.857118234064836\\
1.27	0.854735870310446\\
1.3	0.852426364767277\\
1.33	0.850187798757269\\
1.36	0.84801830736841\\
1.39	0.845915747898281\\
1.42	0.843877978417049\\
1.45	0.841903193344432\\
1.48	0.839989342064749\\
1.5	0.83874624466897\\
};
\addlegendentry{\eqref{eq:forward}}

\addplot [color=mycolor3, line width=1.5pt, dotted]
  table[row sep=crcr]{%
0	1\\
0.02	0.997527423750098\\
0.04	0.995107539975272\\
0.0700000000000001	0.991562033231606\\
0.1	0.988109883634589\\
0.12	0.985863705111237\\
0.15	0.982576575126459\\
0.22	0.975012950083079\\
0.24	0.972787914371553\\
0.26	0.970507481233408\\
0.28	0.968167131469841\\
0.3	0.965767645025885\\
0.32	0.963313220020683\\
0.35	0.959542237966918\\
0.38	0.955685430064565\\
0.42	0.950451185919589\\
0.48	0.942507085175842\\
0.54	0.934575843398724\\
0.58	0.929347307383392\\
0.62	0.924192048828007\\
0.65	0.92038349965186\\
0.68	0.916630497129815\\
0.71	0.912936905369881\\
0.74	0.909305651002319\\
0.77	0.905738781794894\\
0.8	0.902237740226723\\
0.83	0.898803244634573\\
0.86	0.89543558638817\\
0.89	0.892134682119823\\
0.92	0.888900035943571\\
0.95	0.885730938965847\\
0.98	0.882626445730314\\
1.01	0.879585394764227\\
1.04	0.876606539873998\\
1.07	0.873688471506777\\
1.1	0.870829715288272\\
1.13	0.868028724156052\\
1.16	0.865285964118873\\
1.19	0.862613811518728\\
1.22	0.860032116031457\\
1.25	0.857542863429953\\
1.28	0.855140930746721\\
1.31	0.852820804297735\\
1.34	0.850577722432284\\
1.37	0.848407105914913\\
1.4	0.846304987115678\\
1.43	0.844267885979054\\
1.46	0.842292323467555\\
1.49	0.840375417374121\\
1.5	0.839748954544896\\
};
\addlegendentry{\eqref{eq:backfor}}

\end{axis}

\begin{axis}[%
width=0.411\fwidth,
height=0.411\fwidth,
at={(0.54\fwidth,0\fwidth)},
scale only axis,
xmin=0,
xmax=1.5,
xlabel style={font=\color{white!15!black}},
xlabel={$t$},
ymin=1,
ymax=1.18,
ylabel style={font=\color{white!15!black}},
ylabel={$\L1$-Norm},
axis background/.style={fill=white},
title style={font=\bfseries},
title={Lane 2}
]
\addplot [color=mycolor1, line width=1.5pt, dashed, forget plot]
  table[row sep=crcr]{%
0	1\\
0.02	1.0024714101291\\
0.05	1.00608612266272\\
0.0800000000000001	1.00962520121277\\
0.12	1.01426246723667\\
0.14	1.01652765421652\\
0.16	1.0187310495433\\
0.19	1.02193621113936\\
0.22	1.02505880630348\\
0.26	1.02921021179308\\
0.28	1.03133457616792\\
0.3	1.03351157814357\\
0.32	1.03574321224684\\
0.35	1.03918068849044\\
0.38	1.04269701691989\\
0.42	1.04745690731912\\
0.45	1.05109893592478\\
0.53	1.06088695591916\\
0.56	1.06444752938184\\
0.64	1.0738277102464\\
0.67	1.07724253915898\\
0.71	1.08171341740935\\
0.73	1.08389488278469\\
0.76	1.08708887466743\\
0.8	1.0912945396896\\
0.83	1.09435605924988\\
0.86	1.09735298277195\\
0.89	1.10026825169181\\
0.93	1.10408137268612\\
0.96	1.10685391442002\\
0.99	1.10956403039917\\
1.03	1.11308487397477\\
1.06	1.11565943490865\\
1.1	1.11899163583132\\
1.13	1.12142985477574\\
1.16	1.12381324763159\\
1.2	1.12690464032162\\
1.23	1.12916530481732\\
1.26	1.13137298792\\
1.3	1.13424746142007\\
1.34	1.13703309416822\\
1.38	1.13974407402515\\
1.42	1.14237330686031\\
1.46	1.14493409439506\\
1.5	1.147420045104\\
};
\addplot [color=mycolor2, line width=1.5pt, forget plot]
  table[row sep=crcr]{%
0	1\\
0.02	1.00247378438211\\
0.04	1.00489407304089\\
0.0700000000000001	1.00844392261396\\
0.17	1.02016144858843\\
0.21	1.02498012873239\\
0.25	1.02988778088687\\
0.28	1.0336352437759\\
0.31	1.03744480791383\\
0.34	1.04131212920329\\
0.38	1.04654039593687\\
0.44	1.05447251432991\\
0.52	1.06505170191266\\
0.56	1.07028237657205\\
0.6	1.07544587334669\\
0.64	1.0805263088649\\
0.67	1.08427438260453\\
0.7	1.08796431880317\\
0.73	1.09159281158533\\
0.76	1.09515719655019\\
0.79	1.09865537507522\\
0.82	1.10208590586245\\
0.85	1.1054477788886\\
0.88	1.10874031732574\\
0.91	1.11196331088001\\
0.94	1.11511675219896\\
0.97	1.11820091742503\\
1	1.12121632105739\\
1.03	1.12416360180829\\
1.06	1.12704362518593\\
1.09	1.12985729940334\\
1.12	1.13260492045159\\
1.15	1.13528396386158\\
1.18	1.13789145466279\\
1.21	1.14042458730751\\
1.24	1.14288176593516\\
1.27	1.14526412968955\\
1.3	1.14757363523272\\
1.33	1.14981220124273\\
1.36	1.15198169263159\\
1.39	1.15408425210172\\
1.42	1.15612202158295\\
1.45	1.15809680665557\\
1.48	1.16001065793525\\
1.5	1.16125375533103\\
};
\addplot [color=mycolor3, line width=1.5pt,dotted, forget plot]
  table[row sep=crcr]{%
0	1\\
0.02	1.0024725762499\\
0.04	1.00489246002473\\
0.0700000000000001	1.00843796676839\\
0.1	1.01189011636541\\
0.12	1.01413629488876\\
0.15	1.01742342487354\\
0.22	1.02498704991692\\
0.24	1.02721208562845\\
0.26	1.02949251876659\\
0.28	1.03183286853016\\
0.3	1.03423235497411\\
0.32	1.03668677997932\\
0.35	1.04045776203308\\
0.38	1.04431456993543\\
0.42	1.04954881408041\\
0.48	1.05749291482416\\
0.54	1.06542415660128\\
0.58	1.07065269261661\\
0.62	1.07580795117199\\
0.65	1.07961650034814\\
0.68	1.08336950287019\\
0.71	1.08706309463012\\
0.74	1.09069434899768\\
0.77	1.09426121820511\\
0.8	1.09776225977328\\
0.83	1.10119675536543\\
0.86	1.10456441361183\\
0.89	1.10786531788018\\
0.92	1.11109996405643\\
0.95	1.11426906103415\\
0.98	1.11737355426969\\
1.01	1.12041460523577\\
1.04	1.123393460126\\
1.07	1.12631152849322\\
1.1	1.12917028471173\\
1.13	1.13197127584395\\
1.16	1.13471403588113\\
1.19	1.13738618848127\\
1.22	1.13996788396854\\
1.25	1.14245713657005\\
1.28	1.14485906925328\\
1.31	1.14717919570227\\
1.34	1.14942227756772\\
1.37	1.15159289408509\\
1.4	1.15369501288432\\
1.43	1.15573211402095\\
1.46	1.15770767653244\\
1.49	1.15962458262588\\
1.5	1.1602510454551\\
};
\end{axis}

\begin{axis}[%
width=1.227\fwidth,
height=0.675\fwidth,
at={(-0.16\fwidth,-0.074\fwidth)},
scale only axis,
xmin=0,
xmax=1,
ymin=0,
ymax=1,
axis line style={draw=none},
ticks=none,
axis x line*=bottom,
axis y line*=left
]
\end{axis}
\end{tikzpicture}%

%% file: Plots/2Lanes_eta_to_zero_v4.tex
%
%
\definecolor{mycolor1}{rgb}{0.85000,0.32500,0.09800}%
\definecolor{mycolor2}{rgb}{0.92900,0.69400,0.12500}%
\definecolor{mycolor3}{rgb}{0.49400,0.18400,0.55600}%
\definecolor{mycolor4}{rgb}{0.46600,0.67400,0.18800}%
\definecolor{mycolor5}{rgb}{0.30100,0.74500,0.93300}%
\definecolor{mycolor6}{rgb}{0.63500,0.07800,0.18400}%
\definecolor{mycolor7}{rgb}{0.00000,0.44700,0.74100}%
\begin{tikzpicture}

\begin{axis}[%
width=0.411\fwidth,
height=0.21\fwidth,
at={(0\fwidth,0.292\fwidth)},
scale only axis,
xmin=0,
xmax=2,
ymin=0.2,
ymax=0.65,
ylabel style={font=\color{white!15!black}},
ylabel={$\rho$},
axis background/.style={fill=white},
title style={font=\bfseries},
title={\eqref{eq:forward}}
]
\addplot [color=black, line width=1.5pt, forget plot]
  table[row sep=crcr]{%
0.00499999999999989	0.342698381521172\\
0.105	0.323950528177215\\
0.205	0.305351197179039\\
0.305	0.286902369846698\\
0.315	0.285466835890332\\
0.335	0.284912113609609\\
0.355	0.284509663363976\\
0.385	0.284089720622931\\
0.415	0.283837452833848\\
0.445	0.283759776660482\\
0.475	0.283855574931506\\
0.505	0.284123883667164\\
0.535	0.284564577288525\\
0.565	0.285177787022274\\
0.585	0.285711578561004\\
0.595	0.29076111432746\\
0.605	0.481397470258657\\
0.615	0.616909594049621\\
0.835	0.573414660565131\\
1.045	0.532050745367809\\
1.225	0.496609345816916\\
1.455	0.450532335117739\\
1.715	0.39858420224603\\
1.775	0.386768671358378\\
1.835	0.375097236504577\\
1.905	0.361644722375455\\
1.975	0.348355894801108\\
1.995	0.344582245026943\\
};
\addplot [color=mycolor1, dashed, line width=1.5pt, forget plot]
  table[row sep=crcr]{%
0.00499999999999989	0.354112528082215\\
0.0249999999999999	0.351876861509047\\
0.0449999999999999	0.349455626294949\\
0.0649999999999999	0.346889268766532\\
0.085	0.344202480924283\\
0.115	0.339985763652378\\
0.145	0.335589297093199\\
0.175	0.331053958196556\\
0.215	0.324846804142827\\
0.285	0.313787142308783\\
0.305	0.310300814238399\\
0.325	0.306577814007087\\
0.445	0.283263525274844\\
0.585	0.256294933830487\\
0.625	0.248543580873702\\
0.635	0.24713703987128\\
0.645	0.265498807914324\\
0.655	0.503136395416516\\
0.665	0.606172939329172\\
0.835	0.572435436865918\\
0.995	0.540837406151484\\
1.145	0.511225587287518\\
1.205	0.499258273356975\\
1.275	0.484871791349844\\
1.355	0.468269271603034\\
1.475	0.443191886880903\\
1.635	0.409710591152075\\
1.655	0.405751700150458\\
1.675	0.401914742067406\\
1.695	0.398196085189438\\
1.715	0.394592950654566\\
1.735	0.391103006484169\\
1.755	0.387724131381388\\
1.775	0.384454258785452\\
1.795	0.381291322769811\\
1.815	0.378233241137305\\
1.835	0.375277910093737\\
1.865	0.371033658746128\\
1.895	0.367011987061721\\
1.925	0.363206693088095\\
1.955	0.35961180088601\\
1.985	0.356224454303042\\
1.995	0.355160553863802\\
};
\addplot [color=mycolor2, dashed, line width=1.5pt, forget plot]
  table[row sep=crcr]{%
0.00499999999999989	0.341041925442004\\
0.0249999999999999	0.337634210028095\\
0.0449999999999999	0.334382013822372\\
0.0649999999999999	0.331282285784604\\
0.085	0.328332946856114\\
0.105	0.325532698867045\\
0.125	0.322880874432985\\
0.145	0.320377325302113\\
0.165	0.318102353185849\\
0.185	0.316122775931152\\
0.205	0.314434106935129\\
0.225	0.313031184465829\\
0.245	0.311910658529982\\
0.265	0.311068354068202\\
0.275	0.310817267692383\\
0.295	0.311211209611818\\
0.325	0.311901975798207\\
0.335	0.311908631704217\\
0.345	0.311732138356224\\
0.355	0.311376456967564\\
0.365	0.310845896128759\\
0.375	0.310144604834348\\
0.385	0.30927634371729\\
0.395	0.308244444356313\\
0.405	0.307051932805667\\
0.415	0.305701717169555\\
0.425	0.304196736888155\\
0.435	0.30254004247514\\
0.445	0.300734823126683\\
0.455	0.298784406127137\\
0.465	0.296692243139778\\
0.475	0.294462033360003\\
0.515	0.285059642952682\\
0.575	0.27077031075905\\
0.625	0.258873093109381\\
0.635	0.258970470395506\\
0.645	0.342106774714815\\
0.655	0.602069633925376\\
0.665	0.606633500716352\\
1.055	0.529369768010798\\
1.235	0.493725138991678\\
1.415	0.457613540114966\\
1.755	0.389567884222811\\
1.835	0.37384352777892\\
1.975	0.34645504320383\\
1.995	0.342805376669602\\
};
\addplot [color=mycolor3, dashed, line width=1.5pt, forget plot]
  table[row sep=crcr]{%
0.00499999999999989	0.342035199239064\\
0.135	0.317509379359463\\
0.155	0.314000204558512\\
0.175	0.310758575973689\\
0.195	0.307777594082886\\
0.215	0.305054040193965\\
0.225	0.303788328214263\\
0.235	0.302647667965918\\
0.245	0.301631970264053\\
0.255	0.30074334229318\\
0.265	0.299982370048405\\
0.275	0.299348347688615\\
0.285	0.298843485795864\\
0.295	0.298653783972923\\
0.325	0.298571774593893\\
0.355	0.29864908009362\\
0.385	0.298897105602328\\
0.415	0.299334854918659\\
0.445	0.299982757202433\\
0.465	0.300525027113912\\
0.475	0.300643000760773\\
0.485	0.300434478284797\\
0.495	0.299896535148567\\
0.505	0.299030180568931\\
0.515	0.297838548246153\\
0.525	0.296326194701151\\
0.535	0.29449900578069\\
0.545	0.292369804201553\\
0.565	0.287792143302747\\
0.585	0.283010431066982\\
0.605	0.278045779874942\\
0.615	0.275547677758571\\
0.625	0.27618022651135\\
0.635	0.401898757990021\\
0.645	0.610768210990642\\
0.865	0.567249056670789\\
1.245	0.492150663107145\\
1.415	0.45805269348327\\
1.755	0.390086692371696\\
1.815	0.378338842049284\\
1.875	0.366749815542162\\
1.945	0.35338828043416\\
1.995	0.343922706516235\\
};
\addplot [color=mycolor4, dashed, line width=1.5pt, forget plot]
  table[row sep=crcr]{%
0.00499999999999989	0.342350180098022\\
0.105	0.323582426244182\\
0.215	0.303069329798387\\
0.225	0.301298296928423\\
0.235	0.299647622528912\\
0.245	0.298118823799296\\
0.255	0.296711603095859\\
0.265	0.295425403467369\\
0.275	0.2943754031742\\
0.285	0.293562692925786\\
0.295	0.292996215862346\\
0.305	0.292761257243205\\
0.335	0.292420819291833\\
0.365	0.292247724918687\\
0.395	0.292234845982156\\
0.425	0.292379680506295\\
0.455	0.292680367945043\\
0.485	0.29314674876621\\
0.515	0.29381426937262\\
0.535	0.2943397862188\\
0.545	0.294172064215455\\
0.555	0.2933798481028\\
0.565	0.2919680483565\\
0.575	0.289975743614517\\
0.585	0.287793004016494\\
0.605	0.283180187749039\\
0.615	0.287633957378343\\
0.625	0.482567033703472\\
0.635	0.612841953464439\\
0.855	0.569335600763304\\
1.085	0.524013314880784\\
1.225	0.496388242543346\\
1.425	0.456277537837388\\
1.735	0.394322459427301\\
1.795	0.382524416910659\\
1.855	0.370887106192535\\
1.915	0.359394439644299\\
1.995	0.344235694410845\\
};
\addplot [color=mycolor5, dashed, line width=1.5pt, forget plot]
  table[row sep=crcr]{%
0.00499999999999989	0.34250533434825\\
0.105	0.323747675098048\\
0.205	0.305136906733114\\
0.265	0.294075051722276\\
0.275	0.292495474859489\\
0.285	0.291150297739701\\
0.295	0.290210307833639\\
0.305	0.28970596939943\\
0.325	0.289328810660419\\
0.355	0.288939129047156\\
0.385	0.288718398530674\\
0.415	0.288664123926876\\
0.445	0.28877493253279\\
0.475	0.289049707093882\\
0.505	0.289487058765151\\
0.535	0.290089306701875\\
0.565	0.290866328677898\\
0.575	0.290444320415437\\
0.585	0.28882860406372\\
0.595	0.2866588826737\\
0.605	0.287431706829833\\
0.615	0.43912895520907\\
0.625	0.614870009804428\\
0.845	0.571367558518166\\
1.065	0.528024899855581\\
1.225	0.496487916165604\\
1.445	0.452388582575308\\
1.725	0.39644399589182\\
1.785	0.384627432678083\\
1.845	0.372966706319358\\
1.905	0.361449030957618\\
1.975	0.348164402552932\\
1.995	0.344389846281749\\
};
\addplot [color=mycolor6, dashed, line width=1.5pt, forget plot]
  table[row sep=crcr]{%
0.00499999999999989	0.342582276309221\\
0.105	0.323829227500787\\
0.205	0.305223205702798\\
0.285	0.290452044678546\\
0.295	0.288955063079486\\
0.305	0.288113377572788\\
0.315	0.287771854740107\\
0.335	0.287404669500809\\
0.365	0.286992175429548\\
0.395	0.286751281237714\\
0.425	0.286679423461123\\
0.455	0.286776425968599\\
0.485	0.287041036581989\\
0.515	0.287472493971025\\
0.545	0.288070321763549\\
0.575	0.288846249956578\\
0.585	0.288676423287574\\
0.595	0.286966701402179\\
0.605	0.315786002693002\\
0.615	0.585615292377145\\
0.625	0.614893665809812\\
0.845	0.571396975001405\\
1.055	0.530028949541705\\
1.225	0.496536762300069\\
1.445	0.452446530472834\\
1.715	0.398496003456686\\
1.775	0.386660864457889\\
1.835	0.374977937459072\\
1.895	0.36343572333299\\
1.965	0.350130261338877\\
1.995	0.344466394211467\\
};
\end{axis}

\begin{axis}[%
width=0.411\fwidth,
height=0.21\fwidth,
at={(0\fwidth,0\fwidth)},
scale only axis,
xmin=0,
xmax=2,
xlabel style={font=\color{white!15!black}},
xlabel={$x$},
ymin=0.4,
ymax=0.7,
ylabel style={font=\color{white!15!black}},
ylabel={$\rho$},
axis background/.style={fill=white}
]
\addplot [color=black, line width=1.5pt, forget plot]
  table[row sep=crcr]{%
0.00499999999999989	0.488896858170809\\
0.115	0.474454564423322\\
0.215	0.461230765578026\\
0.305	0.449290815405434\\
0.315	0.556839655769866\\
0.325	0.644109458840408\\
0.335	0.650822709376988\\
0.345	0.651664867656761\\
0.395	0.654706830569525\\
0.465	0.658856575432573\\
0.515	0.661708604651049\\
0.555	0.663946836486393\\
0.605	0.666560213227759\\
0.615	0.665754471613117\\
0.725	0.652543254375809\\
0.825	0.640440308179215\\
0.915	0.629456746997601\\
0.995	0.619598042445093\\
1.065	0.610875141066908\\
1.135	0.602045775380268\\
1.205	0.593109092151692\\
1.315	0.5789329746275\\
1.465	0.559513309440598\\
1.615	0.539993164835125\\
1.775	0.519076812937613\\
1.895	0.503290538515755\\
1.945	0.496730859050952\\
1.995	0.490204590543324\\
};
\addplot [color=mycolor1, dashed, line width=1.5pt, forget plot]
  table[row sep=crcr]{%
0.00499999999999989	0.468086201225926\\
0.0149999999999999	0.466432185391852\\
0.0349999999999999	0.463339489308575\\
0.0649999999999999	0.458869340099221\\
0.125	0.45007715200789\\
0.235	0.433976353642199\\
0.245	0.433140703506461\\
0.255	0.587415682674651\\
0.265	0.685703445118553\\
0.275	0.691413825653412\\
0.285	0.691132297410628\\
0.305	0.689976330058961\\
0.325	0.688676010770012\\
0.345	0.68723961488553\\
0.395	0.683444963452758\\
0.445	0.679552314530406\\
0.495	0.675556827181444\\
0.545	0.671453689182223\\
0.595	0.667238213318311\\
0.635	0.663778194953786\\
0.655	0.662009169946799\\
0.665	0.66087676771743\\
0.755	0.650092472185775\\
0.845	0.639210621317606\\
0.935	0.628225432626495\\
1.095	0.608554035037485\\
1.195	0.596312240834903\\
1.275	0.586614101133097\\
1.375	0.574610996156914\\
1.605	0.54707821533771\\
1.615	0.545772867970392\\
1.625	0.544361832226124\\
1.635	0.542856541454182\\
1.645	0.54127328229534\\
1.655	0.539624817278628\\
1.675	0.536169767630516\\
1.695	0.53254815410617\\
1.715	0.528795223526685\\
1.735	0.52493024718304\\
1.755	0.520960075935875\\
1.775	0.516887195153473\\
1.795	0.512715988001007\\
1.815	0.508443884970785\\
1.835	0.504047881724526\\
1.855	0.499524829547937\\
1.915	0.486626559437565\\
1.955	0.47799767616838\\
1.975	0.473828155439722\\
1.985	0.471834963306253\\
1.995	0.469898003574968\\
};
\addplot [color=mycolor2, dashed, line width=1.5pt, forget plot]
  table[row sep=crcr]{%
0.00499999999999989	0.490598612982066\\
0.0149999999999999	0.488873448786673\\
0.0249999999999999	0.487078884032036\\
0.0349999999999999	0.485221366299321\\
0.0449999999999999	0.48330653203356\\
0.0649999999999999	0.479323982007044\\
0.085	0.475163637306562\\
0.105	0.470850366175599\\
0.125	0.466403401769428\\
0.145	0.46183774554207\\
0.155	0.459481542617187\\
0.165	0.457026584034671\\
0.175	0.45448542308293\\
0.185	0.451867199425655\\
0.195	0.449178848122959\\
0.205	0.44642675202812\\
0.215	0.443616819547259\\
0.235	0.437844199610608\\
0.255	0.431897594314672\\
0.265	0.428868687199055\\
0.275	0.494820289078829\\
0.285	0.66032682619023\\
0.295	0.675314841049559\\
0.305	0.676131434298867\\
0.335	0.677408537289536\\
0.345	0.677743654159671\\
0.355	0.678007449045524\\
0.365	0.678200719186449\\
0.375	0.678324121799382\\
0.385	0.678378126889875\\
0.395	0.67836297971331\\
0.405	0.678278698957119\\
0.415	0.678125131960725\\
0.425	0.677902064423144\\
0.435	0.67760934255254\\
0.445	0.677246936864784\\
0.455	0.676814895533109\\
0.465	0.676313206553024\\
0.475	0.67574165763355\\
0.485	0.675099773798821\\
0.515	0.672950240433929\\
0.545	0.670700646545268\\
0.575	0.668349997027328\\
0.605	0.665896771883984\\
0.635	0.663342443256808\\
0.645	0.662459991015705\\
0.655	0.661485101365984\\
0.755	0.649528216042817\\
0.855	0.637478589123784\\
0.935	0.627739640069476\\
1.005	0.619115234579579\\
1.065	0.61163058339305\\
1.135	0.602788124291445\\
1.215	0.592566593892109\\
1.395	0.5694182769032\\
1.545	0.550052233202581\\
1.655	0.53575113475269\\
1.765	0.521445776081734\\
1.815	0.515040586281218\\
1.855	0.509995957823456\\
1.895	0.505048924418449\\
1.925	0.501421521811681\\
1.965	0.496643330955966\\
1.975	0.495279247013314\\
1.985	0.493810351231194\\
1.995	0.492247023343387\\
};
\addplot [color=mycolor3, dashed, line width=1.5pt, forget plot]
  table[row sep=crcr]{%
0.00499999999999989	0.489486767389146\\
0.075	0.480387217283229\\
0.115	0.475281634545669\\
0.125	0.474024373040067\\
0.135	0.472700018724108\\
0.145	0.471235595745958\\
0.155	0.469639622104748\\
0.165	0.467920811946907\\
0.175	0.466087604242119\\
0.185	0.464147867914467\\
0.195	0.462108854769419\\
0.205	0.459977210991386\\
0.215	0.457758994430649\\
0.225	0.455459700284618\\
0.235	0.452990717802955\\
0.245	0.450340385448883\\
0.255	0.447521737027806\\
0.265	0.444550728387409\\
0.275	0.441441035920994\\
0.285	0.447580337044152\\
0.295	0.609264244264274\\
0.305	0.660330666993518\\
0.315	0.662930468330014\\
0.325	0.663498735881491\\
0.365	0.665489437349971\\
0.415	0.667880031195828\\
0.475	0.67065306630256\\
0.485	0.67104166193009\\
0.495	0.67128506572552\\
0.505	0.671384705300385\\
0.515	0.671340904111891\\
0.525	0.671152680772846\\
0.535	0.670818200103786\\
0.545	0.670336099560979\\
0.555	0.66970656651458\\
0.575	0.66829746041764\\
0.595	0.666782904862327\\
0.615	0.665165792412753\\
0.635	0.663433215192323\\
0.645	0.662402995720067\\
0.765	0.648000813399702\\
0.865	0.635907516253861\\
0.945	0.626149323959286\\
1.015	0.617518697947553\\
1.075	0.610028372916193\\
1.135	0.602442896193518\\
1.205	0.593490650193605\\
1.345	0.575446052993082\\
1.525	0.552159411526394\\
1.695	0.530063667889719\\
1.825	0.513070484890037\\
1.995	0.490792368023412\\
};
\addplot [color=mycolor4, dashed, line width=1.5pt, forget plot]
  table[row sep=crcr]{%
0.00499999999999989	0.48920504957768\\
0.115	0.474776832571023\\
0.215	0.461635698248298\\
0.225	0.46018842304188\\
0.235	0.458514129747564\\
0.245	0.456624663401638\\
0.255	0.454533307733871\\
0.265	0.452252291638119\\
0.275	0.449621233461259\\
0.285	0.446619536439395\\
0.295	0.463574592334139\\
0.305	0.618348034107266\\
0.315	0.654797367033505\\
0.325	0.656901405402123\\
0.345	0.658074943711636\\
0.395	0.660823421356729\\
0.435	0.66294302725171\\
0.485	0.665487525362811\\
0.545	0.66841062362298\\
0.555	0.668694698793897\\
0.565	0.668682697446711\\
0.575	0.668368816808836\\
0.585	0.66774772388828\\
0.595	0.667042312145314\\
0.605	0.666281464470094\\
0.615	0.665463984059451\\
0.625	0.664582601379576\\
0.635	0.66347978480765\\
0.745	0.650269136231709\\
0.855	0.63695513169618\\
0.945	0.62596423059638\\
1.015	0.61732843284042\\
1.085	0.608587145518011\\
1.145	0.600995010495033\\
1.215	0.592031520177434\\
1.355	0.573967533020344\\
1.525	0.551943499729094\\
1.695	0.529811015702731\\
1.825	0.512800877933441\\
1.995	0.490511424176835\\
};
\addplot [color=mycolor5, dashed, line width=1.5pt, forget plot]
  table[row sep=crcr]{%
0.00499999999999989	0.489066047610214\\
0.115	0.474631572772473\\
0.215	0.461414291202371\\
0.265	0.45473648002948\\
0.275	0.452974530757069\\
0.285	0.4508093951199\\
0.295	0.448019059906339\\
0.305	0.517498134046888\\
0.315	0.640764828807994\\
0.325	0.65337368833417\\
0.335	0.654420715691376\\
0.365	0.65618620778332\\
0.425	0.659605954421288\\
0.475	0.662349529967386\\
0.515	0.664458240651116\\
0.575	0.667519985003296\\
0.585	0.667686893244056\\
0.595	0.667250223946394\\
0.605	0.666526246306098\\
0.615	0.665680463490457\\
0.625	0.664621680416771\\
0.735	0.651410899861255\\
0.835	0.639310205837049\\
0.925	0.628329882021478\\
1.005	0.618472395562839\\
1.075	0.609744851528843\\
1.135	0.602171958901032\\
1.205	0.593227764646938\\
1.325	0.57775407179112\\
1.495	0.555738956421941\\
1.655	0.534911680603813\\
1.805	0.515296136233305\\
1.995	0.49037301894388\\
};
\addplot [color=mycolor6, dashed, line width=1.5pt, forget plot]
  table[row sep=crcr]{%
0.00499999999999989	0.488997757516521\\
0.115	0.474560167302696\\
0.215	0.461340286456814\\
0.285	0.452027002470013\\
0.295	0.450184997593306\\
0.305	0.462607568280686\\
0.315	0.609099903779496\\
0.325	0.650355399555989\\
0.335	0.652931703998888\\
0.345	0.653597377804615\\
0.405	0.657143805917335\\
0.465	0.660577712852229\\
0.515	0.663322631678985\\
0.555	0.665444161542817\\
0.585	0.666928608708846\\
0.595	0.667356973962363\\
0.605	0.666688337770197\\
0.615	0.665777825891681\\
0.655	0.661001464708427\\
0.765	0.647758763696394\\
0.865	0.63562367153229\\
0.955	0.624602476272234\\
1.025	0.615941788477539\\
1.095	0.607180279021816\\
1.165	0.598303914770851\\
1.235	0.58932215200232\\
1.385	0.569950784995037\\
1.535	0.550490797544604\\
1.695	0.529633136530876\\
1.835	0.511292222004708\\
1.985	0.49161144504925\\
1.995	0.490305026133899\\
};
\draw[draw=black] (axis cs:0.3,0.65) rectangle (axis cs:0.7,0.69);
\addplot [color=black, forget plot]
  table[row sep=crcr]{%
0.7	0.69\\
1.67334465195246	0.523108384458078\\
};
\addplot [color=black, forget plot]
  table[row sep=crcr]{%
0.3	0.65\\
0.478098471986418	0.435173824130879\\
};
\end{axis}

\begin{axis}[%
width=0.411\fwidth,
height=0.21\fwidth,
at={(0.54\fwidth,0.292\fwidth)},
scale only axis,
xmin=0,
xmax=2,
ymin=0.2,
ymax=0.65,
ylabel style={font=\color{white!15!black}},
ylabel={$\rho$},
axis background/.style={fill=white},
title style={font=\bfseries},
title={\eqref{eq:backfor}}
]
\addplot [color=black, line width=1.5pt, forget plot]
  table[row sep=crcr]{%
0.00499999999999989	0.342698381521172\\
0.105	0.323950528177215\\
0.205	0.305351197179039\\
0.305	0.286902369846698\\
0.315	0.285466835890332\\
0.335	0.284912113609609\\
0.355	0.284509663363976\\
0.385	0.284089720622931\\
0.415	0.283837452833848\\
0.445	0.283759776660482\\
0.475	0.283855574931506\\
0.505	0.284123883667164\\
0.535	0.284564577288525\\
0.565	0.285177787022274\\
0.585	0.285711578561004\\
0.595	0.29076111432746\\
0.605	0.481397470258657\\
0.615	0.616909594049621\\
0.835	0.573414660565131\\
1.045	0.532050745367809\\
1.225	0.496609345816916\\
1.455	0.450532335117739\\
1.715	0.39858420224603\\
1.775	0.386768671358378\\
1.835	0.375097236504577\\
1.905	0.361644722375455\\
1.975	0.348355894801108\\
1.995	0.344582245026943\\
};
\addplot [color=mycolor1, dashed, line width=1.5pt, forget plot]
  table[row sep=crcr]{%
0.00499999999999989	0.350365948545074\\
0.0349999999999999	0.345749342933156\\
0.0649999999999999	0.34095981559728\\
0.095	0.336031077034413\\
0.135	0.329290919143679\\
0.175	0.322406848255901\\
0.255	0.308573936974055\\
0.335	0.295208630042422\\
0.625	0.246479777020026\\
0.635	0.244944288461777\\
0.645	0.251307810425964\\
0.655	0.450470556187119\\
0.665	0.620693993086305\\
0.775	0.599666212511618\\
0.825	0.58999461716185\\
0.855	0.58408059304274\\
0.875	0.580001278497443\\
0.885	0.577862014556315\\
0.905	0.573346127842678\\
0.925	0.568589596933572\\
0.955	0.56120495301597\\
1.055	0.536362464051561\\
1.095	0.526600232550783\\
1.135	0.516988428134781\\
1.165	0.509904344361901\\
1.195	0.502957733140188\\
1.245	0.491613396246297\\
1.265	0.487311787772738\\
1.285	0.483173128324414\\
1.375	0.465178345902001\\
1.575	0.425119782724501\\
1.625	0.415136123314994\\
1.655	0.409303449892482\\
1.685	0.403600459465106\\
1.715	0.39803154130805\\
1.745	0.392599298931854\\
1.775	0.387302182805132\\
1.805	0.382134556636212\\
1.835	0.377088018010434\\
1.875	0.370531774809093\\
1.915	0.364155764408494\\
1.955	0.35794059102586\\
1.995	0.351868704113125\\
};
\addplot [color=mycolor2, dashed, line width=1.5pt, forget plot]
  table[row sep=crcr]{%
0.00499999999999989	0.342908635180501\\
0.0349999999999999	0.337755416442569\\
0.0649999999999999	0.332766180164487\\
0.095	0.327931949799015\\
0.125	0.323244971162749\\
0.155	0.318698403080122\\
0.185	0.314286098844405\\
0.215	0.31000247144633\\
0.245	0.305842491433714\\
0.285	0.300505901574562\\
0.305	0.297995671327002\\
0.325	0.295581559795227\\
0.345	0.293015033344218\\
0.365	0.290317116805292\\
0.395	0.286115482707225\\
0.435	0.280340441031406\\
0.485	0.272935401670625\\
0.545	0.263874162248255\\
0.575	0.259280170599598\\
0.595	0.256049797456954\\
0.625	0.251038036786122\\
0.635	0.260687457561251\\
0.645	0.497188822534233\\
0.655	0.62288750069753\\
0.685	0.614677399509991\\
0.715	0.606598355832296\\
0.745	0.59866483754752\\
0.775	0.59089453170739\\
0.805	0.583309103628933\\
0.835	0.575935135628732\\
0.855	0.571152462418361\\
0.875	0.566489560338719\\
0.895	0.561959540535648\\
0.915	0.557577488673615\\
0.935	0.553360871596787\\
0.955	0.549329650086429\\
1.095	0.522045702291265\\
1.215	0.498475854359033\\
1.435	0.454439776386045\\
1.635	0.414382137384203\\
1.715	0.398383550790379\\
1.775	0.386546281340411\\
1.835	0.374864946405088\\
1.915	0.359457599841531\\
1.945	0.353751929505235\\
1.975	0.348236466717128\\
1.995	0.344664593373769\\
};
\addplot [color=mycolor3, dashed, line width=1.5pt, forget plot]
  table[row sep=crcr]{%
0.00499999999999989	0.342654826815788\\
0.105	0.323861851011219\\
0.125	0.320229662004259\\
0.145	0.316733189860825\\
0.165	0.313366857645107\\
0.185	0.310125569855656\\
0.205	0.307005008908491\\
0.225	0.304001519658514\\
0.245	0.301112036045674\\
0.265	0.298343648973112\\
0.295	0.29442022505004\\
0.305	0.293234143752231\\
0.315	0.29213972212655\\
0.325	0.291135059819404\\
0.335	0.29021997874429\\
0.345	0.289395635711432\\
0.355	0.28866335064446\\
0.365	0.288024176573314\\
0.375	0.287479074988689\\
0.385	0.287029168032299\\
0.395	0.286675821653674\\
0.405	0.286420599873816\\
0.425	0.286184418518482\\
0.455	0.286051752301689\\
0.475	0.285949582182979\\
0.485	0.285717289403536\\
0.495	0.285338649901016\\
0.505	0.284814654637573\\
0.515	0.284146665386888\\
0.525	0.283336144540094\\
0.535	0.282384659467207\\
0.545	0.281293928239395\\
0.555	0.280065825098778\\
0.565	0.278702358370008\\
0.575	0.277205644787084\\
0.585	0.275577886662062\\
0.595	0.273821356750003\\
0.605	0.27193871311998\\
0.615	0.269967193667214\\
0.625	0.271047778324065\\
0.635	0.412081910431312\\
0.645	0.617775306311939\\
0.665	0.612220737125556\\
0.685	0.606839783760353\\
0.705	0.60164938842966\\
0.725	0.596668731868448\\
0.745	0.591919439717818\\
0.765	0.587425882418494\\
0.785	0.583214540229181\\
0.815	0.577303436097425\\
0.955	0.5497338509341\\
1.245	0.49253798212023\\
1.455	0.450488308240037\\
1.655	0.410446317599727\\
1.725	0.396489790869118\\
1.785	0.384701099520897\\
1.845	0.373082481700157\\
1.915	0.359691855814973\\
1.995	0.344540609608814\\
};
\addplot [color=mycolor4, dashed, line width=1.5pt, forget plot]
  table[row sep=crcr]{%
0.00499999999999989	0.342664287354686\\
0.105	0.32390694412473\\
0.195	0.307146497624556\\
0.215	0.303540385367319\\
0.235	0.300178476259147\\
0.255	0.297052387910941\\
0.275	0.294155232602819\\
0.295	0.29149849486681\\
0.305	0.290329300502486\\
0.315	0.289328891336871\\
0.325	0.288503534064673\\
0.335	0.287855157671417\\
0.345	0.28738419699911\\
0.355	0.287092373151561\\
0.385	0.2866122732899\\
0.415	0.286285065131946\\
0.445	0.286121851933571\\
0.475	0.286127864443893\\
0.505	0.286308844975112\\
0.535	0.286656688128853\\
0.545	0.286600021934643\\
0.555	0.286250759906041\\
0.565	0.285610015809193\\
0.575	0.284680670982656\\
0.585	0.283466244909451\\
0.595	0.281971160950005\\
0.605	0.280249581822954\\
0.615	0.283589915764549\\
0.625	0.468934157855365\\
0.635	0.615849555445578\\
0.645	0.613093335917822\\
0.655	0.610447425112261\\
0.665	0.60791734436422\\
0.675	0.605508464184315\\
0.685	0.60322587621278\\
0.695	0.601076622070135\\
0.725	0.595123740400348\\
1.005	0.539899149675687\\
1.235	0.494572010260833\\
1.455	0.450500457882672\\
1.745	0.392588838355771\\
1.805	0.380854819562375\\
1.865	0.3692693422908\\
1.935	0.355903669114053\\
1.995	0.344549301935008\\
};
\addplot [color=mycolor5, dashed, line width=1.5pt, forget plot]
  table[row sep=crcr]{%
0.00499999999999989	0.34266154545052\\
0.105	0.323909297401199\\
0.205	0.305306810802826\\
0.255	0.296070988858451\\
0.265	0.29434041059543\\
0.275	0.292724835431051\\
0.285	0.291223244084523\\
0.295	0.28983913777254\\
0.305	0.288583750036689\\
0.315	0.287684008413246\\
0.325	0.287120971085172\\
0.345	0.286652899745474\\
0.375	0.286159710743546\\
0.405	0.285836128871258\\
0.435	0.285684060041522\\
0.465	0.285702039065174\\
0.495	0.285889318629732\\
0.525	0.286245155534818\\
0.555	0.286777337460451\\
0.565	0.286978938711419\\
0.575	0.286869442193413\\
0.585	0.286173547663663\\
0.595	0.284914571854761\\
0.605	0.285825764477926\\
0.615	0.43154538884925\\
0.625	0.616001745935837\\
0.635	0.613400592749845\\
0.645	0.611053576571791\\
0.655	0.608974550848068\\
0.935	0.553682448441755\\
1.185	0.504513769609307\\
1.225	0.496583875583553\\
1.445	0.452502319590286\\
1.725	0.396560705679742\\
1.785	0.38476799245437\\
1.845	0.373124053474861\\
1.915	0.359702363165923\\
1.985	0.346432181497907\\
1.995	0.344545900878271\\
};
\addplot [color=mycolor6, dashed, line width=1.5pt, forget plot]
  table[row sep=crcr]{%
0.00499999999999989	0.34266008611599\\
0.105	0.323909952706626\\
0.205	0.30530823202286\\
0.275	0.292375423123417\\
0.285	0.290568566708752\\
0.295	0.288971264680838\\
0.305	0.287603959591863\\
0.315	0.286703638310337\\
0.325	0.286398319529453\\
0.355	0.285818787838106\\
0.385	0.285418182457307\\
0.415	0.285190709025887\\
0.445	0.285134960793026\\
0.475	0.285250371877794\\
0.505	0.285536050570748\\
0.535	0.285991523183651\\
0.565	0.286617808320192\\
0.575	0.286862991497307\\
0.585	0.286866852835888\\
0.595	0.286115418475885\\
0.605	0.320987995813108\\
0.615	0.593661148461728\\
0.625	0.614926445751765\\
0.705	0.599084399079296\\
0.915	0.557623316024852\\
1.145	0.512374826530601\\
1.225	0.496584860894197\\
1.455	0.450503184340314\\
1.725	0.396571544181\\
1.785	0.384774981253237\\
1.845	0.373124988400982\\
1.915	0.359698331884568\\
1.985	0.346430066558284\\
1.995	0.344544122303349\\
};
\end{axis}

\begin{axis}[%
width=0.411\fwidth,
height=0.21\fwidth,
at={(0.54\fwidth,0\fwidth)},
scale only axis,
xmin=0,
xmax=2,
xlabel style={font=\color{white!15!black}},
xlabel={$x$},
ymin=0.4,
ymax=0.7,
ylabel style={font=\color{white!15!black}},
ylabel={$\rho$},
axis background/.style={fill=white},
legend style={at={(-0.43\fwidth,-0.18\fwidth)}, anchor=south west, align=left, draw=white!15!black},
legend columns=4
]
\addplot [color=black, line width=1.5pt]
  table[row sep=crcr]{%
0.00499999999999989	0.488896858170809\\
0.115	0.474454564423322\\
0.215	0.461230765578026\\
0.305	0.449290815405434\\
0.315	0.556839655769866\\
0.325	0.644109458840408\\
0.335	0.650822709376988\\
0.345	0.651664867656761\\
0.395	0.654706830569525\\
0.465	0.658856575432573\\
0.515	0.661708604651049\\
0.555	0.663946836486393\\
0.605	0.666560213227759\\
0.615	0.665754471613117\\
0.725	0.652543254375809\\
0.825	0.640440308179215\\
0.915	0.629456746997601\\
0.995	0.619598042445093\\
1.065	0.610875141066908\\
1.135	0.602045775380268\\
1.205	0.593109092151692\\
1.315	0.5789329746275\\
1.465	0.559513309440598\\
1.615	0.539993164835125\\
1.775	0.519076812937613\\
1.895	0.503290538515755\\
1.945	0.496730859050952\\
1.995	0.490204590543324\\
};
\addlegendentry{\eqref{eq:local_source}}

\addplot [color=mycolor1, dashed, line width=1.5pt]
  table[row sep=crcr]{%
0.00499999999999989	0.476316537239677\\
0.0249999999999999	0.473491081584619\\
0.0549999999999999	0.469410432904085\\
0.115	0.461376894747841\\
0.225	0.446720951563908\\
0.235	0.497012760783226\\
0.245	0.657826317433429\\
0.255	0.687296168440326\\
0.265	0.688214819598999\\
0.275	0.68728076589354\\
0.465	0.66768975448226\\
0.535	0.660598123790486\\
0.595	0.654621451934735\\
0.645	0.649744519947174\\
0.655	0.648781060948474\\
0.665	0.647712168141395\\
0.705	0.643271254963173\\
0.735	0.64003240995539\\
0.765	0.636890326908229\\
0.795	0.633864066336787\\
0.825	0.630975479801847\\
0.845	0.629138308288899\\
0.865	0.627381203112132\\
0.895	0.624885341224986\\
0.925	0.622363523772915\\
0.955	0.619747426243477\\
0.985	0.617034427833442\\
1.015	0.614222420795367\\
1.045	0.61130904761904\\
1.075	0.60829215453666\\
1.105	0.605170072943721\\
1.135	0.60194168126582\\
1.165	0.598606234440873\\
1.195	0.595162942922028\\
1.225	0.591610440715085\\
1.255	0.587947243528347\\
1.285	0.584168255986542\\
1.505	0.55580753588282\\
1.595	0.544172841243455\\
1.615	0.541428019874153\\
1.635	0.538525436627387\\
1.655	0.535503383879599\\
1.675	0.532389393575654\\
1.705	0.527585412114001\\
1.735	0.52264927167177\\
1.765	0.517590639554026\\
1.795	0.512426219581801\\
1.825	0.507170691782135\\
1.885	0.496565571031683\\
1.965	0.482721681995491\\
1.985	0.479406667326613\\
1.995	0.477826180713391\\
};
\addlegendentry{$\nds=0.64$}

\addplot [color=mycolor2, dashed, line width=1.5pt]
  table[row sep=crcr]{%
0.00499999999999989	0.488315508791189\\
0.0249999999999999	0.485101903164095\\
0.0449999999999999	0.481787399748079\\
0.0649999999999999	0.478386004587486\\
0.085	0.474909442189992\\
0.115	0.469574789452982\\
0.145	0.464119917655573\\
0.175	0.458566109955253\\
0.215	0.45103583748986\\
0.235	0.447225898440482\\
0.245	0.448962379502278\\
0.255	0.594976312154808\\
0.265	0.676173523893645\\
0.275	0.681552074760988\\
0.285	0.680837110095728\\
0.305	0.678805574067171\\
0.325	0.676924865984405\\
0.375	0.672449294960844\\
0.415	0.668964651893424\\
0.445	0.666422329969508\\
0.475	0.663975720012702\\
0.495	0.662432386380566\\
0.515	0.660995693197249\\
0.535	0.659700322601301\\
0.555	0.658576456659504\\
0.575	0.657643978360131\\
0.595	0.656743070442876\\
0.615	0.655675994927292\\
0.635	0.654439528109377\\
0.655	0.653136393856597\\
0.675	0.65174780099619\\
0.695	0.650291822626599\\
0.715	0.648768187985478\\
0.735	0.647176470522128\\
0.755	0.645516078525159\\
0.775	0.64378624412991\\
0.795	0.641986013378485\\
0.815	0.640114242817101\\
0.835	0.638169587129592\\
0.855	0.636150459617823\\
0.875	0.634055009087771\\
0.895	0.631881020876849\\
0.915	0.629625694643929\\
0.935	0.627286141444327\\
0.955	0.624858868054873\\
1.025	0.616076400549597\\
1.105	0.605951626319637\\
1.185	0.595720164276738\\
1.335	0.576392592677522\\
1.515	0.553134704793396\\
1.645	0.53622751620766\\
1.895	0.503658450601578\\
1.945	0.497166521005315\\
1.965	0.494369739592247\\
1.985	0.491411454120025\\
1.995	0.48987935790813\\
};
\addlegendentry{$\nds=0.32$}

\addplot [color=mycolor3, dashed, line width=1.5pt]
  table[row sep=crcr]{%
0.00499999999999989	0.488968987316764\\
0.105	0.475903619077421\\
0.115	0.474528014964664\\
0.125	0.47308708679074\\
0.135	0.471584816362173\\
0.145	0.470025377177652\\
0.165	0.46675023335578\\
0.185	0.463289372253498\\
0.205	0.459666379918121\\
0.225	0.455901192562151\\
0.245	0.45201058795211\\
0.255	0.450022747688042\\
0.265	0.544538089451124\\
0.275	0.65423284527805\\
0.285	0.664124652894244\\
0.295	0.663401174783172\\
0.305	0.662182911782019\\
0.315	0.661064301649608\\
0.325	0.66007527945825\\
0.335	0.659220376868442\\
0.345	0.658502885634765\\
0.355	0.657925764799931\\
0.365	0.657491649338688\\
0.375	0.657202811904199\\
0.385	0.657061093121625\\
0.395	0.657067819928325\\
0.405	0.657223769287783\\
0.415	0.657529159268179\\
0.425	0.657983304030558\\
0.485	0.661253832872076\\
0.495	0.661665611396997\\
0.505	0.661994137473724\\
0.515	0.662241138715318\\
0.525	0.662408088879903\\
0.535	0.66249616942584\\
0.545	0.6625062390223\\
0.555	0.662438831914738\\
0.565	0.662294209075847\\
0.575	0.662072471003349\\
0.585	0.661773695059052\\
0.595	0.661397999324557\\
0.605	0.660945459743564\\
0.615	0.660415953172999\\
0.635	0.659159346113229\\
0.655	0.657764162914763\\
0.675	0.656225770719747\\
0.695	0.654557617514514\\
0.715	0.652757441930552\\
0.735	0.650822616416074\\
0.755	0.648750112340701\\
0.775	0.646534915925148\\
0.795	0.644172052980522\\
0.945	0.625846778372006\\
1.015	0.617186659748744\\
1.075	0.60967573592858\\
1.145	0.60080287091237\\
1.225	0.590546528415304\\
1.395	0.568612724995391\\
1.535	0.550461129785757\\
1.695	0.529607666354362\\
1.825	0.512577955479412\\
1.995	0.49027544844242\\
};
\addlegendentry{$\nds=0.16$}

\addplot [color=mycolor4, dashed, line width=1.5pt]
  table[row sep=crcr]{%
0.00499999999999989	0.488939910289509\\
0.115	0.474498381756975\\
0.205	0.462591859155399\\
0.215	0.461131677525022\\
0.225	0.459563229454989\\
0.235	0.457892699986527\\
0.245	0.456126623221342\\
0.255	0.454271107133615\\
0.265	0.452331800832632\\
0.275	0.450461048747833\\
0.285	0.560132662214732\\
0.295	0.649228064952319\\
0.305	0.655350418799316\\
0.315	0.654649629100997\\
0.325	0.653895145946221\\
0.335	0.65341005021304\\
0.345	0.653212206550984\\
0.355	0.653305216984321\\
0.365	0.653688720326632\\
0.405	0.65596513712155\\
0.455	0.658733225502469\\
0.505	0.661413719648087\\
0.545	0.663505678400363\\
0.555	0.663919249083223\\
0.565	0.664175818592617\\
0.575	0.66427801377371\\
0.585	0.664228127516178\\
0.595	0.664027688169286\\
0.605	0.663677149465439\\
0.615	0.663175910680475\\
0.625	0.662529031437343\\
0.635	0.661830269376121\\
0.645	0.661051800560683\\
0.655	0.660208175922767\\
0.665	0.659298837300944\\
0.675	0.658322766253855\\
0.685	0.657278590901018\\
0.695	0.65616501649526\\
0.725	0.652568923095361\\
0.845	0.638038690181411\\
0.935	0.627038556208134\\
1.015	0.617154920794495\\
1.085	0.608398255719209\\
1.155	0.599526478330123\\
1.235	0.589266487256487\\
1.395	0.568606971520733\\
1.535	0.550446041840545\\
1.695	0.529585444905486\\
1.835	0.511239512321878\\
1.985	0.491554499555372\\
1.995	0.490247620807414\\
};
\addlegendentry{$\nds=0.08$}

\addplot [color=mycolor5, dashed, line width=1.5pt]
  table[row sep=crcr]{%
0.00499999999999989	0.488932692955697\\
0.115	0.47449178217169\\
0.215	0.461269267679887\\
0.255	0.455944862542202\\
0.265	0.454435024549785\\
0.275	0.452723369005883\\
0.285	0.450819833456226\\
0.295	0.492632197953072\\
0.305	0.63039002823566\\
0.315	0.651389320314992\\
0.325	0.651858242929684\\
0.335	0.651863477197462\\
0.345	0.652317802145924\\
0.445	0.658093428627074\\
0.505	0.661417730753117\\
0.555	0.664066336961078\\
0.575	0.665114530781387\\
0.585	0.665476068641818\\
0.595	0.66552754677787\\
0.605	0.66526778908585\\
0.615	0.664698016759678\\
0.625	0.663987367140741\\
0.635	0.663117499567988\\
0.645	0.662114291300277\\
0.655	0.660978494089025\\
0.775	0.646521663621914\\
0.875	0.63437241030571\\
0.965	0.623333396117604\\
1.035	0.61465395887447\\
1.105	0.605870532053806\\
1.175	0.59697408477082\\
1.255	0.58669397763195\\
1.415	0.566020873681501\\
1.565	0.546537328921196\\
1.735	0.524348017612584\\
1.855	0.508600091055591\\
1.965	0.494158158813802\\
1.995	0.490240307259408\\
};
\addlegendentry{$\nds=0.04$}

\addplot [color=mycolor6, dashed, line width=1.5pt]
  table[row sep=crcr]{%
0.00499999999999989	0.488930891910206\\
0.115	0.474490133053072\\
0.215	0.461267637976214\\
0.275	0.453290228945995\\
0.285	0.451915470682013\\
0.295	0.450198019337555\\
0.305	0.5020912860661\\
0.315	0.633575110182605\\
0.325	0.650599558175787\\
0.335	0.651751355584776\\
0.375	0.654153762156438\\
0.435	0.657672481005069\\
0.485	0.660522602629293\\
0.545	0.663819965724783\\
0.575	0.665373693696326\\
0.585	0.665869062100192\\
0.595	0.666306235910688\\
0.605	0.666169837124531\\
0.615	0.66550904806904\\
0.625	0.664562473340385\\
0.675	0.658574716440652\\
0.785	0.6453090183147\\
0.885	0.633148069926455\\
0.965	0.623329482874441\\
1.035	0.614650425824069\\
1.105	0.605869005180038\\
1.175	0.596975963358816\\
1.255	0.5866972321828\\
1.415	0.566020692191444\\
1.565	0.546536112218664\\
1.735	0.524346089554704\\
1.855	0.508597984981022\\
1.965	0.494156211780861\\
1.995	0.490238476203226\\
};
\addlegendentry{$\nds=0.02$}

\draw[draw=black] (axis cs:0.3,0.65) rectangle (axis cs:0.7,0.69);
\addplot [color=black]
  table[row sep=crcr]{%
0.7	0.69\\
1.67130730050934	0.523108384458078\\
};

\addplot [color=black]
  table[row sep=crcr]{%
0.3	0.65\\
0.476061120543294	0.435173824130879\\
};

\end{axis}

\begin{axis}[%
width=0.245\fwidth,
height=0.062\fwidth,
at={(0.098\fwidth,0.025\fwidth)},
scale only axis,
xmin=0.3,
xmax=0.7,
ymin=0.65,
ymax=0.69,
ticks=none,
axis background/.style={fill=white}
]
\addplot [color=mycolor1, forget plot, dashed, line width=1.5pt]
  table[row sep=crcr]{%
0.295	0.69057557706451\\
0.305	0.689976330058961\\
0.315	0.6893430913023\\
0.325	0.688676010770012\\
0.335	0.687974929137087\\
0.345	0.68723961488553\\
0.365	0.685733206915617\\
0.385	0.684211579327942\\
0.405	0.682674428628401\\
0.425	0.681121445367289\\
0.445	0.679552314530406\\
0.465	0.677966734522093\\
0.485	0.676364393570903\\
0.505	0.67474494355109\\
0.525	0.673108132299768\\
0.545	0.671453689182223\\
0.565	0.669780974984203\\
0.585	0.668090153927979\\
0.605	0.66638147857198\\
0.625	0.664650854733044\\
0.645	0.662903577683248\\
0.655	0.662009169946799\\
0.665	0.66087676771743\\
0.705	0.656095098907968\\
};
\addplot [color=black, forget plot, line width=1.5pt]
  table[row sep=crcr]{%
0.327816133777953	0.646\\
0.335	0.650822709376988\\
0.345	0.651664867656761\\
0.355	0.652282300687601\\
0.385	0.654104274636945\\
0.405	0.655306435762664\\
0.435	0.657090994668012\\
0.455	0.658271620930679\\
0.475	0.659435586264505\\
0.495	0.66057531000994\\
0.515	0.661708604651049\\
0.535	0.662848682195907\\
0.545	0.663407693878784\\
0.555	0.663946836486393\\
0.565	0.664464065963154\\
0.575	0.664970165076842\\
0.585	0.665485623128159\\
0.595	0.666028686136004\\
0.605	0.666560213227759\\
0.615	0.665754471613117\\
0.655	0.660962086253002\\
0.695	0.656156420028693\\
0.705	0.654952892273569\\
};
\addplot [color=mycolor2, forget plot, dashed, line width=1.5pt]
  table[row sep=crcr]{%
0.295	0.675314841049559\\
0.305	0.676131434298867\\
0.315	0.676572245421621\\
0.325	0.677001581267847\\
0.335	0.677408537289536\\
0.345	0.677743654159671\\
0.355	0.678007449045524\\
0.365	0.678200719186449\\
0.375	0.678324121799382\\
0.385	0.678378126889875\\
0.395	0.67836297971331\\
0.405	0.678278698957118\\
0.415	0.678125131960725\\
0.425	0.677902064423144\\
0.435	0.67760934255254\\
0.445	0.677246936864784\\
0.455	0.676814895533109\\
0.465	0.676313206553024\\
0.475	0.67574165763355\\
0.485	0.675099773798821\\
0.495	0.674394342161395\\
0.505	0.673677848922796\\
0.515	0.672950240433929\\
0.525	0.672211493185673\\
0.535	0.671461622351862\\
0.545	0.670700646545268\\
0.555	0.669928511967298\\
0.565	0.669145039666677\\
0.575	0.668349997027328\\
0.585	0.667543327728202\\
0.595	0.66672537797971\\
0.605	0.665896771883984\\
0.615	0.66505771376889\\
0.625	0.664207094896228\\
0.635	0.663342443256808\\
0.645	0.662459991015705\\
0.655	0.661485101365984\\
0.685	0.657907002452855\\
0.705	0.655516953286686\\
};
\addplot [color=mycolor3, forget plot, dashed, line width=1.5pt]
  table[row sep=crcr]{%
0.302193720212301	0.646\\
0.305	0.660330666993518\\
0.315	0.662930468330014\\
0.325	0.663498735881491\\
0.345	0.664504511365723\\
0.365	0.665489437349971\\
0.385	0.666456810805106\\
0.405	0.667409140739603\\
0.425	0.668347702305624\\
0.455	0.669737702932947\\
0.475	0.67065306630256\\
0.485	0.67104166193009\\
0.495	0.671285065725519\\
0.505	0.671384705300385\\
0.515	0.671340904111891\\
0.525	0.671152680772846\\
0.535	0.670818200103786\\
0.545	0.670336099560979\\
0.555	0.66970656651458\\
0.565	0.669015422379\\
0.575	0.66829746041764\\
0.585	0.667553054957602\\
0.595	0.666782904862327\\
0.605	0.665987398502708\\
0.615	0.665165792412753\\
0.625	0.664315990151445\\
0.635	0.663433215192323\\
0.645	0.662402995720067\\
0.685	0.657614397704158\\
0.705	0.65521560093286\\
};
\addplot [color=mycolor4, forget plot, dashed, line width=1.5pt]
  table[row sep=crcr]{%
0.312586412061009	0.646\\
0.315	0.654797367033505\\
0.325	0.656901405402123\\
0.335	0.657512863671209\\
0.355	0.658632209691024\\
0.375	0.65973556000336\\
0.395	0.660823421356729\\
0.415	0.661893361670123\\
0.435	0.66294302725171\\
0.455	0.663974001857354\\
0.475	0.664988076285845\\
0.495	0.665980758071458\\
0.535	0.667924670098445\\
0.545	0.66841062362298\\
0.555	0.668694698793897\\
0.565	0.668682697446711\\
0.575	0.668368816808836\\
0.585	0.66774772388828\\
0.595	0.667042312145314\\
0.605	0.666281464470094\\
0.615	0.665463984059451\\
0.625	0.664582601379576\\
0.635	0.66347978480765\\
0.675	0.658687158449467\\
0.705	0.655084309095054\\
};
\addplot [color=mycolor5, forget plot, dashed, line width=1.5pt]
  table[row sep=crcr]{%
0.319151978361832	0.646\\
0.325	0.65337368833417\\
0.335	0.654420715691376\\
0.345	0.655021007953558\\
0.365	0.65618620778332\\
0.385	0.6573394403677\\
0.405	0.658480211877751\\
0.425	0.659605954421288\\
0.445	0.660714437252208\\
0.465	0.661807751432389\\
0.485	0.662887182464726\\
0.495	0.663418910226925\\
0.505	0.66394279883312\\
0.525	0.66496721872833\\
0.545	0.665985169942859\\
0.565	0.667019407032103\\
0.575	0.667519985003295\\
0.585	0.667686893244056\\
0.595	0.667250223946394\\
0.605	0.666526246306098\\
0.615	0.665680463490457\\
0.625	0.664621680416771\\
0.665	0.659829251687255\\
0.705	0.655023762465767\\
};
\addplot [color=mycolor6, forget plot, dashed, line width=1.5pt]
  table[row sep=crcr]{%
0.32394428621593	0.646\\
0.325	0.650355399555989\\
0.335	0.652931703998888\\
0.345	0.653597377804615\\
0.365	0.654791521141724\\
0.385	0.655973722894966\\
0.405	0.657143805917335\\
0.425	0.658301285562402\\
0.445	0.65944690883895\\
0.465	0.660577712852229\\
0.485	0.661687285593534\\
0.515	0.663322631678985\\
0.535	0.664403866601877\\
0.545	0.664932071379602\\
0.555	0.665444161542817\\
0.565	0.665938752168795\\
0.575	0.666425065962598\\
0.585	0.666928608708846\\
0.595	0.667356973962363\\
0.605	0.666688337770197\\
0.615	0.665777825891681\\
0.625	0.664595582959136\\
0.665	0.659801788422519\\
0.705	0.654994800150927\\
};
\end{axis}

\begin{axis}[%
width=0.245\fwidth,
height=0.062\fwidth,
at={(0.638\fwidth,0.025\fwidth)},
scale only axis,
xmin=0.3,
xmax=0.7,
ymin=0.65,
ymax=0.69,
ticks=none,
axis background/.style={fill=white}
]
\addplot [color=mycolor1, forget plot, dashed, line width=1.5pt]
  table[row sep=crcr]{%
0.295	0.685209291654134\\
0.345	0.680027847883006\\
0.385	0.675896885080551\\
0.415	0.672809355861016\\
0.445	0.669733137324702\\
0.475	0.666670580347396\\
0.505	0.663624442341427\\
0.525	0.661604471505921\\
0.545	0.65959436807093\\
0.565	0.65759532114453\\
0.585	0.655609223649665\\
0.605	0.653637105557365\\
0.625	0.651679688854346\\
0.645	0.649744519947174\\
};
\addplot [color=black, forget plot, line width=1.5pt]
  table[row sep=crcr]{%
0.327816133777953	0.646\\
0.335	0.650822709376988\\
0.345	0.651664867656761\\
0.355	0.652282300687601\\
0.385	0.654104274636945\\
0.405	0.655306435762664\\
0.435	0.657090994668012\\
0.455	0.658271620930679\\
0.475	0.659435586264505\\
0.495	0.66057531000994\\
0.515	0.661708604651049\\
0.535	0.662848682195907\\
0.545	0.663407693878784\\
0.555	0.663946836486393\\
0.565	0.664464065963154\\
0.575	0.664970165076842\\
0.585	0.665485623128159\\
0.595	0.666028686136004\\
0.605	0.666560213227759\\
0.615	0.665754471613117\\
0.655	0.660962086253002\\
0.695	0.656156420028693\\
0.705	0.654952892273569\\
};
\addplot [color=mycolor2, forget plot, dashed, line width=1.5pt]
  table[row sep=crcr]{%
0.295	0.679810181903907\\
0.305	0.678805574067171\\
0.315	0.677843862973276\\
0.325	0.676924865984405\\
0.345	0.675120118700028\\
0.365	0.67333447818954\\
0.385	0.671569463819729\\
0.405	0.669826812414595\\
0.415	0.668964651893424\\
0.425	0.668109269682301\\
0.435	0.667261452245308\\
0.445	0.666422329969508\\
0.455	0.665593482885541\\
0.465	0.66477703581944\\
0.475	0.663975720012702\\
0.485	0.66319287697938\\
0.495	0.662432386380566\\
0.505	0.66169851372476\\
0.515	0.660995693197249\\
0.525	0.66032828043042\\
0.535	0.659700322601301\\
0.545	0.659115387981014\\
0.555	0.658576456659504\\
0.565	0.658085833698231\\
0.575	0.657643978360131\\
0.585	0.657210658636859\\
0.595	0.656743070442876\\
0.605	0.656231899059772\\
0.615	0.655675994927292\\
0.625	0.65507518043414\\
0.645	0.65379949856152\\
0.655	0.653136393856596\\
0.665	0.652450509566785\\
0.675	0.65174780099619\\
0.685	0.651028248368158\\
0.695	0.650291822626599\\
0.705	0.649538485246592\\
};
\addplot [color=mycolor3, forget plot, dashed, line width=1.5pt]
  table[row sep=crcr]{%
0.295	0.663401174783172\\
0.305	0.662182911782019\\
0.315	0.661064301649608\\
0.325	0.66007527945825\\
0.335	0.659220376868442\\
0.345	0.658502885634765\\
0.355	0.657925764799931\\
0.365	0.657491649338688\\
0.375	0.657202811904199\\
0.385	0.657061093121625\\
0.395	0.657067819928325\\
0.405	0.657223769287783\\
0.415	0.657529159268179\\
0.425	0.657983304030558\\
0.435	0.658536029049504\\
0.455	0.65966194030079\\
0.465	0.660213858817102\\
0.475	0.660756852935354\\
0.485	0.661253832872076\\
0.495	0.661665611396997\\
0.505	0.661994137473724\\
0.515	0.662241138715318\\
0.525	0.662408088879903\\
0.535	0.66249616942584\\
0.545	0.662506239022299\\
0.555	0.662438831914738\\
0.565	0.662294209075847\\
0.575	0.662072471003349\\
0.585	0.661773695059052\\
0.595	0.661397999324557\\
0.605	0.660945459743564\\
0.615	0.660415953172999\\
0.625	0.659809035047685\\
0.635	0.659159346113228\\
0.645	0.658485220725187\\
0.655	0.657764162914763\\
0.665	0.657011072394682\\
0.675	0.656225770719747\\
0.685	0.655408040441092\\
0.695	0.654557617514514\\
0.705	0.653674194970417\\
};
\addplot [color=mycolor4, forget plot, dashed, line width=1.5pt]
  table[row sep=crcr]{%
0.295	0.649228064952319\\
0.305	0.655350418799315\\
0.315	0.654649629100997\\
0.325	0.653895145946221\\
0.335	0.653410050213039\\
0.345	0.653212206550984\\
0.355	0.65330521698432\\
0.365	0.653688720326632\\
0.375	0.654245604604466\\
0.395	0.65539816777661\\
0.415	0.656526665480277\\
0.435	0.65763683788816\\
0.455	0.658733225502469\\
0.475	0.659816401673716\\
0.495	0.660884839211598\\
0.525	0.662464854235714\\
0.545	0.663505678400363\\
0.555	0.663919249083223\\
0.565	0.664175818592617\\
0.575	0.66427801377371\\
0.585	0.664228127516178\\
0.595	0.664027688169286\\
0.605	0.663677149465439\\
0.615	0.663175910680475\\
0.625	0.662529031437343\\
0.635	0.661830269376121\\
0.645	0.661051800560683\\
0.655	0.660208175922767\\
0.665	0.659298837300944\\
0.675	0.658322766253855\\
0.685	0.657278590901018\\
0.695	0.65616501649526\\
0.705	0.654982482782649\\
};
\addplot [color=mycolor5, forget plot, dashed, line width=1.5pt]
  table[row sep=crcr]{%
0.312433570477218	0.646\\
0.315	0.651389320314992\\
0.325	0.651858242929684\\
0.335	0.651863477197462\\
0.345	0.652317802145924\\
0.365	0.653488427164705\\
0.385	0.654655135399164\\
0.405	0.655812539149865\\
0.425	0.656959557537885\\
0.445	0.658093428627074\\
0.465	0.659213975856737\\
0.485	0.660323292852254\\
0.505	0.661417730753117\\
0.525	0.66248854932248\\
0.565	0.664592717367719\\
0.575	0.665114530781387\\
0.585	0.665476068641818\\
0.595	0.66552754677787\\
0.605	0.66526778908585\\
0.615	0.664698016759678\\
0.625	0.663987367140741\\
0.635	0.663117499567988\\
0.645	0.662114291300277\\
0.655	0.660978494089025\\
0.705	0.654968929244905\\
};
\addplot [color=mycolor6, forget plot, dashed, line width=1.5pt]
  table[row sep=crcr]{%
0.322298262958289	0.646\\
0.325	0.650599558175787\\
0.335	0.651751355584776\\
0.345	0.652361552284797\\
0.375	0.654153762156438\\
0.395	0.655337799211737\\
0.415	0.656510786293171\\
0.435	0.657672481005069\\
0.455	0.658823832455409\\
0.475	0.659961387517348\\
0.495	0.661078294140055\\
0.525	0.662727336876492\\
0.545	0.663819965724783\\
0.555	0.664354082632722\\
0.565	0.664872090020307\\
0.585	0.665869062100192\\
0.595	0.666306235910688\\
0.605	0.666169837124531\\
0.615	0.66550904806904\\
0.625	0.664562473340385\\
0.635	0.663373032266257\\
0.675	0.658574716440652\\
0.705	0.654967235102384\\
};
\end{axis}

\end{tikzpicture}%

%% file: Plots/Keimer_different_flux_05.tex
%
%
\definecolor{mycolor1}{rgb}{0.00000,0.44700,0.74100}%
\definecolor{mycolor2}{rgb}{0.85000,0.32500,0.09800}%
\begin{tikzpicture}

\begin{axis}[%
width=0.411\fwidth,
height=0.23\fwidth,
at={(0\fwidth,0.32\fwidth)},
scale only axis,
xmin=0,
xmax=2,
xlabel style={font=\color{white!15!black}},
ymin=0,
ymax=0.7,
ylabel style={font=\color{white!15!black}},
ylabel={$\rho$},
axis background/.style={fill=white},
title style={font=\bfseries, align=center},
title={ $T=0.5$\\ Model \eqref{eq:model}},
legend style={legend cell align=left, align=left, draw=white!15!black, font=\small,at={(0.4\fwidth,0.25\fwidth)},anchor=south west}
]
\addplot [color=mycolor1, line width=1.5pt]
  table[row sep=crcr]{%
0.000500000000000167	0\\
0.3665	7.28630369266625e-05\\
0.3675	0.00020253728055053\\
0.3695	0.00086197702518831\\
0.3885	0.00820641624433804\\
0.4045	0.0141004824686783\\
0.4175	0.0186066009673991\\
0.4285	0.0221549544020947\\
0.4395	0.0254264271750633\\
0.4495	0.0281384202263228\\
0.4585	0.0302993098165802\\
0.4685	0.0323950188481188\\
0.4795	0.0344207418724189\\
0.4905	0.0361557259692753\\
0.5015	0.0376065833602088\\
0.5125	0.0387786447677554\\
0.5225	0.0395960555697488\\
0.5325	0.0401568998667976\\
0.5415	0.0404158321346086\\
0.5555	0.0404790431218403\\
0.5715	0.0403359029319166\\
0.5845	0.0399673634065776\\
0.5965	0.0393528971030008\\
0.6055	0.038749052076307\\
0.6065	0.0388930268916057\\
0.6075	0.0398581026207565\\
0.6085	0.0445800643302321\\
0.6095	0.0663183823335345\\
0.6105	0.161200171330716\\
0.6115	0.452560007476774\\
0.6125	0.636902268764595\\
0.6135	0.637366167868117\\
0.6345	0.633411160278238\\
0.6475	0.631243201588824\\
0.6585	0.629669084193991\\
0.6655	0.628947767428408\\
0.6715	0.628588775061677\\
0.6775	0.628475459065946\\
0.6835	0.628606983788008\\
0.6895	0.628975276187026\\
0.6955	0.629554713030473\\
0.6965	0.624509144092756\\
0.7005	0.591236606861896\\
0.7045	0.555914408304889\\
0.7085	0.518506293373109\\
0.7105	0.499853877573529\\
0.7335	0.496352083525907\\
0.7605	0.491984249791333\\
0.7735	0.490158571893434\\
0.7845	0.488872759395243\\
0.7955	0.487853902431274\\
0.8035	0.487353778880064\\
0.8115	0.487102742528246\\
0.8265	0.486783148496453\\
0.8305	0.486412411242072\\
0.8345	0.485779017460927\\
0.8385	0.484850707825698\\
0.8425	0.483632933908523\\
0.8465	0.482157023932995\\
0.8515	0.480013555234904\\
0.8575	0.477108637832111\\
0.8645	0.47339740248003\\
0.8735	0.468289179534044\\
0.8845	0.461712741135253\\
0.8975	0.453620901234637\\
0.9125	0.443956542068826\\
0.9275	0.433984626592827\\
0.9425	0.423713876580765\\
0.9575	0.413135713373621\\
0.9725	0.402232316024136\\
0.9865	0.39174108956595\\
1.0005	0.38092108345275\\
1.0135	0.370553927362032\\
1.0265	0.359850661737128\\
1.0385	0.349646282738687\\
1.0505	0.339103823700389\\
1.0625	0.328194296671919\\
1.0735	0.317842892753585\\
1.0845	0.307124585281197\\
1.0945	0.297032207174811\\
1.1045	0.286575089710892\\
1.1145	0.275715360936131\\
1.1235	0.265560124153829\\
1.1325	0.255002828916796\\
1.1405	0.24524334199726\\
1.1485	0.235090311995583\\
1.1565	0.224497879070308\\
1.1635	0.214826758546969\\
1.1705	0.20473568552994\\
1.1775	0.194175987198678\\
1.1845	0.18309205476937\\
1.1905	0.173127082189371\\
1.1965	0.162688399022845\\
1.2025	0.151731923910555\\
1.2085	0.140215809353257\\
1.2145	0.128107216421984\\
1.2205	0.115394203419339\\
1.2265	0.102105098572302\\
1.2345	0.0836727715489358\\
1.2475	0.053502664754506\\
1.2525	0.0426554841265405\\
1.2565	0.0346422351105491\\
1.2595	0.0291330333941335\\
1.2625	0.0241196264566459\\
1.2655	0.0196442954089449\\
1.2685	0.0157293108284526\\
1.2715	0.0123756327805586\\
1.2745	0.00956388374325723\\
1.2775	0.00725732805889656\\
1.2805	0.00540624718336957\\
1.2835	0.00395293926108353\\
1.2865	0.00283660328908475\\
1.2895	0.00199754044171963\\
1.2935	0.00121522443311317\\
1.2975	0.000714811207644495\\
1.3025	0.000351164332118348\\
1.3095	0.000118817041672603\\
1.3225	1.20737050006703e-05\\
1.3915	1.79856129989275e-13\\
2.0995	0\\
};
\addlegendentry{Lane 1}

\addplot [color=mycolor2, dashed,  line width=1.5pt]
  table[row sep=crcr]{%
0.000500000000000167	0\\
0.3625	0.000101838658459474\\
0.3635	0.000399719896926776\\
0.3645	0.00156477923245779\\
0.3655	0.00611144066258174\\
0.3665	0.0238134760314783\\
0.3675	0.0918536022756893\\
0.3685	0.318123804507139\\
0.3695	0.603593236585439\\
0.3705	0.627906912306699\\
0.3765	0.628626575059738\\
0.3825	0.629606474824832\\
0.3905	0.631212346328046\\
0.4115	0.635754856567188\\
0.4305	0.640218657964764\\
0.4425	0.642697523504777\\
0.4465	0.643469383383278\\
0.4475	0.643140531338451\\
0.4505	0.622367594838704\\
0.4535	0.599703462463881\\
0.4565	0.575221177555683\\
0.4605	0.540334178707267\\
0.4615	0.531286306816094\\
0.4625	0.526413951844735\\
0.4655	0.526128293266588\\
0.4695	0.525843974848584\\
0.4715	0.525417112271716\\
0.4745	0.524402201379226\\
0.4855	0.520210737809654\\
0.4915	0.518311679198312\\
0.4985	0.516410863534212\\
0.5065	0.514545824848387\\
0.5155	0.512722203196668\\
0.5275	0.510587229282519\\
0.5435	0.508046744766144\\
0.5575	0.506055462540259\\
0.5705	0.504487370646129\\
0.5855	0.502967343138731\\
0.6035	0.501404130132254\\
0.6125	0.500614430107557\\
0.6175	0.499851244579507\\
0.6195	0.499026700532015\\
0.6225	0.497285075071551\\
0.6295	0.492740085727972\\
0.6445	0.482659560507657\\
0.6605	0.471607687157291\\
0.6765	0.460239757496295\\
0.6925	0.448537632815392\\
0.7045	0.439470646374762\\
0.7175	0.429254350999317\\
0.7335	0.416374251745493\\
0.7485	0.403981057407793\\
0.7635	0.391237960078159\\
0.7775	0.378989587613408\\
0.7905	0.367273766423521\\
0.8035	0.355191468054584\\
0.8155	0.34367807875783\\
0.8275	0.331779969432751\\
0.8385	0.32049595675634\\
0.8495	0.308807378429991\\
0.8595	0.297790741904653\\
0.8695	0.286361171031989\\
0.8785	0.275683872383242\\
0.8875	0.264596380130825\\
0.8965	0.253053547412414\\
0.9045	0.242368640832688\\
0.9125	0.231239937526965\\
0.9205	0.219619145804421\\
0.9275	0.209004966429188\\
0.9345	0.197933540367942\\
0.9415	0.186364228902044\\
0.9485	0.174257936168603\\
0.9555	0.161582765539035\\
0.9625	0.148324117466112\\
0.9695	0.134501575452732\\
0.9775	0.118120304010198\\
0.9975	0.0768087243942985\\
1.0025	0.0673003370312313\\
1.0065	0.0602252629715228\\
1.0105	0.0537323925447337\\
1.0135	0.0492915383297117\\
1.0165	0.0452410012636824\\
1.0195	0.0415880475003214\\
1.0225	0.0383273936159174\\
1.0255	0.0354420884628199\\
1.0285	0.0329053859283626\\
1.0315	0.0306833146031682\\
1.0345	0.0287375696289685\\
1.0385	0.0265043954347832\\
1.0425	0.0246011047854791\\
1.0475	0.0225649558435599\\
1.0535	0.0204650117294589\\
1.0605	0.0183111359470414\\
1.0695	0.0158267096693625\\
1.0795	0.0133288418035913\\
1.0905	0.0108603450468809\\
1.1015	0.00868040014959792\\
1.1125	0.00678890769319684\\
1.1235	0.00518054382804278\\
1.1345	0.0038436964675399\\
1.1455	0.00276092049196741\\
1.1575	0.00184294869330781\\
1.1705	0.00111917623128033\\
1.1855	0.000571083964107988\\
1.2025	0.00022651500818105\\
1.2255	4.45335094281241e-05\\
1.2745	1.5986992485395e-07\\
2.0995	0\\
};
\addlegendentry{Lane 2}

\end{axis}

\begin{axis}[%
width=0.411\fwidth,
height=0.23\fwidth,
at={(0.54\fwidth,0.32\fwidth)},
scale only axis,
xmin=0,
xmax=2,
xlabel style={font=\color{white!15!black}},
ymin=0,
ymax=0.7,
ylabel style={font=\color{white!15!black}},
ylabel={$\rho$},
axis background/.style={fill=white},
title style={font=\bfseries, align=center},
title={ $T=1$\\ Model \eqref{eq:model}}
]
\addplot [color=mycolor1, line width=1.5pt, forget plot]
  table[row sep=crcr]{%
0.000500000000000167	0\\
0.7205	9.0614458038818e-05\\
0.7225	0.000266402858939419\\
0.7245	0.000678214405652522\\
0.7275	0.00170097451702445\\
0.7475	0.00888629515101158\\
0.7625	0.0140732563106081\\
0.7745	0.0180415211075911\\
0.7845	0.0211687757452892\\
0.7935	0.0237792570430484\\
0.8015	0.0258853071752188\\
0.8085	0.0275389481697785\\
0.8155	0.0290167017536915\\
0.8235	0.0305161347056124\\
0.8325	0.0320147293101942\\
0.8435	0.033652075093892\\
0.8555	0.0352470583365481\\
0.8675	0.0366457084421961\\
0.8785	0.037734314684819\\
0.8885	0.0385401049702998\\
0.8985	0.0391465464365099\\
0.9075	0.0395005353923228\\
0.9175	0.0396873729220002\\
0.9445	0.0401463446116299\\
0.9455	0.0403644716574831\\
0.9465	0.0408181061655135\\
0.9475	0.0417663128745094\\
0.9485	0.0437504712234014\\
0.9495	0.0478968688788872\\
0.9505	0.0565273007770881\\
0.9515	0.0743251895007053\\
0.9525	0.11018327490207\\
0.9535	0.1778924271109\\
0.9555	0.39511443783066\\
0.9565	0.454116831911287\\
0.9575	0.470387728168056\\
0.9585	0.473475096693567\\
0.9595	0.473959223075332\\
0.9615	0.473944158704955\\
0.9775	0.472930995104413\\
0.9835	0.47237243665779\\
0.9895	0.471624732965967\\
0.9955	0.47066741940084\\
1.0015	0.469509286726746\\
1.0085	0.46793981949808\\
1.0165	0.465920232051711\\
1.0255	0.463440337493202\\
1.0365	0.460208580387178\\
1.0505	0.455893322509082\\
1.0685	0.450142425888324\\
1.0915	0.442586875732516\\
1.1185	0.433509009827203\\
1.1465	0.423896448576713\\
1.1745	0.414088588166104\\
1.2015	0.404436217755998\\
1.2275	0.394942658799361\\
1.2515	0.385985656577941\\
1.2745	0.377205154646784\\
1.2965	0.368603407770808\\
1.3175	0.360184716895352\\
1.3375	0.35195616026601\\
1.3565	0.343927870408997\\
1.3745	0.336113036944731\\
1.3915	0.328527782094566\\
1.4075	0.321190962889777\\
1.4235	0.313645927413057\\
1.4385	0.306367710275653\\
1.4535	0.298875106404196\\
1.4675	0.291673128353063\\
1.4815	0.284253284726901\\
1.4945	0.277152911007488\\
1.5075	0.269833529378336\\
1.5195	0.262867362409521\\
1.5315	0.255683621409735\\
1.5435	0.248264733468927\\
1.5545	0.241240732751603\\
1.5655	0.233985522451735\\
1.5755	0.227172975789662\\
1.5855	0.22013676780983\\
1.5955	0.212857972937459\\
1.6045	0.206081759057338\\
1.6135	0.199073377623295\\
1.6225	0.191812160594354\\
1.6305	0.185126461001031\\
1.6385	0.178203526087709\\
1.6465	0.171022125711906\\
1.6545	0.163558404437253\\
1.6615	0.15677518182768\\
1.6685	0.149735850917807\\
1.6755	0.142419341376723\\
1.6825	0.134803544707666\\
1.6895	0.126866193792764\\
1.6955	0.119790945334927\\
1.7015	0.11245325555412\\
1.7075	0.10484646637038\\
1.7135	0.0969709526930136\\
1.7205	0.0874592607557099\\
1.7285	0.0762228513817136\\
1.7525	0.0420994544405868\\
1.7575	0.0355144780434204\\
1.7615	0.0305516577492906\\
1.7655	0.0259218498607363\\
1.7685	0.0226979903189921\\
1.7715	0.0197059652268914\\
1.7745	0.0169580082260423\\
1.7775	0.0144614326101977\\
1.7805	0.01221839923282\\
1.7835	0.0102259842468366\\
1.7865	0.00847653146258809\\
1.7895	0.00695824538837853\\
1.7925	0.00565596039885063\\
1.7955	0.00455201123787985\\
1.7985	0.00362713021305749\\
1.8015	0.0028613051511277\\
1.8055	0.00205316110616538\\
1.8095	0.00144701327201169\\
1.8145	0.000911001565228275\\
1.8205	0.00050370966178459\\
1.8275	0.000239651662388152\\
1.8375	7.53143015632673e-05\\
1.8555	7.04552749164478e-06\\
1.9635	2.22044604925031e-15\\
2.0995	0\\
};
\addplot [color=mycolor2, dashed,  line width=1.5pt, forget plot]
  table[row sep=crcr]{%
0.000500000000000167	0\\
0.7125	5.26434761742856e-05\\
0.7145	0.00021329147487581\\
0.7155	0.000428941756025925\\
0.7165	0.000862145097904055\\
0.7175	0.00173193822356765\\
0.7185	0.0034774815800005\\
0.7195	0.00697872270429079\\
0.7205	0.0139962623733254\\
0.7215	0.0280354915018215\\
0.7225	0.0559442812118274\\
0.7235	0.110104625099185\\
0.7245	0.206338421365849\\
0.7255	0.336828210192268\\
0.7265	0.436275093062509\\
0.7275	0.471513158405366\\
0.7285	0.478564691743263\\
0.7295	0.479820954251198\\
0.7305	0.480161717105333\\
0.7345	0.480903553657654\\
0.7405	0.481867394488194\\
0.7445	0.482352343346691\\
0.7485	0.482653506155761\\
0.7525	0.482749479048216\\
0.7565	0.482647845142645\\
0.7615	0.482287059305235\\
0.7665	0.481733738664072\\
0.7735	0.480756041695904\\
0.7835	0.479142310695412\\
0.7955	0.477013729850712\\
0.8065	0.474869038461683\\
0.8175	0.47251270093203\\
0.8305	0.469514889632535\\
0.8455	0.465861560954665\\
0.8605	0.462013759798398\\
0.8745	0.458224617921465\\
0.8875	0.45451153888057\\
0.8995	0.450896992215972\\
0.9115	0.447081886181954\\
0.9315	0.440430529235538\\
0.9555	0.43231836554621\\
0.9645	0.428876524441705\\
0.9865	0.420279427521351\\
1.0065	0.41226937378815\\
1.0265	0.404052510583901\\
1.0455	0.39604066162029\\
1.0645	0.387816070515174\\
1.0825	0.379816688350745\\
1.1005	0.371603198244995\\
1.1175	0.363637550855092\\
1.1345	0.355456830884008\\
1.1505	0.347548619811527\\
1.1665	0.339425132008838\\
1.1815	0.331601406141299\\
1.1965	0.323563373199721\\
1.2105	0.315855358551876\\
1.2245	0.307935502238262\\
1.2385	0.299789571512194\\
1.2515	0.292009195100883\\
1.2645	0.284005948893859\\
1.2765	0.276406726663582\\
1.2885	0.268589845509191\\
1.3005	0.260539335018097\\
1.3115	0.252939016362765\\
1.3225	0.245111314124361\\
1.3335	0.237038253536255\\
1.3435	0.229469055181162\\
1.3535	0.221662868819412\\
1.3635	0.213600157483356\\
1.3725	0.206106299705987\\
1.3815	0.198369104477017\\
1.3905	0.190368991073839\\
1.3985	0.183019873607307\\
1.4065	0.175430267727681\\
1.4145	0.167583652944077\\
1.4225	0.159463556708813\\
1.4305	0.15105463768613\\
1.4385	0.142344464596875\\
1.4465	0.133326318967363\\
1.4545	0.12400343288366\\
1.4625	0.11439508838788\\
1.4715	0.103299955683583\\
1.4995	0.0685047400732639\\
1.5055	0.0615180042594212\\
1.5105	0.0559870164700458\\
1.5155	0.0507811456360625\\
1.5195	0.0468824839779312\\
1.5235	0.0432399728370259\\
1.5275	0.0398645944997917\\
1.5315	0.0367599376032972\\
1.5355	0.0339222167728566\\
1.5395	0.0313408956964953\\
1.5435	0.0289998233941553\\
1.5475	0.0268787232673726\\
1.5515	0.0249548380925817\\
1.5565	0.0227913464030136\\
1.5615	0.0208544369065051\\
1.5665	0.0191038565940489\\
1.5725	0.0172017982214894\\
1.5795	0.0152005496512557\\
1.5865	0.0133874357273065\\
1.5945	0.0115075118114181\\
1.6025	0.00980944003912532\\
1.6115	0.00810112188828249\\
1.6205	0.00659713157391417\\
1.6295	0.00528929844919945\\
1.6385	0.00416841952265612\\
1.6485	0.00312857715483927\\
1.6585	0.00228735485887421\\
1.6695	0.0015663894469462\\
1.6815	0.000989374119283237\\
1.6945	0.000564710596836715\\
1.7095	0.000268847236714276\\
1.7295	8.27402420209289e-05\\
1.7615	7.2559176169662e-06\\
1.9275	0\\
2.0995	0\\
};
\end{axis}

\begin{axis}[%
width=0.411\fwidth,
height=0.23\fwidth,
at={(0\fwidth,-0.01\fwidth)},
scale only axis,
xmin=0,
xmax=2,
xlabel style={font=\color{white!15!black}},
xlabel={$x$},
ymin=0,
ymax=0.7,
ylabel style={font=\color{white!15!black}},
ylabel={$\rho$},
axis background/.style={fill=white},
title style={font=\bfseries},
title={Model \eqref{eq:nonlocaltransport}}
]
\addplot [color=mycolor1, line width=1.5pt, forget plot]
  table[row sep=crcr]{%
0.000500000000000167	0\\
0.4195	0.000108371858718925\\
0.4285	0.000345998166029204\\
0.4355	0.000731746361901209\\
0.4415	0.0012651193757991\\
0.4475	0.00202640017436195\\
0.4535	0.0030383567134793\\
0.4595	0.00430693263326409\\
0.4655	0.00582407397457452\\
0.4715	0.00757251323737762\\
0.4785	0.00987544338888036\\
0.4855	0.0124268598244313\\
0.4935	0.0156022137974205\\
0.5025	0.0194485022504538\\
0.5125	0.0239908404154283\\
0.5245	0.0297106345364515\\
0.5425	0.0385830259411009\\
0.5635	0.0488832168598372\\
0.5745	0.0540273794214383\\
0.5835	0.0579915139296112\\
0.5915	0.0612547838681836\\
0.5985	0.0638477334605891\\
0.6045	0.065826650156882\\
0.6095	0.0672684670710559\\
0.6145	0.0684876056144672\\
0.6195	0.0694461170012075\\
0.6235	0.0699956628168348\\
0.6275	0.0703243091355938\\
0.6315	0.0704040926498029\\
0.6355	0.0702067759485114\\
0.6385	0.069862025565802\\
0.6415	0.06934132296495\\
0.6455	0.0683759577626368\\
0.6495	0.067135126490816\\
0.6605	0.0634092428463555\\
0.6625	0.062997365729176\\
0.6645	0.0628011129502162\\
0.6665	0.0628920921483966\\
0.6685	0.0633491720060291\\
0.6695	0.0637412469385912\\
0.6705	0.0642568682349571\\
0.6715	0.0649071176853604\\
0.6725	0.0657031193869835\\
0.6735	0.066655938964999\\
0.6745	0.067776475456645\\
0.6755	0.0690753469535097\\
0.6775	0.072248443525754\\
0.6795	0.0762499512868993\\
0.6815	0.0811422725262094\\
0.6835	0.086971707148265\\
0.6855	0.0937651149687371\\
0.6875	0.101527313342937\\
0.6895	0.110239474772808\\
0.6915	0.119858727589804\\
0.6945	0.135838216360684\\
0.6975	0.153397607692018\\
0.7015	0.178564021711323\\
0.7145	0.262312674232895\\
0.7185	0.285768120067535\\
0.7215	0.302159624851402\\
0.7245	0.317450790461429\\
0.7275	0.331625438928218\\
0.7305	0.344699609555125\\
0.7335	0.356711704120575\\
0.7365	0.367714610287588\\
0.7395	0.377769695860751\\
0.7425	0.386942400586269\\
0.7455	0.395299108749575\\
0.7485	0.402905006981367\\
0.7515	0.409822678731567\\
0.7545	0.416111238255251\\
0.7575	0.421825852912562\\
0.7605	0.427017540078464\\
0.7635	0.431733154136465\\
0.7665	0.436015501202905\\
0.7695	0.439903535894748\\
0.7725	0.443432606932644\\
0.7755	0.446634727709478\\
0.7785	0.449538854947845\\
0.7815	0.452171163808976\\
0.7845	0.454555311735688\\
0.7875	0.456712686231913\\
0.7915	0.459269597974225\\
0.7955	0.46149842153978\\
0.7995	0.463435027237715\\
0.8035	0.465110955839731\\
0.8075	0.466554006421137\\
0.8115	0.467788742532362\\
0.8165	0.469072142365547\\
0.8215	0.470101164459257\\
0.8265	0.470908165975974\\
0.8325	0.471623066734791\\
0.8385	0.472101368082459\\
0.8455	0.472409499451441\\
0.8525	0.472497129580074\\
0.8605	0.472384422952446\\
0.8705	0.472007917414016\\
0.8835	0.471269667668982\\
0.9295	0.468425834834926\\
0.9405	0.468090162265341\\
0.9495	0.468039563803253\\
0.9575	0.46821854197914\\
0.9645	0.46858896026564\\
0.9715	0.469194977723089\\
0.9775	0.469925993417648\\
0.9835	0.470869201272871\\
0.9895	0.472034788350113\\
0.9955	0.473426338409724\\
1.0015	0.475041079042249\\
1.0075	0.476871357593059\\
1.0145	0.479265419099633\\
1.0215	0.48192274994499\\
1.0285	0.484830933056232\\
1.0355	0.487983133871126\\
1.0425	0.491377805963471\\
1.0495	0.495017077349511\\
1.0565	0.498904698276321\\
1.0635	0.503044087041764\\
1.0705	0.507436597246988\\
1.0775	0.512079859488734\\
1.0855	0.517682978003357\\
1.0935	0.523578274784551\\
1.1025	0.530503677661362\\
1.1165	0.541627463176174\\
1.1265	0.549468891200767\\
1.1325	0.553919328057399\\
1.1375	0.55735586078324\\
1.1415	0.559849281827614\\
1.1455	0.56204198473132\\
1.1485	0.563441047846186\\
1.1515	0.56458648361262\\
1.1545	0.565434046524948\\
1.1565	0.565809317761585\\
1.1585	0.56601410999864\\
1.1605	0.566031067728855\\
1.1625	0.565841569316008\\
1.1645	0.56542568019659\\
1.1665	0.564762113829246\\
1.1685	0.563828202933301\\
1.1705	0.562599884012602\\
1.1725	0.561051698667369\\
1.1745	0.559156815751663\\
1.1765	0.556887079029623\\
1.1785	0.554213085604557\\
1.1805	0.551104301016875\\
1.1825	0.547529217492658\\
1.1845	0.543455562322136\\
1.1865	0.538850563686604\\
1.1885	0.533681281343909\\
1.1905	0.527915009318029\\
1.1925	0.521519756992191\\
1.1945	0.514464813641923\\
1.1965	0.506721399328944\\
1.1985	0.498263402089368\\
1.2005	0.489068197406513\\
1.2025	0.479117541036365\\
1.2055	0.462748539183742\\
1.2085	0.444636335219815\\
1.2115	0.424804549504948\\
1.2145	0.403323211100846\\
1.2175	0.380314354587915\\
1.2215	0.347574540716841\\
1.2265	0.304173192652618\\
1.2395	0.189202938955033\\
1.2435	0.156663082878532\\
1.2465	0.134016329894453\\
1.2495	0.113154653723488\\
1.2525	0.0942547938630862\\
1.2545	0.0827989500411923\\
1.2565	0.0722750859228816\\
1.2585	0.0626818438570407\\
1.2605	0.0540049344217479\\
1.2625	0.0462184131881838\\
1.2645	0.0392862051066114\\
1.2665	0.0331638005407808\\
1.2685	0.0278000451225839\\
1.2705	0.0231389480244788\\
1.2725	0.0191214394587229\\
1.2745	0.0156870175007784\\
1.2765	0.0127752358440243\\
1.2785	0.0103269968983697\\
1.2805	0.00828562782626063\\
1.2825	0.00659772982059703\\
1.2845	0.00521380245665704\\
1.2865	0.00408865475198894\\
1.2885	0.00318162229408525\\
1.2915	0.00215232369678953\\
1.2945	0.00143052258055354\\
1.2975	0.000934004643065478\\
1.3015	0.000514477076166919\\
1.3065	0.000233345831092802\\
1.3145	5.92676113346791e-05\\
1.3315	2.08389971456668e-06\\
1.6175	0\\
2.0995	0\\
};
\addplot [color=mycolor2, dashed,  line width=1.5pt, forget plot]
  table[row sep=crcr]{%
0.000500000000000167	0\\
0.3905	9.55482608513414e-05\\
0.3965	0.000351072436177713\\
0.4005	0.000773135850715256\\
0.4035	0.00134257093310453\\
0.4055	0.0019034384993466\\
0.4075	0.00265876473476823\\
0.4095	0.00365966922700123\\
0.4115	0.00496491207026661\\
0.4135	0.00664021270244497\\
0.4155	0.00875694200119659\\
0.4175	0.0113901473563618\\
0.4195	0.0146159307408671\\
0.4215	0.0185082708741744\\
0.4235	0.0231354542721118\\
0.4255	0.0285563459927798\\
0.4275	0.0348167782881581\\
0.4295	0.0419463546302725\\
0.4315	0.0499559516743262\\
0.4335	0.0588361517610214\\
0.4355	0.0685567584871665\\
0.4385	0.0845982503203473\\
0.4415	0.102168472163002\\
0.4455	0.127421358102871\\
0.4515	0.167460702158324\\
0.4585	0.214099734502906\\
0.4625	0.239409189260661\\
0.4665	0.26318314328577\\
0.4695	0.279850599032362\\
0.4725	0.295458901093699\\
0.4755	0.309991473919802\\
0.4785	0.323459692746504\\
0.4815	0.335895588126781\\
0.4845	0.347345555147395\\
0.4875	0.35786518591323\\
0.4905	0.367515197421378\\
0.4935	0.376358343437763\\
0.4965	0.384457161512076\\
0.4995	0.391872399028716\\
0.5025	0.398661972569908\\
0.5055	0.404880333828845\\
0.5085	0.41057813699463\\
0.5115	0.415802123523433\\
0.5145	0.420595158806006\\
0.5175	0.424996370823641\\
0.5205	0.429041353462255\\
0.5235	0.432762407048775\\
0.5265	0.436188796320797\\
0.5295	0.439347011878651\\
0.5335	0.443181907479263\\
0.5375	0.446633038694173\\
0.5415	0.449744784858096\\
0.5455	0.452556112354128\\
0.5495	0.45510131319303\\
0.5535	0.457410645390735\\
0.5585	0.460006118011514\\
0.5635	0.462320455083808\\
0.5685	0.464393298216919\\
0.5745	0.466610187421615\\
0.5805	0.468580273147755\\
0.5875	0.470625398612385\\
0.5955	0.472701384955624\\
0.6045	0.474793512912894\\
0.6165	0.477331654687213\\
0.6525	0.484774245412213\\
0.6645	0.487586957503859\\
0.6895	0.493613883267387\\
0.6955	0.494744198902825\\
0.7005	0.495467044248987\\
0.7115	0.496737754973058\\
0.7225	0.498109765141163\\
0.7315	0.499477810108173\\
0.7395	0.500909331723336\\
0.7475	0.50255017708781\\
0.7555	0.504403859930767\\
0.7635	0.50647490720644\\
0.7715	0.508769506797924\\
0.7795	0.511295131329645\\
0.7875	0.51405981987526\\
0.7955	0.517071328206277\\
0.8035	0.520336159239518\\
0.8115	0.523858397645909\\
0.8195	0.527638223966695\\
0.8275	0.531669940422897\\
0.8355	0.535939285782614\\
0.8445	0.540992686333852\\
0.8565	0.548028039185404\\
0.8735	0.558023332253669\\
0.8805	0.561842835440859\\
0.8855	0.564327633029237\\
0.8895	0.566105281426863\\
0.8935	0.567639133829751\\
0.8965	0.568592340386458\\
0.8995	0.569343238339118\\
0.9025	0.569857838455241\\
0.9055	0.570098057692831\\
0.9085	0.57002135734379\\
0.9105	0.569770982512088\\
0.9125	0.56934352717697\\
0.9145	0.568722644501273\\
0.9165	0.567890893693949\\
0.9185	0.566829711174957\\
0.9205	0.565519390669113\\
0.9225	0.563939074813193\\
0.9245	0.562066761326714\\
0.9265	0.559879327313809\\
0.9285	0.557352575823661\\
0.9305	0.554461309384908\\
0.9325	0.551179435819967\\
0.9345	0.547480112200697\\
0.9365	0.543335933274065\\
0.9385	0.538719170995679\\
0.9405	0.533602071872146\\
0.9425	0.527957218525495\\
0.9445	0.521757961136289\\
0.9465	0.514978923072331\\
0.9485	0.507596582947501\\
0.9505	0.499589932480486\\
0.9535	0.486371592347124\\
0.9565	0.471667470203429\\
0.9595	0.455462385732968\\
0.9625	0.437774583489543\\
0.9655	0.418661896300577\\
0.9685	0.398226850791755\\
0.9725	0.369190315161181\\
0.9775	0.330733852758236\\
0.9905	0.228803069083972\\
0.9945	0.199878841689411\\
0.9975	0.179708391418109\\
1.0005	0.161084758561282\\
1.0035	0.144160082044132\\
1.0065	0.129021418206246\\
1.0085	0.119934807590518\\
1.0105	0.111644708309816\\
1.0125	0.104131796191732\\
1.0145	0.0973671881195437\\
1.0165	0.0913139267099718\\
1.0185	0.0859285780137697\\
1.0205	0.0811628700378257\\
1.0225	0.0769653047598218\\
1.0245	0.0732826844504446\\
1.0265	0.0700615036805945\\
1.0285	0.0672491704340166\\
1.0305	0.0647950323306552\\
1.0325	0.0626511962287348\\
1.0345	0.0607731406842511\\
1.0365	0.0591201303501552\\
1.0395	0.0569833779849187\\
1.0425	0.0551646145313738\\
1.0465	0.0530869778906946\\
1.0515	0.0508377950935812\\
1.0585	0.0479993756773998\\
1.0865	0.0368981519852301\\
1.1275	0.0202444537738002\\
1.1375	0.0165042092417824\\
1.1465	0.0133700370493655\\
1.1545	0.0108126786434579\\
1.1625	0.00850562541621747\\
1.1695	0.0067148471658518\\
1.1765	0.00515187391484151\\
1.1835	0.00382480564034271\\
1.1905	0.00273410678223618\\
1.1975	0.00187144349425727\\
1.2045	0.00121924475769886\\
1.2125	0.00069779486534971\\
1.2215	0.000337842980140834\\
1.2335	0.000107640545443655\\
1.2525	1.1390113809906e-05\\
1.3405	2.25153229393982e-13\\
2.0995	0\\
};
\end{axis}

\begin{axis}[%
width=0.411\fwidth,
height=0.23\fwidth,
at={(0.54\fwidth,-0.01\fwidth)},
scale only axis,
xmin=0,
xmax=2,
xlabel style={font=\color{white!15!black}},
xlabel={$x$},
ymin=0,
ymax=0.7,
ylabel style={font=\color{white!15!black}},
ylabel={$\rho$},
axis background/.style={fill=white},
title style={font=\bfseries},
title={Model \eqref{eq:nonlocaltransport}}
]
\addplot [color=mycolor1, line width=1.5pt, forget plot]
  table[row sep=crcr]{%
0.000500000000000167	0\\
0.8265	8.13785630424668e-05\\
0.8355	0.000285814644080506\\
0.8415	0.000582419509350895\\
0.8465	0.000979931356298813\\
0.8515	0.00154980581656705\\
0.8565	0.00231603690575666\\
0.8615	0.00328981474024115\\
0.8665	0.0044694575910853\\
0.8715	0.00584321816867028\\
0.8765	0.00739338636348297\\
0.8825	0.00945863502250033\\
0.8885	0.011717138942962\\
0.8955	0.0145606337689168\\
0.9025	0.017598918441581\\
0.9095	0.0208113323209096\\
0.9175	0.0246811694015174\\
0.9255	0.0287563961763659\\
0.9335	0.0330380022051262\\
0.9415	0.0375320908901906\\
0.9495	0.0422473517765134\\
0.9565	0.046561489744394\\
0.9635	0.0510556823584225\\
0.9705	0.0557306329567591\\
0.9785	0.0612890161409259\\
0.9865	0.0670608611440624\\
0.9955	0.0737708660899794\\
1.0075	0.0829556048807962\\
1.0235	0.0951917195244878\\
1.0305	0.100315437126126\\
1.0355	0.103790420369161\\
1.0405	0.107042664190894\\
1.0445	0.109435833044686\\
1.0485	0.1116004286895\\
1.0515	0.113050779181034\\
1.0545	0.11433838844437\\
1.0575	0.115457125399147\\
1.0605	0.116411922734942\\
1.0645	0.117471212298665\\
1.0725	0.119439548033339\\
1.0745	0.120105880741121\\
1.0765	0.120947548980293\\
1.0785	0.122033115410124\\
1.0805	0.123438863360508\\
1.0825	0.12524692611195\\
1.0835	0.126328454215817\\
1.0845	0.127542632782866\\
1.0855	0.128900070283429\\
1.0865	0.130411116143002\\
1.0875	0.132085724625488\\
1.0885	0.133933314747718\\
1.0905	0.138181592058515\\
1.0925	0.143215275550775\\
1.0945	0.149076474752071\\
1.0965	0.155786555691582\\
1.0985	0.163343621897967\\
1.1005	0.17172127752599\\
1.1025	0.180868833511997\\
1.1045	0.190712977725125\\
1.1075	0.206577652350576\\
1.1105	0.223421232077338\\
1.1225	0.292235025336756\\
1.1255	0.308003372221015\\
1.1285	0.322710584954141\\
1.1315	0.33624294352084\\
1.1335	0.344585018724039\\
1.1355	0.352380378454049\\
1.1375	0.359636077044621\\
1.1395	0.366365940595212\\
1.1415	0.372588983329849\\
1.1435	0.378327984603519\\
1.1455	0.38360825065972\\
1.1475	0.388456568814701\\
1.1495	0.39290034960238\\
1.1515	0.396966944121665\\
1.1535	0.400683118762593\\
1.1555	0.404074666970808\\
1.1575	0.407166137091523\\
1.1595	0.409980656041896\\
1.1615	0.41253983012101\\
1.1635	0.414863706308414\\
1.1655	0.41697077964895\\
1.1675	0.41887803457862\\
1.1695	0.420601010185029\\
1.1715	0.422153881339588\\
1.1735	0.423549549348227\\
1.1755	0.424799737231342\\
1.1785	0.426425248426204\\
1.1815	0.427778983438316\\
1.1845	0.428888463050241\\
1.1875	0.429777508225012\\
1.1905	0.430466796649038\\
1.1935	0.430974351810679\\
1.1965	0.431315967486746\\
1.1995	0.431505572716038\\
1.2025	0.43155554324662\\
1.2065	0.431423965469955\\
1.2105	0.43108714983629\\
1.2145	0.430565886788781\\
1.2185	0.429878932621268\\
1.2225	0.429043455737678\\
1.2275	0.427814380887101\\
1.2325	0.42640684955584\\
1.2385	0.424519713791961\\
1.2445	0.422456438522973\\
1.2515	0.419877864860632\\
1.2605	0.416373473275919\\
1.2745	0.410702847380797\\
1.2985	0.400973582270871\\
1.3115	0.39590775998076\\
1.3225	0.3917927791682\\
1.3335	0.387858439432156\\
1.3435	0.384452297277192\\
1.3535	0.381218910229469\\
1.3635	0.378166993714548\\
1.3735	0.375304604934041\\
1.3825	0.372897027134227\\
1.3915	0.370655369620002\\
1.4005	0.368586290953553\\
1.4095	0.366697630865176\\
1.4185	0.364999110256247\\
1.4265	0.363658948229137\\
1.4345	0.362490046733963\\
1.4425	0.361506471684478\\
1.4495	0.360811315422579\\
1.4565	0.360284584688011\\
1.4635	0.359941026633662\\
1.4695	0.359803948852755\\
1.4755	0.359822087023378\\
1.4815	0.360003980694366\\
1.4875	0.360356798353016\\
1.4935	0.360885960490664\\
1.4995	0.361594996899999\\
1.5055	0.362485676601093\\
1.5115	0.3635583876966\\
1.5175	0.364812689797238\\
1.5235	0.366247930131712\\
1.5295	0.367863812922519\\
1.5355	0.369660835847396\\
1.5415	0.371640545701565\\
1.5475	0.373805604168373\\
1.5535	0.376159683332178\\
1.5595	0.378707224111309\\
1.5655	0.381453089377417\\
1.5715	0.384402130892414\\
1.5775	0.387558670181889\\
1.5835	0.39092587208464\\
1.5895	0.394504968131983\\
1.5955	0.398294265342843\\
1.6015	0.402287853297384\\
1.6075	0.406473896711259\\
1.6145	0.411573552591284\\
1.6225	0.417621101170228\\
1.6405	0.431322675724136\\
1.6455	0.434853533300826\\
1.6495	0.437466585655381\\
1.6525	0.439257257929469\\
1.6555	0.440864507621399\\
1.6585	0.442248656062994\\
1.6605	0.443025607375581\\
1.6625	0.443669074969232\\
1.6645	0.444163413176206\\
1.6665	0.444491871438296\\
1.6685	0.444636587690374\\
1.6705	0.4445785973051\\
1.6725	0.444297860686234\\
1.6745	0.443773312797541\\
1.6765	0.442982938046357\\
1.6785	0.441903873980406\\
1.6805	0.440512547175759\\
1.6825	0.438784844463786\\
1.6845	0.436696322236311\\
1.6865	0.434222455953604\\
1.6885	0.431338931136505\\
1.6905	0.428021976036427\\
1.6925	0.424248734840636\\
1.6945	0.419997678694626\\
1.6965	0.415249050035801\\
1.6985	0.409985333780466\\
1.7005	0.404191746858585\\
1.7025	0.397856735539373\\
1.7045	0.390972468047105\\
1.7065	0.383535308258196\\
1.7085	0.375546254934422\\
1.7105	0.367011330120297\\
1.7125	0.357941900141046\\
1.7155	0.343374194989746\\
1.7185	0.327724696840372\\
1.7215	0.311103601642322\\
1.7245	0.293648521412915\\
1.7285	0.269361111201547\\
1.7345	0.231665576221067\\
1.7415	0.187721014348775\\
1.7455	0.16355904880204\\
1.7485	0.146221683241709\\
1.7515	0.129727053871157\\
1.7545	0.114200810775177\\
1.7575	0.0997402761268549\\
1.7605	0.0864134962781198\\
1.7625	0.0781793092479242\\
1.7645	0.0704716896865025\\
1.7665	0.0632902319485567\\
1.7685	0.0566297771665751\\
1.7705	0.0504808462713502\\
1.7725	0.044830114744868\\
1.7745	0.0396609149829614\\
1.7765	0.0349537527907606\\
1.7785	0.0306868255018715\\
1.7805	0.0268365304341818\\
1.7825	0.0233779538040233\\
1.7845	0.0202853317518348\\
1.7865	0.0175324767264331\\
1.7885	0.0150931640769354\\
1.7905	0.012941475262827\\
1.7925	0.0110520955720239\\
1.7945	0.00940056559966829\\
1.7965	0.00796348695982996\\
1.7985	0.00671868375893059\\
1.8005	0.00564532224142011\\
1.8025	0.00472399172005611\\
1.8045	0.00393675042637209\\
1.8075	0.00297170689781723\\
1.8105	0.00222230376910959\\
1.8135	0.00164629732732235\\
1.8175	0.00108737557003735\\
1.8215	0.000706143408840454\\
1.8265	0.000401905620257548\\
1.8335	0.000174535892173644\\
1.8435	4.83613455921628e-05\\
1.8645	2.28891973863554e-06\\
2.0995	0\\
};
\addplot [color=mycolor2, dashed,  line width=1.5pt, forget plot]
  table[row sep=crcr]{%
0.000500000000000167	0\\
0.7985	8.0264276837827e-05\\
0.8035	0.000232832211520151\\
0.8075	0.000517488930759313\\
0.8105	0.000912629702454648\\
0.8125	0.00131196578174908\\
0.8145	0.00186289955336383\\
0.8165	0.00261246758042732\\
0.8185	0.00361796605302578\\
0.8205	0.00494755688515092\\
0.8215	0.00575817754351249\\
0.8225	0.00668028109141572\\
0.8235	0.00772531806419918\\
0.8245	0.00890526652609758\\
0.8255	0.0102325502026561\\
0.8265	0.0117199380487727\\
0.8275	0.0133804251256473\\
0.8285	0.0152270951891569\\
0.8295	0.0172729659808386\\
0.8305	0.0195308188440477\\
0.8315	0.0220130149406343\\
0.8325	0.0247313009920989\\
0.8345	0.0309188476319608\\
0.8365	0.0381674639565874\\
0.8385	0.0465285643126254\\
0.8405	0.0560254281487818\\
0.8425	0.0666490224346283\\
0.8445	0.0783557186783468\\
0.8465	0.0910672538250648\\
0.8485	0.104673042611017\\
0.8515	0.126449610864203\\
0.8555	0.157163540382163\\
0.8625	0.211577219274984\\
0.8655	0.233776352539427\\
0.8685	0.25472312270603\\
0.8715	0.27416662897862\\
0.8735	0.286219394533414\\
0.8755	0.297520261314694\\
0.8775	0.308066547249771\\
0.8795	0.317868654914594\\
0.8815	0.326947379124207\\
0.8835	0.335331390206848\\
0.8855	0.343054976900478\\
0.8875	0.350156093684848\\
0.8895	0.356674726636818\\
0.8915	0.362651569518157\\
0.8935	0.368126986895295\\
0.8955	0.373140232360676\\
0.8975	0.377728885974878\\
0.8995	0.381928474556492\\
0.9015	0.385772240274239\\
0.9035	0.389291026227064\\
0.9055	0.392513251639012\\
0.9075	0.39546495345089\\
0.9095	0.398169875128547\\
0.9115	0.400649587225397\\
0.9135	0.402923627526738\\
0.9155	0.405009651425223\\
0.9175	0.406923585534372\\
0.9205	0.409502813105891\\
0.9235	0.411769331171802\\
0.9265	0.413759465039565\\
0.9295	0.415504252562453\\
0.9325	0.417030259672144\\
0.9355	0.418360286173798\\
0.9385	0.419513970068359\\
0.9415	0.420508301093111\\
0.9445	0.421358054831227\\
0.9485	0.42228827772004\\
0.9525	0.423010777975342\\
0.9565	0.423548222304976\\
0.9605	0.423920279067047\\
0.9645	0.424144272116679\\
0.9695	0.424239500497388\\
0.9745	0.424154623434261\\
0.9795	0.423914099262543\\
0.9855	0.423451341699078\\
0.9925	0.422715595620926\\
1.0005	0.421677714750667\\
1.0105	0.420180245915526\\
1.0255	0.417730218650799\\
1.0465	0.414316869699475\\
1.0595	0.412388608720341\\
1.0705	0.410922119503712\\
1.0835	0.409372730779101\\
1.1015	0.407242083886084\\
1.1085	0.406244819887665\\
1.1165	0.404913644514061\\
1.1305	0.402361500736573\\
1.1465	0.399503102564503\\
1.1575	0.397712491243941\\
1.1675	0.396253670525162\\
1.1765	0.395098262883957\\
1.1855	0.394106361435273\\
1.1945	0.393290102469046\\
1.2035	0.392660721854241\\
1.2115	0.392266844737752\\
1.2195	0.392036518517305\\
1.2275	0.391977188152113\\
1.2355	0.392096484337057\\
1.2435	0.392402355387191\\
1.2505	0.392829609980114\\
1.2575	0.39341202096467\\
1.2645	0.39415579047772\\
1.2715	0.395067396043673\\
1.2785	0.396153545269316\\
1.2855	0.397421097985426\\
1.2925	0.398876948149562\\
1.2995	0.400527853139068\\
1.3065	0.402380191856929\\
1.3125	0.404132555187771\\
1.3185	0.406039981411273\\
1.3245	0.408104427121826\\
1.3305	0.41032651967683\\
1.3375	0.413116446842581\\
1.3445	0.416112763081136\\
1.3515	0.419302332454991\\
1.3595	0.423154646267747\\
1.3695	0.428195638532036\\
1.3855	0.436296058137769\\
1.3915	0.439103145686127\\
1.3955	0.44081480746038\\
1.3995	0.442345385607398\\
1.4025	0.443341716713881\\
1.4055	0.444178767642693\\
1.4085	0.444826679921718\\
1.4115	0.445252190187441\\
1.4135	0.445394271547421\\
1.4155	0.445409403275323\\
1.4175	0.445285044609626\\
1.4195	0.445007932075513\\
1.4215	0.444564095461773\\
1.4235	0.443938886404398\\
1.4255	0.443117021502528\\
1.4275	0.442082641940104\\
1.4295	0.440819391580713\\
1.4315	0.439310515429931\\
1.4335	0.437538980204416\\
1.4355	0.435487618495863\\
1.4375	0.433139297657302\\
1.4395	0.430477114057827\\
1.4415	0.427484612741989\\
1.4435	0.424146031788732\\
1.4455	0.420446569795521\\
1.4475	0.416372673927596\\
1.4495	0.411912344890915\\
1.4515	0.407055454041303\\
1.4535	0.401794066673276\\
1.4555	0.396122764392085\\
1.4575	0.390038958422797\\
1.4595	0.38354318481882\\
1.4615	0.376639371870536\\
1.4645	0.365536130637453\\
1.4675	0.353574338863129\\
1.4705	0.340815654185245\\
1.4735	0.327341744868195\\
1.4775	0.308440711331414\\
1.4825	0.283720338301794\\
1.4965	0.213484418966052\\
1.5005	0.19463813215773\\
1.5035	0.181209547277736\\
1.5065	0.168493709679186\\
1.5095	0.156564936750688\\
1.5125	0.145475701839092\\
1.5155	0.135256614781986\\
1.5175	0.12893288546924\\
1.5195	0.122998583802469\\
1.5215	0.117448196791389\\
1.5235	0.112273028172251\\
1.5255	0.107461570655774\\
1.5275	0.102999899474884\\
1.5295	0.098872076294104\\
1.5315	0.0950605532627695\\
1.5335	0.0915465679729945\\
1.5355	0.0883105212576432\\
1.5375	0.0853323310696696\\
1.5395	0.0825917570605097\\
1.5415	0.0800686918647089\\
1.5435	0.0777434164498993\\
1.5465	0.0745847612307089\\
1.5495	0.071767280711541\\
1.5525	0.0692349139071293\\
1.5555	0.066937215860976\\
1.5585	0.0648296625669738\\
1.5625	0.0622495228864635\\
1.5665	0.0598637579105255\\
1.5715	0.0570660327331707\\
1.5785	0.0533497326835883\\
1.5915	0.0466811048457938\\
1.6185	0.0330250269749497\\
1.6295	0.0276446569355033\\
1.6385	0.023422869238229\\
1.6465	0.019858927650648\\
1.6535	0.0169213309856109\\
1.6605	0.0141817282996963\\
1.6665	0.0120108258982858\\
1.6725	0.0100190073429709\\
1.6785	0.00821830119915035\\
1.6845	0.00661743167543394\\
1.6905	0.00522097202054583\\
1.6965	0.0040285822474333\\
1.7025	0.00303448020354669\\
1.7085	0.00222731536172649\\
1.7145	0.00159058887062269\\
1.7215	0.0010354806196462\\
1.7285	0.000647244713285922\\
1.7375	0.000332793164112388\\
1.7485	0.000134408116902129\\
1.7655	2.71157652895759e-05\\
1.8075	1.8488387087956e-07\\
2.0995	0\\
};
\end{axis}

\begin{axis}[%
width=1.227\fwidth,
height=0.675\fwidth,
at={(-0.16\fwidth,-0.074\fwidth)},
scale only axis,
xmin=0,
xmax=1,
ymin=0,
ymax=1,
axis line style={draw=none},
ticks=none,
axis x line*=bottom,
axis y line*=left
]
\end{axis}
\end{tikzpicture}%

%% file: nonlocal_multilane.bbl
\begin{thebibliography}{10}

\bibitem{AmorimColomboTeixeira}
P.~Amorim, R.~M. Colombo, and A.~Teixeira.
\newblock On the numerical integration of scalar nonlocal conservation laws.
\newblock {\em ESAIM Math. Model. Numer. Anal.}, 49(1):19--37, 2015.

\bibitem{BayenKeimer-preprint}
A.~Bayen, A.~Keimer, L.~Pflug, and T.~Veeravalli.
\newblock Modeling multi-lane traffic with moving obstacles by nonlocal balance
  laws.
\newblock Preprint, 2020.

\bibitem{BlandinGoatin}
S.~Blandin and P.~Goatin.
\newblock Well-posedness of a conservation law with non-local flux arising in
  traffic flow modeling.
\newblock {\em Numer. Math.}, 132(2):217--241, 2016.

\bibitem{chalons2018high}
C.~Chalons, P.~Goatin, and L.~M. Villada.
\newblock High-order numerical schemes for one-dimensional nonlocal
  conservation laws.
\newblock {\em SIAM J. Sci. Comput.}, 40(1):A288--A305, 2018.

\bibitem{friedrich2020micromacro}
F.~A. Chiarello, J.~Friedrich, P.~Goatin, and S.~G\"{o}ttlich.
\newblock Micro-macro limit of a nonlocal generalized {A}w-{R}ascle type model.
\newblock {\em SIAM J. Appl. Math.}, 80(4):1841--1861, 2020.

\bibitem{chiarelloFriedrichGoatinGK}
F.~A. Chiarello, J.~Friedrich, P.~Goatin, S.~G{\"o}ttlich, and O.~Kolb.
\newblock A non-local traffic flow model for 1-to-1 junctions.
\newblock {\em European J. Appl. Math.}, 31(6):1029–1049, 2020.

\bibitem{ChiarelloGoatin}
F.~A. Chiarello and P.~Goatin.
\newblock Global entropy weak solutions for general non-local traffic flow
  models with anisotropic kernel.
\newblock {\em ESAIM Math. Model. Numer. Anal.}, 52(1):163--180, 2018.

\bibitem{chiarello2019multiclass}
F.~A. Chiarello and P.~Goatin.
\newblock Non-local multi-class traffic flow models.
\newblock {\em Netw. Heterog. Media}, 14(2):371--387, 2019.

\bibitem{NARWA2019}
F.~A. Chiarello, P.~Goatin, and E.~Rossi.
\newblock Stability estimates for non-local scalar conservation laws.
\newblock {\em Nonlinear Anal. Real World Appl.}, 45:668--687, 2019.

\bibitem{shen2019stationary}
J.~Chien and W.~Shen.
\newblock Stationary wave profiles for nonlocal particle models of traffic flow
  on rough roads.
\newblock {\em NoDEA Nonlinear Differential Equations Appl.}, 26(6):Paper No.
  53, 2019.

\bibitem{colombo2007non}
R.~M. Colombo, A.~Corli, and M.~D. Rosini.
\newblock Non local balance laws in traffic models and crystal growth.
\newblock {\em ZAMM Z. Angew. Math. Mech.}, 87(6):449--461, 2007.

\bibitem{crandall1980monotone}
M.~G. Crandall and A.~Majda.
\newblock Monotone difference approximations for scalar conservation laws.
\newblock {\em Math. Comp.}, 34(149):1--21, 1980.

\bibitem{friedrich2019maximum}
J.~Friedrich and O.~Kolb.
\newblock Maximum principle satisfying {CWENO} schemes for nonlocal
  conservation laws.
\newblock {\em SIAM J. Sci. Comput.}, 41(2):A973--A988, 2019.

\bibitem{friedrich2018godunov}
J.~Friedrich, O.~Kolb, and S.~G\"{o}ttlich.
\newblock A {G}odunov type scheme for a class of {LWR} traffic flow models with
  non-local flux.
\newblock {\em Netw. Heterog. Media}, 13(4):531--547, 2018.

\bibitem{GoatinRossi}
P.~Goatin and E.~Rossi.
\newblock A multilane macroscopic traffic flow model for simple networks.
\newblock {\em SIAM J. Appl. Math.}, 79(5):1967--1989, 2019.

\bibitem{GoatinRossi2017}
P.~Goatin and F.~Rossi.
\newblock A traffic flow model with non-smooth metric interaction:
  well-posedness and micro-macro limit.
\newblock {\em Commun. Math. Sci.}, 15(1):261--287, 2017.

\bibitem{GoatinScialanga}
P.~Goatin and S.~Scialanga.
\newblock Well-posedness and finite volume approximations of the {LWR} traffic
  flow model with non-local velocity.
\newblock {\em Netw. Heterog. Media}, 11(1):107--121, 2016.

\bibitem{HoldenRisebroBook2015}
H.~Holden and N.~H. Risebro.
\newblock {\em Front tracking for hyperbolic conservation laws}, volume 152 of
  {\em Applied Mathematical Sciences}.
\newblock Springer, Heidelberg, second edition, 2015.

\bibitem{HoldenRisebro}
H.~Holden and N.~H. Risebro.
\newblock Models for dense multilane vehicular traffic.
\newblock {\em SIAM J. Math. Anal.}, 51(5):3694--3713, 2019.

\bibitem{KeimerPflug2017}
A.~Keimer and L.~Pflug.
\newblock Existence, uniqueness and regularity results on nonlocal balance
  laws.
\newblock {\em J. Differential Equations}, 263(7):4023--4069, 2017.

\bibitem{keimer2019nonlocal}
A.~Keimer and L.~Pflug.
\newblock Nonlocal conservation laws with time delay.
\newblock {\em NoDEA Nonlinear Differential Equations Appl.}, 26(6):Paper No.
  54, 34, 2019.

\bibitem{keimer2019local}
A.~Keimer and L.~Pflug.
\newblock On approximation of local conservation laws by nonlocal conservation
  laws.
\newblock {\em J. Math. Anal. Appl.}, 475(2):1927--1955, 2019.

\bibitem{Kruzkov}
S.~N. Kru{\v{z}}kov.
\newblock First order quasilinear equations with several independent variables.
\newblock {\em Mat. Sb. (N.S.)}, 81 (123):228--255, 1970.

\bibitem{LeVeque}
R.~J. LeVeque.
\newblock {\em Numerical methods for conservation laws}.
\newblock Lectures in Mathematics ETH Z\"{u}rich. Birkh\"{a}user Verlag, Basel,
  second edition, 1992.

\bibitem{LighthillWhitham}
M.~J. Lighthill and G.~B. Whitham.
\newblock On kinematic waves. {I}{I}. {A} theory of traffic flow on long
  crowded roads.
\newblock {\em Proc. Roy. Soc. London. Ser. A.}, 229:317--345, 1955.

\bibitem{Richards}
P.~I. Richards.
\newblock Shock waves on the highway.
\newblock {\em Oper. Res.}, 4:42--51, 1956.

\bibitem{RidderShen}
J.~Ridder and W.~Shen.
\newblock Traveling waves for nonlocal models of traffic flow.
\newblock {\em Discrete Contin. Dyn. Syst.}, 39(7):4001--4040, 2019.

\end{thebibliography}
